\documentclass[reqno]{amsart}
\usepackage{mathrsfs}
\usepackage{color}
\usepackage{amsmath}
\usepackage{amsfonts}
\usepackage{amssymb}
\usepackage{graphicx}
\usepackage{hyperref}

 \newtheorem{Theorem}{Theorem}[section]
 \newtheorem{Corollary}[Theorem]{Corollary}
 \newtheorem{Lemma}[Theorem]{Lemma}
 \newtheorem{Proposition}[Theorem]{Proposition}

 \newtheorem{Definition}[Theorem]{Definition}

 \newtheorem{Remark}[Theorem]{Remark}

 \numberwithin{equation}{section}

\begin{document}

\title[A solution of an $L^{2}$ extension problem with optimal estimate]
 {A solution of an $L^{2}$ extension problem with \\
 optimal estimate and applications}

\author{Qi'an Guan}
\address{Qi'an Guan: Beijing
International Center for Mathematical Research, and School of
Mathematical Sciences, Peking University, Beijing 100871, China.}
\email{guanqian@amss.ac.cn}
\author{Xiangyu Zhou}
\address{Xiangyu Zhou: Institute of Mathematics, AMSS, and Hua Loo-Keng Key Laboratory of Mathematics,
Chinese Academy of Sciences, Beijing 100190, China.}
\email{xyzhou@math.ac.cn}

\thanks{The authors were partially supported by NSFC}



\date{}

\dedicatory{}

\commby{}


\begin{abstract}
In this paper, we prove an $L^2$ extension theorem with optimal
estimate in a precise way, which implies optimal estimate versions
of various well-known $L^2$ extension theorems. As applications, we
give proofs of a conjecture of Suita on the equality condition in
Suita's conjecture, the so-called $L-$conjecture, and the extended
Suita conjecture. As other applications, we give affirmative answer
to a question by Ohsawa about limiting case for the extension
operators between the weighted Bergman spaces, and we present a
relation of our result to Berndtsson's important result on
log-plurisubharmonicity of Bergman kernel.
\end{abstract}

\maketitle

\section{background and notations}

The $L^2$ extension problem is stated as follows (for the
background, see Demailly \cite{demailly2010}): for a suitable pair
$(M, S)$, where $S$ is a closed complex subvariety of a complex
manifold $M$, given a holomorphic function $f$ (or a holomorphic
section of a holomorphic vector bundle)) on $Y$ satisfying suitable
$L^2$ conditions on $S$, find an $L^2$ holomorphic extension $F$ on
$M$ together with a good or even optimal $L^2$ estimate for $F$ on
$M$.

The famous Ohsawa-Takegoshi $L^2$ extension theorem (Ohsawa wrote a
series of papers on $L^2$ extension theorem in more general
settings) gives an answer to the first part of the problem
---existence of $L^2$ extension. There have been some new proofs and a lot of important
applications of the theorem in complex geometry and several complex
variables, thanks to the works of Y.-T. Siu, J.P. Demailly, Ohsawa,
and Berndtsson et al. An unsolved problem is left
--- the second part of optimal estimate in the $L^2$ extension
problem. A first exception is Blocki's recent work on optimal
estimate of Ohsawa-Takegoshi's $L^2$ extension theorem for bounded
pseudoconvex domains (see \cite{blocki12a}) as a continuation of an
earlier work towards the $L^{2}$ extension problem with optimal
estimate \cite{guan-zhou-zhu10} (see also \cite{blocki12}). Another
exception is our recent work \cite{guan-zhou12} based on
\cite{guan-zhou-zhu10} about optimal estimate of Ohsawa's $L^2$
extension theorem with negligible weight for Stein manifolds
\cite{{ohsawa3}}.

In the present paper, we shall further discuss the $L^2$ extension
problem with optimal estimate and give a solution of the problem
with its applications, by putting it into a vision with a wider
scope.

The paper is organized as follows. In the rest of this section, we
recall some notations used in the paper. In section 2, we present
our main theorems, solving the $L^2$ extension problem with optimal
estimate. In section 3, we introduce the main applications and
corollaries of our main theorems, among others, including: we give
proofs of a conjecture of Suita on the equality condition in Suita's
conjecture, the so-called $L-$conjecture, and the extended Suita
conjecture; we find a relation of our result to Berndtsson's theorem
on log-plurisubharmonicity of Bergman kernel; we give an affirmative
answer to a question by Ohsawa in \cite{ohsawa6} about a limiting
case for the extension operators between the weighted Bergman
spaces; and we also obtain optimal estimate versions of various
well-known $L^2$ extension theorems. In Section 4 we recall or prove
some preliminary results used in the proofs of the main theorems and
corollaries. In Section 5, we give the detailed proofs of the main
theorems. In
Section 6, we give the proofs of the main corollaries.\\

Now let's recall some notations in \cite{ohsawa5}. Let $M$ be a
complex $n-$dimensional manifold, and $S$ be a closed complex
subvariety of $M$. Let $dV_{M}$ be a continuous volume form on $M$.
We consider a class of upper-semi-continuous function $\Psi$ from
$M$ to the interval $[-\infty,A)$, where $A\in(-\infty,+\infty]$,
such that

$1)$ $\Psi^{-1}(-\infty)\supset S$, and $\Psi^{-1}(-\infty)$ is a closed subset of $M$;

$2)$ If $S$ is $l-$dimensional around a point $x\in S_{reg}$
($S_{reg}$ is the regular part of $S$), there exists a local
coordinate $(z_{1},\cdots,z_{n})$ on a neighborhood $U$ of $x$ such
that $z_{l+1}=\cdots=z_{n}=0$ on $S\cap U$ and
$$\sup_{U\setminus S}|\Psi(z)-(n-l)\log\sum_{l+1}^{n}|z_{j}|^{2}|<\infty.$$

The set of such polar functions $\Psi$ will be denoted by
$\#_{A}(S)$.

For each $\Psi\in\#_{A}(S)$, one can associate a positive measure
$dV_{M}[\Psi]$ on $S_{reg}$ as the minimum element of the partially
ordered set of positive measures $d\mu$ satisfying

$$\int_{S_{l}}fd\mu\geq\limsup_{t\to\infty}\frac{2(n-l)}
{\sigma_{2n-2l-1}}\int_{M}fe^{-\Psi}\mathbb{I}_{\{-1-t<\Psi<-t\}}dV_{M}$$
for any nonnegative continuous function $f$ with $supp
f\subset\subset M$, where $\mathbb{I}_{\{-1-t<\Psi<-t\}}$ is the
characteristic function of the set $\{-1-t<\Psi<-t\}$. Here denote
by $S_{l}$ the $l$-dimensional component of $S_{reg}$, denote by
$\sigma_{m}$ the volume of the unit sphere in $\mathbb{R}^{m+1}$.

Let $\omega$ be a K\"{a}hler metric on $M\setminus (X\cup S)$, where
$X$ is a closed subset of $M$ such that $S_{sing}\subset X$
($S_{sing}$ is the singular part of $S$).

We can define measure $dV_{\omega}[\Psi]$ on $S\setminus X$ as the
minimum element of the partially ordered set of positive measures
$d\mu'$ satisfying
$$\int_{S_{l}}fd\mu'\geq\limsup_{t\to\infty}\frac{2(n-l)}
{\sigma_{2n-2l-1}}\int_{M\setminus (X\cup S)}fe^{-\Psi}\mathbb{I}_{\{-1-t<\Psi<-t\}}dV_{\omega}$$
for any nonnegative continuous function $f$ with
$supp (f)\subset\subset M\setminus X$
(As $$Supp(\mathbb{I}_{\{-1-t<\Psi<-t\}})\cap Supp(f)\subset\subset M\setminus (X\cup S),$$
right hand side of the above inequality is well-defined).

Let $u$ be a continuous section of $K_{M}\otimes E$, where $E$ is a holomorphic vector bundle
equipped with a continuous metric $h$ on $M$.

We define
$$|u|^{2}_{h}|_{V}:=\frac{c_{n}h(e,e)v\wedge\bar{v}}{dV_{M}},$$
and
$$|u|^{2}_{h,\omega}|_{V}:=\frac{c_{n}h(e,e)v\wedge\bar{v}}{dV_{\omega}},$$
where $u|_{V}=v\otimes e$ for an open set $V\subset M\setminus (X\cup S)$, $v$ is a continuous section of
$K_{M}|_{V}$ and $e$ is a continuous
section of $E|_{V}$ (especially, we define
$$|u|^{2}|_{V}:=\frac{c_{n}u\wedge\bar{u}}{dV_{M}},$$
when $u$ is a continuous section of $K_{M}$). It is clear that
$|u|^{2}_{h}$ is independent of the choice of $V$.

The following argument shows a relationship between
$dV_{\omega}[\Psi]$ and $dV_{M}[\Psi]$ (resp. $dV_{\omega}$ and
$dV_{M}$), precisely
\begin{equation}
\label{equ:9.1}
\int_{M\setminus(X\cup S)}f|u|^{2}_{h,\omega}dV_{\omega}[\Psi]
=\int_{M\setminus(X\cup S)}f|u|^{2}_{h}dV_{M}[\Psi],
\end{equation}

\begin{equation}
\label{equ:9.2}
(resp.
\int_{M\setminus(X\cup S)}f|u|^{2}_{h,\omega}dV_{\omega}
=\int_{M\setminus(X\cup S)}f|u|^{2}_{h}dV_{M})
\end{equation}
where $f$ is a continuous function with compact support on $M\setminus X$.

For the neighborhood $U$, let $u|_{U}=v\otimes e$. Note that
\begin{equation}
\label{equ:9.3}
\begin{split}
&\int_{M\setminus(X\cup S)}f\mathbb{I}_{\{-1-t<\Psi<-t\}}|u|^{2}_{h,\omega}e^{-\Psi}dV_{\omega}
\\=&\int_{M\setminus(X\cup S)}f\mathbb{I}_{\{-1-t<\Psi<-t\}}h(e,e)c_{n}v\wedge\bar{v}e^{-\Psi}
\\=&\int_{M\setminus(X\cup S)}f\mathbb{I}_{\{-1-t<\Psi<-t\}}|u|^{2}_{h}e^{-\Psi}dV_{M},
\end{split}
\end{equation}
and
\begin{equation}
\begin{split}
&(resp. \int_{M\setminus (X\cup S)}f |u|^{2}_{h,\omega}e^{-\Psi}dV_{\omega}
\\=&\int_{M\setminus (X\cup S)}f h(e,e)c_{n}v\wedge\bar{v}e^{-\Psi}
\\=&\int_{M\setminus (X\cup S)}f |u|^{2}_{h}e^{-\Psi}dV_{M},)
\end{split}
\end{equation}
where $f$ is a continuous function with compact support on $M\setminus(X\cup S)$. As
$$Supp(\mathbb{I}_{\{-1-t<\Psi<-t\}})\cap Supp(f)\subset\subset M\setminus (X\cup S),$$
equality \ref{equ:9.3} is well defined.
Then we have equality \ref{equ:9.1} and \ref{equ:9.2}.

It is clear that $|u|^{2}_{h}$ is independent of the choice of $U$,
while $|u|^{2}_{h}dV_{M}$ is independent of the choice of $dV_{M}$
(resp. $|u|^{2}_{h}dV_{M}[\Psi]$ is independent of the choice of
$dV_{M}$). Then the space of $L^{2}$ integrable holomorphic sections
of $K_{M}$ is denoted by $A^{2}(M,K_{M}\otimes
E,dV_{M}^{-1},dV_{M})$ (resp. the space of holomorphic sections of
$K_{M}|_{S}$ which is $L^{2}$ integrable with respect to the measure
$dV_{M}[\Psi]$ is denoted by $A^{2}(S,K_{M}|_{S}\otimes
E|_{S},dV_{M}^{-1},$ $dV_{M}[\Psi])$).

Denote by
$$|u|^{2}_{h}dV_{M}:=\{u,u\}_{h},$$
for any continuous section $u$ of $K_{M}\otimes E$.
Define
$$\{f,f\}_{h}:=\langle e,e\rangle_{h}\sqrt{-1}^{dimS^{2}}f_{1}\wedge\bar{f}_{1}$$
for any continuous section $f$ of $K_{S}\otimes E|_{S}$, where $f=f_{1}\otimes e$ locally (see \cite{demailly2010}).
It is clear that $\{f,f\}_{h}$ is well defined.

\begin{Definition}
Let $M$ be a complex manifold with a continuous volume form
$dV_{M}$, and $S$ be a closed complex subvariety of $M$. We call
$(M,S)$ satisfies condition $(ab)$ if $M$ and $S$ satisfy the
following conditions:

There exists a closed subset $X\subset M$ such that:

$(a)$ $X$ is locally negligible with respect to $L^2$ holomorphic
functions, i.e., for any local coordinate neighborhood $U\subset M$
and for any $L^2$ holomorphic function $f$ on $U\setminus X$, there
exists an $L^2$ holomorphic function $\tilde{f}$ on $U$ such that
$\tilde{f}|_{U\setminus X}=f$ with the same $L^{2}$ norm.

$(b)$ $M\setminus X$ is a Stein manifold which intersects with every component of $S$,
such that $S_{sing}\subset X$.
\end{Definition}

When $S$ is smooth, the condition $(ab)$ is the same as condition 1)
in Theorem 4 in \cite{ohsawa5}. There are the following examples
satisfying condition $(ab)$:

$1).$ $M$ is a Stein manifold (including open Riemann surfaces),
and $S$ is any closed complex subvariety of $M$;

$2).$ $M$ is a complex projective algebraic manifold (including
compact Riemann surfaces), and $S$ is any closed complex subvariety
of $M$;

$3).$ $M$ is a projective family (see \cite{siu00}),
and $S$ is any closed complex subvariety of $M$.

The Hermitian metric $h$ on $E$ is said to be semi-positive in the
sense of Nakano if the curvature tensor $\Theta_{h}$ is
semi-positive definite as a hermitian form on $T_{X}\otimes E$, i.e.
if for every $u \in T_{X}\otimes E$, we have
$\sqrt{-1}\Theta_{h}(u,u)\geq 0$ (see \cite{demailly-book}).

Let $\Delta_{A,h,\delta}(S)$ be the subset of functions $\Psi$ in
$\#_{A}(S)$ which satisfies that both $he^{-\Psi}$ and
$he^{-(1+\delta)\Psi}$ are semi-positive in the sense of Nakano on
$M\setminus (X\cup S)$.

Let $\Delta_{A}(S)$ be the subset of plurisubharmonic functions
$\Psi$ in $\#_{A}(S)$.

\section{main theorems}

In the present section, we state an $L^{2}$ extension theorem with
optimal estimate, related to a kind of positive real function
$c_{A}(t)$ which will be explained later on, solving the $L^2$
extension problem with optimal estimate. The theorem is stated first
in a general setting and then in a less general but sufficiently
useful setting. BTW, the word "optimal" depends on the considered
setting. If the setting becomes narrower, the estimate possibly
couldn't be optimal again.

Given $\delta>0$, let $c_{A}(t)$ be a positive function on
$(-A,+\infty)$ $(A\in(-\infty,+\infty))$, which is in
$C^{\infty}((-A,+\infty))$ and satisfies both
$\int_{-A}^{\infty}c_{A}(t)e^{-t}dt<\infty$ and
\begin{equation}
\label{equ:c_A_delta}
\begin{split}
&(\frac{1}{\delta}c_{A}(-A)e^{A}+\int_{-A}^{t}c_{A}(t_{1})e^{-t_{1}}dt_{1})^{2}>\\&c_{A}(t)e^{-t}
(\int_{-A}^{t}(\frac{1}{\delta}c_{A}(-A)e^{A}+\int_{-A}^{t_{2}}c_{A}(t_{1})e^{-t_{1}}dt_{1})
dt_{2}+\frac{1}{\delta^{2}}c_{A}(-A)e^{A}),
\end{split}
\end{equation}
for any $t\in(-A,+\infty)$.

If $c_{A}(t)e^{-t}$ is decreasing with respect to $t$, then
inequality \ref{equ:c_A_delta} holds.

We establish the following $L^{2}$ extension theorem with an optimal
estimate as follows:

\begin{Theorem}\label{t:guan-zhou-semicontinu2}(main theorem 1)
Let $(M,S)$ satisfy condition $(ab)$,
$h$ be a smooth metric on a holomorphic vector bundle $E$ on $M$ with rank $r$.
Let $\Psi\in \#_{A}(S)\cap C^{\infty}(M\setminus S)$, which satisfies

1), $he^{-\Psi}$ is semi-positive in the sense of Nakano on
$M\setminus (S\cup X)$ ($X$ is as in the definition of condition
$(a,b)$),

2), there exists a continuous function $a(t)$ on $(-A,+\infty]$,
such that $0<a(t)\leq s(t)$ and
$a(-\Psi)\sqrt{-1}\Theta_{he^{-\Psi}}+\sqrt{-1}\partial\bar\partial\Psi$
is semi-positive in the sense of Nakano on $M\setminus (S\cup X)$,
where $$s(t)=\frac{\int_{-A}^{t}(\frac{1}
{\delta}c_{A}(-A)e^{A}+\int_{-A}^{t_{2}}c_{A}(t_{1})e^{-t_{1}}dt_{1})dt_{2}+\frac{1}{\delta^{2}}c_{A}(-A)e^{A}}
{\frac{1}{\delta}c_{A}(-A)e^{A}+\int_{-A}^{t}c_{A}(t_{1})e^{-t_{1}}dt_{1}}.$$
Then there exists a uniform constant $\mathbf{C}=1$, which is
optimal, such that, for any holomorphic section $f$ of $K_{M}\otimes
E|_{S}$ on $S$ satisfying
\begin{equation}
\label{equ:condition}
\sum_{k=1}^{n}\frac{\pi^{k}}{k!}\int_{S_{n-k}}|f|^{2}_{h}dV_{M}[\Psi]<\infty,
\end{equation}
there
exists a holomorphic section $F$ of $K_{M}\otimes E$ on $M$ satisfying $F = f$ on $ S$ and
\begin{equation}
\label{equ:optimal_delta}
\int_{M}c_{A}(-\Psi)|F|^{2}_{h}dV_{M}
\leq\mathbf{C}(\frac{1}{\delta}c_{A}(-A)e^{A}+\int_{-A}^{\infty}c_{A}(t)e^{-t}dt)
\sum_{k=1}^{n}\frac{\pi^{k}}{k!}\int_{S_{n-k}}|f|^{2}_{h}dV_{M}[\Psi],
\end{equation}
where $c_{A}(t)$ satisfies $c_{A}(-A)e^{A}:=\lim_{t\to -A^{+}}c_{A}(t)e^{-t}<\infty$ and $c_{A}(-A)e^{A}\neq0$.
\end{Theorem}

Using Remark \ref{r:c_A_continu} and Lemma \ref{l:c_A} which will be
discussed later on, we can replace smoothness of $c_{A}$ in the
above theorem by continuity.

Now we consider a useful and simpler class of functions as follows:

Let $c_{A}(t)$ be a positive function in $C^{\infty}((-A,+\infty))$
$(A\in(-\infty,+\infty])$, satisfying
$\int_{-A}^{\infty}c_{A}(t)e^{-t}dt<\infty$ and
\begin{equation}
\label{equ:c_A}
(\int_{-A}^{t}c_{A}(t_{1})e^{-t_{1}}dt_{1})^{2}>c_{A}(t)e^{-t}
\int_{-A}^{t}\int_{-A}^{t_{2}}c_{A}(t_{1})e^{-t_{1}}dt_{1}dt_{2},
\end{equation}
for any $t\in(-A,+\infty)$.

When $c_{A}(t)e^{-t}$ is decreasing with respect to $t$ and $A$ is finite, inequality \ref{equ:c_A} holds.

For such a simpler and sufficiently useful class of functions, we
establish the following $L^2$ extension theorem with an optimal
estimate, whose simpler version was announced in
\cite{guan-zhou13a}:

\begin{Theorem}\label{t:guan-zhou-unify}(main theorem 2)
Let $(M,S)$ satisfy condition $(ab)$, and $\Psi$ be a
plurisubharmonic function in $\Delta_{A}(S)\cap
C^{\infty}(M\setminus (S\cup X))$ ($X$ is as in the definition of
condition $(a,b)$), Let $h$ be a smooth metric on a holomorphic
vector bundle $E$ on $M$ with rank $r$, such that $he^{-\Psi}$ is
semi-positive in the sense of Nakano on $M\setminus (S\cup X)$,
(when $E$ is a line bundle, $h$ can be chosen as a semipositive
singular metric). Then there exists a uniform constant
$\mathbf{C}=1$, which is optimal, such that, for any holomorphic
section $f$ of $K_{M}\otimes E|_{S}$ on $S$ satisfying condition
\ref{equ:condition}  there exists a holomorphic section $F$ of
$K_{M}\otimes E$ on $M$ satisfying $F = f$ on $ S$ and
\begin{eqnarray*}
\int_{M}c_{A}(-\Psi)|F|^{2}_{h}dV_{M}
\leq\mathbf{C}\int_{-A}^{\infty}c_{A}(t)e^{-t}dt\sum_{k=1}^{n}\frac{\pi^{k}}{k!}\int_{S_{n-k}}|f|^{2}_{h}dV_{M}[\Psi].
\end{eqnarray*}
\end{Theorem}

Similarly as before, we can replace smoothness of $c_{A}$ in the
above theorem by continuity.

\section{Applications and main corollaries}

In this section, we present applications and main corollaries of our
main theorems, among others, solutions of a conjecture of Suita on
the equality condition in Suita's conjecture, $L-$ conjecture, the
extended Suita conjecture; a relation to Berndtsson's log-
plurisubharmonicity of Bergman kernel; optimal constant versions of
various known $L^2$ extension theorems; an affirmative answer to a
question by Ohsawa about a limiting case for the extension operators
between the weighted Bergman spaces; and so on.

\subsection{A conjecture of Suita}
$\\$

In this subsection, we present a corollary of Theorem
\ref{t:guan-zhou-unify}, which solves a conjecture of Suita on the
equality condition in Suita's conjecture on the comparison between
the Bergman kernel and the logarithmic capacity.

Let $\Omega$ be an open Riemann surface, which admits a nontrivial
Green function $G_{\Omega}$. Let $w$ be a local coordinate on a
neighborhood $V_{z_0}$ of $z_{0}\in\Omega$ satisfying $w(z_0)=0$.
Let $\kappa_{\Omega}$ be the Bergman kernel for holomorphic $(1,0)$
forms on $\Omega$. We define
$$B_{\Omega}(z)|dw|^{2}:=\kappa_{\Omega}(z)|_{V_{z_0}},$$
and
$$B_{\Omega}(z,\bar{t})dw\otimes d\bar{t}:=\kappa_{\Omega}(z,\bar{t})|_{V_{z_0}}.$$
Let $c_{\beta}(z)$ be the logarithmic capacity which is locally defined by
$$c_{\beta}(z_{0}):=\exp\lim_{\xi\rightarrow z}(G_{\Omega}(z,z_{0})-\log|w(z)|)$$
on $\Omega$ (see \cite{sario}).

Suita's conjecture in \cite{suita} says that on any open Riemann
surface $\Omega$ as above, $(c_{\beta}(z_{0}))^{2}\leq\pi
B_{\Omega}(z_{0})$.

The above conjecture was first proved for bounded planar domains by
Blocki \cite{blocki12,blocki12a}, and then by Guan-Zhou
\cite{guan-zhou12} for open Riemann surfaces. For earlier works, see
 \cite{guan-zhou-zhu10}.

In the same paper \cite{suita}, Suita also conjectured a necessary
and sufficient condition for the equality holding in his inequality:

$\\$ \emph{\textbf{A conjecture of Suita}:
$(c_{\beta}(z_{0}))^{2}=\pi B_{\Omega}(z_{0})$, for $z_{0}\in\Omega$
if and only if $\Omega$ is conformally equivalent to the unit disc
less a (possible) closed set of inner capacity zero. } $\\$

In fact, a closed set of inner capacity zero is a polar set (locally singularity set of a subharmonic function).

Using Theorem \ref{t:guan-zhou-unify}, we solve the conjecture of
Suita:

\begin{Theorem}
\label{c:suita_equ}The above conjecture of Suita holds.
\end{Theorem}

\subsection{$L-$conjecture}
$\\$

In this subsection, we give a proof of $L$-conjecture.

Let $\Omega$ be an open Riemann surface which admits a nontrivial
Green function $G_{\Omega}$ and is not biholomorphic to the unit
disc less a (possible) closed set of inner capacity zero.

Assume that $G_{\Omega}(\cdot,t)$ is an exhaustion function for any
$t\in\Omega$. Associated to the Bergman kernel
$\kappa_{\Omega}(z,\bar{t})$, one may define the adjoint $L$-kernel
$L_{\Omega}(z,t):=\frac{2}{\pi}\frac{\partial^{2}G_{\Omega}(z,t)}{\partial
z \partial t}$ (see \cite{schiffer}). In \cite{yamada98}, there is a
conjecture on the zero points of the adjoint $L$-kernel as follows:

$\\$
\emph{\textbf{$L$-Conjecture (LC)}:
For any $t\in\Omega$, $\exists z\in\Omega$, we have $L_{\Omega}(z,t)=0$.}
$\\$

It is known that, for finite Riemann surface $\Omega$,
$G_{\Omega}(\cdot,t)$ is an exhaustion function for any $t\in\Omega$
(see \cite{yamada98}).

By Theorem 6 in \cite{yamada98}, $L$-conjecture for finite Riemann
surfaces is deduced from the above conjecture of Suita.

Using Theorem \ref{c:suita_equ}, we solve the $L-$conjecture for any
open Riemann surface with exhaustion Green function:

\begin{Theorem}
\label{c:L_conj_proof}The above $L-$conjecture holds.
\end{Theorem}

The following example shows that the assumption that
$G_{\Omega}(\cdot,t)$ is an exhaustion function for any $t\in\Omega$
is necessary.

Let $m$ and $p$ denote the numbers of the boundary contours and the
genus of $\Omega$, respectively (see \cite{suita76}). In fact, for
any finite Riemann surface $\Omega$, which is not simply connected,
the Bergman kernel $\kappa_{\Omega}(z,\bar{t})$ of $\Omega$ has
exactly $2p+m-1$ zeros for suitable $t$ (see \cite{suita76}).

Let $\Omega$ be an annulus, then we have $2p+m-1=1$ (see page 93,
\cite{schiffer}). It is known that
$\#\{z|L_{\Omega}(z,t)=0\}+\#\{z|\kappa_{\Omega}(z,\bar{t})=0\}\leq4p+2m-2=2$
for all $t\in\Omega$, (see \cite{suita76}). Note that
$\kappa_{\Omega}(z,\bar{t})$ has exactly $2p+m-1=1$ zeros for
suitable $t$. Using Theorem \ref{c:L_conj_proof}, we have
$\#\{z|L_{\Omega}(z,t)=0\}=1=4p+2m-2-1$ for suitable $t\in\Omega$.
Let $t_{1}\in\Omega$ satisfy $\#\{z|L_{\Omega}(z,t_{1})=0\}=1$.
Assume that $z_{1}\in\{z|L_{\Omega}(z,t_{1})=0\}$. Note that
$z_{1}\neq t_{1}$. As
$G_{\Omega\setminus\{z_{1}\}}=G_{\Omega}|_{\Omega\setminus\{z_{1}\}}$,
then we have $\#\{z|L_{\Omega\setminus\{z_{1}\}}(z,t_{1})=0\}=0$.

\subsection{Extended Suita Conjecture}
$\\$

Let $\Omega$ be an open Riemann surface, which admits a nontrivial
Green function $G_{\Omega}$. Take $z_{0}\in\Omega$ with a local
coordinate $z$. Let $p:\Delta\to\Omega$ be the universal covering
from unit disc $\Delta$ to $\Omega$.

We call the holomorphic function $f$ (resp. holomorphic $(1,0)$ form
$F$) on $\Delta$ is a multiplicative function (resp. multiplicative
differential (Prym differential)) if there is a character $\chi$,
which is the representation of the fundamental group of $\Omega$,
such that $g^{*}f=\chi(g)f$ (resp. $g^{*}F=\chi(g)F$), where
$|\chi|=1$ and $g$ is an element of the fundamental group of
$\Omega$ which naturally acts on the universal covering of $\Omega$
(see \cite{farkas}). Denote the set of such kinds of $f$ (resp. $F$)
by $\mathcal{O}^{\chi}(\Omega))$ (resp. $\Gamma^{\chi}(\Omega)$).

As $p$ is a universal covering, then for any harmonic function
$h_{\Omega}$ on $\Omega$, there exists a $\chi_{h}$ and a
multiplicative function $f_{h}\in\mathcal{O}^{\chi_{h}}(\Omega))$,
such that $|f_{h}|=p^{*}e^{h_{\Omega}}$.

For Green function $G_{\Omega}(\cdot,z_{0})$, one can also find a
$\chi_{z_{0}}$ and a multiplicative function
$f_{z_{0}}\in\mathcal{O}^{\chi_{z_{0}}}(\Omega)$, such that
$|f_{z_{0}}|=p^{*}e^{G_{\Omega}(\cdot,z_{0})}$.

Because $g^{*}|f|=|g^{*}f|=|\chi(g)f|=|f|$ and
$g^{*}(F\wedge\bar{F})=g^{*}F\wedge
\overline{g^{*}F}=\chi(g)F\wedge\overline{\chi(g)F}=F\wedge\bar{F}$,
it follows that $|f|$ and $F\wedge\bar{F}$ are fibre constant
respect to $p$.

As $F\wedge\bar{F}$ is fibre constant, one can define multiplicative
Bergman kernel $\kappa^{\chi}(x,\bar{y})$ for
$\Gamma^{\chi}(\Omega)$ on $\Omega\times\Omega$. Let
$B_{\Omega}^{\chi}(z)|dz|^{2}:=\kappa_{\Omega}^{\chi}(z,\bar{z})$.
The extended Suita conjecture is formulated as follows
(\cite{yamada98}):

$\\$
\emph{\textbf{Extended Suita Conjecture:} $c_{\beta}^{2}(z_{0})\leq \pi B_{\Omega}^{\chi}(z_{0})$,
and equality holds if and only if $\chi=\chi_{z_{0}}$.}
$\\$

The weighted Bergman kernel $\kappa_{\Omega,\rho}$ with weight
$\rho$ of holomorphic $(1,0)$ form on a Riemann surface $\Omega$ is
defined by $\kappa_{\Omega,\rho}:=\sum_{i} e_{i}\otimes
\bar{e}_{i}$, where $\{e_{i}\}_{i=1,2,\cdots}$ are holomorphic
$(1,0)$ forms on $\Omega$ and satisfy
$\sqrt{-1}\int_{\Omega}\rho\frac{e_{i}}{\sqrt{2}}\wedge\frac{\bar{e}_{j}}{\sqrt{2}}
=\delta_{i}^{j}$.

Let $\Omega$ be an open Riemann surface which admits a Green
function. Let $h_{\Omega}$ be a harmonic function on $\Omega$, and
$\rho=e^{-2h_{\Omega}}$. Related to the weighted Bergman kernel,
there is an equivalent form of the extended Suita conjecture in
\cite{yamada98}:

$\\$ \emph{\textbf{Conjecture:} $c_{\beta}^{2}(z_{0})\leq
\pi\rho(z_{0}) B_{\Omega,\rho}(z_{0})$, and the equality holds if
and only if $\chi_{-h}=\chi_{z_{0}}$.} $\\$

The reason of the equivalence between the above two conjectures is
as follows:

By the above argument,
we have $f^{-1}_{h}p^{*}e_{j}\in\Gamma^{\chi_{-h}}$.
Note that $\{f^{-1}_{h}p^{*}e_{j}\}_{j=1,2,\cdots}$ is the orthogonal basis of $\Gamma^{\chi_{-h}}$,
then we have $\rho(z_{0}) B_{\Omega,\rho}(z_{0})=B_{\Omega}^{\chi_{-h}}(z_{0})$.

It suffices to show that for any $\chi$ s.t. $\Gamma^{\chi}$ has a
nonzero element $F_{0}$, there is a harmonic function $h$ on
$\Omega$ which satisfies $\chi=\chi_{h}$.

As $\Omega$ is noncompact, for $F_{0}\in\Gamma^{\chi}$, one can find
a holomorphic function $h_{0}$ on $\Omega$ such that
$F_{0}p^{*}h_{0}^{-1}$ does not have any zero point on $\Delta$.

As $\Omega$ is noncompact, one can find a holomorphic $(1,0)$ form
$H_{0}$ on $\Omega$ such that $H_{0}$ does not have any zero point
on $\Omega$.

Then one obtains a holomorphic function
$f_{1}:=\frac{F_{0}p^{*}h_{0}^{-1}}{p^{*}H^{-1}_{0}}$, which does
not have any zero point. It is clear that $\log|f_{1}|$ is harmonic
and fibre constant, which can be seen as a harmonic function on
$\Omega$.

Note that $f_{1}\in\mathcal{O}^{\chi}(\Omega)$. Set
$h:=\log|f_{1}|$, then $\chi_{h}=\chi$. It is also easy to see that
the equality part of two conjectures are also equivalent. Then we
prove the equivalence of the two conjectures.\\

In \cite{guan-zhou12p}, we have proved $c_{\beta}^{2}(z_{0})\leq \pi\rho(z_{0}) B_{\Omega,\rho}(z_{0})$.
Combining this result in \cite{guan-zhou12p} with Theorem \ref{t:guan-zhou-unify},
we completely solve the extended Suita conjecture:

\begin{Theorem}
\label{c:extend_suita_conj}
(a complete solution of the extended Suita conjecture)
$$c_{\beta}^{2}(z_{0})\leq \pi\rho(z_{0}) B_{\Omega,\rho}(z_{0})$$ holds, and
the equality holds if and only if $\chi_{-h}=\chi_{z_{0}}$.
\end{Theorem}

\subsection{A question posed by Ohsawa}
$\\$

Let $\Omega$ be a Stein manifold with a continuous volume form
$dV_{\Omega}$. Let $D$ be a strongly pseudoconvex relatively compact
domain in $\Omega$, with $C^{2}$ smooth plurisubharmonic defining
function $\rho$. Let $\delta(z)$ be a distance induced by a
Riemannian metric from $z$ to the boundary $\partial D$ of $D$.

Let $H$ be a closed smooth complex hypersurface on $\Omega$.
 Then there exists a continuous
function $s$ on $\Omega$, which satisfies

$1).$ $H=\{s=0\}$;

$2).$ $s^{2}$ is a smooth function on $\Omega$;

$3).$ $\log|s|$ is a plurisubharmonic function on $\Omega$;

$4).$ for any point $z\in H$, there exists a local holomorphic
defining function $e$ of $H$, such that $2\log|s|-2\log|e|$ is
continuous near $z$.

In fact, associated to the hypersurface $H$, there exists a
holomorphic line bundle $L_{H}$ on $\Omega$ with a smooth Hermitian
metric $h_{H}$ and there is a holomorphic section $f$ of $L_{H}$,
such that $\{f=0\}=H$ and $df|_{z}\neq 0$ for any $z\in H$. As
$\Omega$ is Stein, then there exits a smooth plurisubharmonic
function $s_{1}$ on $\Omega$, which satisfies $s_{1}+\log|f|_{h_H}$
is a plurisubharmonic function on $\Omega$. Let
$s:=e^{s_{1}}|f|_{h_H}$. Then we obtain the existence of the
function $s$.

Assuming that $\partial D$ intersects with $H$ transversally.

Let
$$A_{\alpha,\varphi}^{2}(D):=
\{f\in\Gamma(D,K_{\Omega})|\int_{D}e^{-\varphi}\delta^{\alpha}|f|^{2}dV_{\Omega}<\infty\},$$
and
$$\parallel f\parallel^{2}_{\alpha,\varphi}=(\alpha+1)\int_{D}e^{-\varphi}\delta^{\alpha}|f|^{2}dV_{\Omega},$$

We put

$$A_{-1,\varphi}^{2}(D):=\{f\in\Gamma(D,K_{\Omega})|\lim_{\alpha\searrow-1}
(1+\alpha)\int_{D}e^{-\varphi}\delta^{\alpha}|f|^{2}dV_{\Omega}<\infty\},$$
and
$$\parallel f\parallel^{2}_{-1,\varphi}:=\lim_{\alpha\searrow-1}
(1+\alpha)\int_{D}e^{-\varphi}\delta^{\alpha}|f|^{2}dV_{\Omega}<\infty.$$

In \cite{D-H92}, when $\Omega$ is $\mathbb{C}^{n}$ and $H$ is a
smooth complex hypersurface, Diederich and Herbort gave an $L^{2}$
extension theorem from $A_{\alpha+1,\varphi}^{2}(D\cap H)$ to
$A_{\alpha,\varphi}^{2}(D)$, where $\alpha>-1$.
\begin{Theorem}
\label{t:D-H}\cite{D-H92}
For any $\alpha>-1$,
the extension operator from $A_{\alpha+1,\varphi}^{2}(D\cap H)$ to $A_{\alpha,\varphi}^{2}(D)$ is bounded.
\end{Theorem}

In \cite{ohsawa6}, Ohsawa gave an $L^{2}$ extension theorem from
$A_{0,\varphi}^{2}(D\cap H)$ to $A_{-1,\varphi}^{2}(D)$, which is
called a limiting case.

\begin{Theorem}
\label{t:ohsawa6}\cite{ohsawa6}
The extension operator from $A_{0,\varphi}^{2}(D\cap H)$ to $A_{-1,\varphi}^{2}(D)$ is bounded.
\end{Theorem}

In \cite{ohsawa6}, Ohsawa posed a question about unifying Diederich
and Herbort's theorem with his theorem.

Using Theorem \ref{t:guan-zhou-semicontinu2}, we give the Ohsawa's
question an affirmative answer:

\begin{Theorem}
\label{t:guan-zhou-limiting} Without assuming that $\partial D$
intersects with $H$ transversally. The extension operator from
$A_{\alpha+1,\varphi}^{2}(D\cap H)$ to $A_{\alpha,\varphi}^{2}(D)$
for every $\alpha>-1$ has a bound
$C_{0}\max\{C_{1}^{\alpha},C_{2}^{\alpha}\}$, where $C_{0}$, $C_{1}$
and $C_{2}$ are positive constants, which are independent of
$\alpha$ ($\alpha>-1$).

Consequently, the extension operator from $A_{0,\varphi}^{2}(D\cap
H)$ to $A_{-1,\varphi}^{2}(D)$ is bounded.
\end{Theorem}

\subsection{Application to a log- plurisubharmonicity of the Bergman kernel}\label{sec:subharm_bergman}
$\\$

In this subsection, we give a relation between Theorem
\ref{t:guan-zhou-unify} and Berndtsson's theorem on
log-plurisubharmonicity of Bergman kernel in the following
framework:

Let $M$ be a complex $(n+m)$-dimensional manifold fibred over
complex $m$-dimensional manifold $Y$ with $n$-dimensional fibres,
let $p:M\to Y$ be the projection which satisfies, for any point
$t\in Y$, there exists a unit disc $\Delta_{t}\subset Y$ such that
$(p^{-1}(\Delta_{t}),p^{-1}(t))$ satisfies condition $(a,b)$. Let
$(L,h)$ be a semipositive holomorphic line bundle on $M$ with
Hermitian metric $h$ over $M$.

There are two such examples:

$1),$ $M$ is a pseudoconvex domain in $\mathbb{C}^{n+m}$ with coordinate $(z_{1},\cdots,z_{n},t_{1},\cdots,t_{m})$,
$Y$ is a domain in $\mathbb{C}^{m}$ with coordinate $(t_{1},\cdots,t_{m})$,
$$p((z_{1},\cdots,z_{n},t_{1},\cdots,t_{m}))=(t_{1},\cdots,t_{m});$$

$2),$ $M$ is a projective family, and $Y$ is a complex manifold, and $p$ is a projection map.

Let $(z,t)$ be the coordinate of $S\times\mathbb{B}^{m}$, which is
the local trivialization of the fibration $p$ with fibre $S$, and
$e$ be the local frame of $L$.

Let $\kappa_{M_t}$ be the Bergman kernel of $K_{M_t}\otimes L$ on
$M_{t}$, and $\kappa_{M_t}:=B_{t}(z)dz\otimes e\otimes
d\bar{z}\otimes \bar{e}$ locally.

In this section, we prove that $\log B_{t}(z)$ is plurisubharmonic
with respect to $(z,t)$, using our result on $L^2$ extension with
optimal estimate. It should be noted that we can not get the
log-plurisubharmonicity without optimal estimate.

Without loss of generality, we assume that $Y$ is 1-dimensional.
Then $(z,t)$ is the coordinate of $S\times\Delta_{1}$.

In order to show that $\log B_{t}(z)$ is plurisubharmonic with
respect to $(z,t)$, we need to check that for any complex line
$\emph{L}$ on $(z,t)$, $\log B_{t}(z)|_{\emph{L}}$ is subharmonic.
As we can change the coordinate locally, we only need to check that
for the complex lines $\{t|(z,t)\}$. Then it suffices to check the
submean value inequality for disc small enough (see chapter 1 of
\cite{demailly-book}).

Consider the framework at the beginning of the present subsection.
For any point $w_{0}\in M$, there is a unit disc
$\Delta_{p(w_{0})}\subset Y$, such that
$(p^{-1}\Delta_{p(w_{0})},p^{-1}(p(w_{0})))$ satisfies condition
$(a,b)$. Then we have that $(p^{-1}(\Delta_{1}),p^{-1}(p(w))$
satisfies condition $(a,b)$ for any point $w\in S\times \Delta_{1}$,
by choosing $\Delta_{1}$ small enough.

For any given $t$, if $\kappa_{M_{t}}\not\equiv 0$, by extremal
property of Bergman kernel, there exists a holomorphic section
$u_{t}$ of $K_{p^{-1}(t)}\otimes L$ on $p^{-1}(t)$ such that
$$B_{t}(z)=\frac{|g(z,t)|^{2}}{\int_{M_{t}}\frac{1}{2^{n}}\{u_{t},u_{t}\}},$$
where $u_t=g(z,t)dz\otimes e$ on $(z,t)$.

If $\log B_{t_{0}}(z)=-\infty$, we are done.
Then we can assume that
$$B_{t_{0}}(z)=\frac{|g(z)|^{2}}{\int_{M_{t_{0}}}\frac{1}{2^{n}}\{u_{t_{0}},u_{t_{0}}\}_{h}},$$
where $u_{t_{0}}$ is a holomorphic section of
$K_{p^{-1}(t_0)}\otimes L$ on $p^{-1}(t_0)$, and
$u_{t_{0}}=g(z)dz\otimes e$ on $(z,t)$.

Let $\Delta_{r}$ be the unit disc with center $(z,t_{0})$ and radius $r$ on the line $\{t|(z,t)\}$.

In Theorem \ref{t:guan-zhou-unify}, let $\Psi=\log|t|^{2}$ and
$c_{A}\equiv1$ where $A=2\log r$, we obtain a holomorphic section
$\tilde{u}$ on $p^{-1}(p(\Delta_{r}))$ such that

\begin{equation}
\label{equ:bern09.b}
\int_{M_{t_0}}\{u_{t_{0}},u_{t_{0}}\}^{2}_{h}\geq \frac{1}{\pi
r^{2}}\int_{\Delta_{r}}\int_{M_{t}}
\{\frac{\tilde{u}}{dt}|_{M_{t}},\frac{\tilde{u}}{dt}|_{M_{t}}\}_{h}d\lambda_{\Delta_{r}}(t),
\end{equation}
where $\tilde{u}=\tilde{g}(z,t)dz\wedge dt\otimes e$ on $(z,t)$ and
$\tilde{g}(z,t_{0})=g(z)$.

Using extremal property of the Bergman kernel, we have
$$B_{t}(z)\geq \frac{|\tilde{g}(z,t)|^{2}}{\int_{M_{t}}\frac{1}{2^{n}}\{\frac{\tilde{u}}
{dt}|_{M_{t}},\frac{\tilde{u}}{dt}|_{M_{t}}\}_{h}},$$
for any $(z,t)\in\Delta_{r}$,
if $\int_{M_{t}}\{\frac{\tilde{u}}{dt}|_{M_{t}},\frac{\tilde{u}}{dt}|_{M_{t}}\}_{h}\neq 0$.

Note that the Lebesgue measure of
$\{t|\int_{M_{t}}\{\frac{\tilde{u}}{dt}|_{M_{t}},\frac{\tilde{u}}{dt}|_{M_{t}}\}_{h}=
0\}$ is zero. Using convexity of function $y=e^{x}$ and inequality
\ref{equ:bern09.b}, we have
\begin{equation}
\begin{split}
e^{2\log|g(z)|-\log B_{t_{0}}(z)}=\frac{|g(z)|^{2}}{B_{t_{0}}(z)}\geq e^{\frac{1}{\pi r^{2}}\int_{\Delta_{r}}
(2\log |\tilde{g}(z,t)|-\log B_{t}(z))d\lambda_{\Delta_{1}}(t)}.
\end{split}
\end{equation}
Since $\log|\tilde{g}|$ is a plurisubharmonic function, then we
obtain the relation to log-plurisubharmonicity of the Bergman
kernel:

\begin{Corollary}
\label{c:subhar.bergman}
$\log B_{t}(z)$ is a plurisubharmonic function with respect to $(z,t)$.
\end{Corollary}

The above result is due to \cite{berndtsson06} and
\cite{berndtsson98} in the case of example $1)$, and due to
\cite{B-P08} in the case of example $2)$.

\subsection{$L^{p}$ extension theorems with optimal estimates and Ohsawa's question}
$\\$

Denote by the smooth form $dV_{M}=e^{-\varphi}c_{n}dz\wedge
d\bar{z}$ on the local coordinate $z=(z_{1},\cdots,z_{n})$.

Using Theorem \ref{t:guan-zhou-semicontinu2} (resp. Theorem
\ref{t:guan-zhou-semicontinu}) and similar method as in the proof of
Proposition 0.2 in \cite{B-P10} (see also \cite{B-P08b}), we obtain
an $L^{p}$ $(0<p<2)$ extension theorem with optimal estimate:

\begin{Theorem}
\label{t:Lp_GZ} Let $M$ be a Stein manifold, and $S$ be a closed
complex submanifold of $M$. Let $h$ be a smooth metric on a
holomorphic line bundle $L$ on $M$ (resp. holomorphic line bundle
$L$ with locally integrable singular metric $h$), which satisfies

1),
$\sqrt{-1}\frac{p}{2}\Theta_{h}+\frac{2-p}{2}\sqrt{-1}\partial\bar\partial\varphi+
\sqrt{-1}\partial\bar\partial\Psi\geq 0$ on $M\setminus S$,

2),
$a(-\Psi)(\frac{p}{2}\sqrt{-1}\Theta_{h}+\frac{2-p}{2}\sqrt{-1}\partial\bar\partial\varphi+
\sqrt{-1}\partial\bar\partial\Psi)+
\sqrt{-1}\partial\bar\partial\Psi\geq0$ on $M\setminus S$, where $a$
and $\Psi$ are as in Theorem \ref{t:guan-zhou-semicontinu2} (resp.

1),
$\frac{p}{2}\sqrt{-1}\Theta_{h}+\frac{2-p}{2}\sqrt{-1}\partial\bar\partial\varphi+
\sqrt{-1}\partial\bar\partial\Psi\geq0$ in the sense of currents on
$M\setminus S$,

2),
$\frac{p}{2}\sqrt{-1}\Theta_{h}+\frac{2-p}{2}\sqrt{-1}\partial\bar\partial\varphi+
(1+\delta)\sqrt{-1}\partial\bar\partial\Psi\geq0$ in the sense of
currents on $M\setminus S$, where $\Psi$ is as in Theorem
\ref{t:guan-zhou-semicontinu}).

Then for any holomorphic section $f$ of $K_{M}\otimes L|_{S}$ on $S$
satisfying
\begin{equation}
\sum_{k=1}^{n}\frac{\pi^{k}}{k!}\int_{S_{n-k}}|f|^{p}_{h}dV_{M}[\Psi]=1,
\end{equation}
there exists a holomorphic section $F$ of $K_{M}\otimes L$ on $M$
satisfying $F = f$ on $ S$ and
\begin{equation}
\label{equ:optimal_delta_616b}
\int_{M}c_{A}(-\Psi)|F|^{p}_{h}dV_{M}
\leq\frac{1}{\delta}c_{A}(-A)e^{A}+\int_{-A}^{\infty}c_{A}(t)e^{-t}dt,
\end{equation}
where $c_{A}(t)$ is as in Theorem \ref{t:guan-zhou-semicontinu2}
(resp. Theorem \ref{t:guan-zhou-semicontinu}).
\end{Theorem}

By the similar method as in the proof of Theorem
\ref{t:guan-zhou-limiting} and assume
$$h=e^{-\frac{2}{p}(\varphi-\alpha\log(-r+\varepsilon_{0}|s|^{2}))},$$
and
$$dV_{M}=c_{n}e^{-\varphi_{\Omega}}dz\wedge d\bar{z},$$
where $\varphi_{\Omega}$ is a smooth plurisubharmonic function on
$\Omega$ we answer the above mentioned Ohsawa's question for any $p$
$(0<p<2)$ as follows:

\begin{Theorem}
\label{t:guan-zhou-limiting_Lp} Without assuming that $\partial D$
intersects with $H$ transversally. The extension operator from
$A_{\alpha+1,\varphi}^{p}(D\cap H)$ to $A_{\alpha,\varphi}^{p}(D)$
for each $\alpha>-1$ has a bound
$C_{0}\max\{C_{1}^{\alpha},C_{2}^{\alpha}\}$, where $C_{0}$, $C_{1}$
and $C_{2}$ are positive constants, which are independent of
$\alpha$, $\alpha>-1$.

Consequently, the extension operator from $A_{0,\varphi}^{p}(D\cap
H)$ to $A_{-1,\varphi}^{p}(D)$ is bounded.
\end{Theorem}

\subsection{$L^{\frac{2}{m}}$ extension theorems with optimal estimates on Stein manifolds}
$\\$

Replace $L$ by $(m-1)K_{M}+L$. Take $e^{\varphi}$ as the Hermitian
metric on $K_{M}$. Let $p=\frac{2}{m}$.

Using Theorem \ref{t:Lp_GZ}, we give an optimal estimate of the
$L^{\frac{2}{m}}$ extension theorem:

\begin{Theorem}
\label{t:B_GZ} Let $M$ be a Stein manifold and $S$ be a closed
complex submanifold on $M$. Let $h$ be a smooth metric on a
holomorphic line bundle $L$ on $M$ (resp. holomorphic line bundle
$L$ with locally integrable singular metric $h$), which satisfies

1),
$\frac{1}{m}\sqrt{-1}\Theta_{h}+\sqrt{-1}\partial\bar\partial\Psi\geq0$
 on $M\setminus S$,

2),
$a(-\Psi)(\frac{1}{m}\sqrt{-1}\Theta_{h}+\sqrt{-1}\partial\bar\partial\Psi)+\sqrt{-1}\partial\bar\partial\Psi\geq0$
on $M\setminus S$, where $a$ is as in Theorem
\ref{t:guan-zhou-semicontinu2} (resp.

1),
$\frac{1}{m}\sqrt{-1}\Theta_{h}+\sqrt{-1}\partial\bar\partial\Psi\geq0$
in the sense of currents on $M\setminus S$,

2),
$\frac{1}{m}\sqrt{-1}\Theta_{h}+(1+\delta)\sqrt{-1}\partial\bar\partial\Psi\geq0$
in the sense of currents on $M\setminus S$,).

Then for any holomorphic section $f$ of $K_{M}^{m}\otimes L|_{S}$ on
$S$ satisfying
\begin{equation}
\sum_{k=1}^{n}\frac{\pi^{k}}{k!}\int_{S_{n-k}}|f|^{\frac{2}{m}}_{h}dV_{M}[\Psi]=1,
\end{equation}
there exists a holomorphic section $F$ of $mK_{M}\otimes L$ on $M$
satisfying $F = f$ on $ S$ and
\begin{equation}
\label{equ:optimal_delta_616}
\int_{M}c_{A}(-\Psi)|F|^{\frac{2}{m}}_{h}dV_{M}
\leq\frac{1}{\delta}c_{A}(-A)e^{A}+\int_{-A}^{\infty}c_{A}(t)e^{-t}dt,
\end{equation}
where $c_{A}(t)$ is as in Theorem \ref{t:guan-zhou-semicontinu2}
(resp. Theorem \ref{t:guan-zhou-semicontinu}).
\end{Theorem}

Using the arguments in subsection \ref{sec:subharm_bergman}, we
obtain the relation to log-plurisubharmonicity of the fiberwise
$m$-Bergman kernels on Stein manifolds (see \cite{B-P10} or
\cite{B-P08b}).

\subsection{Interpolation hypersurfaces in Bargmann-Fock space}
$\\$

In this subsection, we give an application of Theorem \ref{t:Lp_GZ}
to a generalization of interpolation hypersurfaces in Bargmann-Fock
space (see \cite{ort06}).

We say that $W$ is a uniformly flat submanifold in $\mathbb{C}^n$
 (the case of hypersurface is referred to \cite{ort06}),
if there exists $T$, which is a plurisubharmonic polar function of
$W$ on $\mathbb{C}^{n}$, such that $(\partial\bar\partial
T*\frac{\mathbf{1}_{B(0,r)}}{Vol(B(0,r))})(z)$ has a uniform upper
bound on $\mathbb{C}^{n}$ which is independent of
$z\in\mathbb{C}^{n}$ and $r$.

We say that $W$ is an interpolation submanifold if for each $f\in
\mathfrak{b}\mathfrak{f}^{p}_{\varphi}$ there exists
$F\in\mathfrak{B}\mathfrak{F}^{p}_{\varphi}$ such that $F|_{W}=f$,
where the plurisubharmonic function $\varphi$ satisfies
$$\sqrt{-1}\partial\bar\partial\varphi\simeq\omega=\sqrt{-1}\partial\bar\partial|z|^{2},$$
where $$\mathfrak{b}\mathfrak{f}^{p}_{\varphi}:= \{f\in
\mathcal{O}(W):\int_{W}|f|^{p}e^{-p\varphi}\omega^{n-1}<+\infty\},$$
and
$$\mathfrak{B}\mathfrak{F}^{p}_{\varphi}:=
\{F\in
\mathcal{O}(\mathbb{C}^{n}):\int_{\mathbb{C}^{n}}|F|^{p}e^{-p\varphi}\omega^{n}<+\infty\}.$$

Let $T$ be a plurisubharmonic function in $\#(W)\cap
C^{\infty}(\mathbb{C}^{n}\setminus W)$, for any $z\in\mathbb{C}^{n}$
and $r>0$, consider the $(1,1)$ form
$$\Upsilon_{W,T}(z,r):=
\sum_{i,j=1}^{n}(\frac{1}{Vol(B(z,r))}\int_{B(z,r)}\frac{\partial^2\log|T|}
{\partial\xi^{i}\partial\bar{\xi}^{j}}\omega^{n}(\xi))\sqrt{-1}dz^{i}\wedge
d\bar{z}^{j}.$$

The density of $W$ in the ball of radius $r$ and center $z$ is
$$D(W,T,z,r):=\sup\{\frac{\Upsilon_{W,T}(z,r)(v,v)}{\sqrt{-1}\partial\bar\partial\varphi_{r}(v,v)}:=
v\in T_{\mathbb{C}^{n},z}-\{0\}\},$$ where
$\varphi_{r}:=\varphi*\frac{\mathbf{1}_{B(0,r)}}{Vol(B(0,r))}$.

The upper density of $W$ is
$$D^{+}(W):=\sup_{T}\limsup_{r\to\infty}\sup_{z\in\mathbb{C}^{n}}D(W,T,z,r).$$

In \cite{ort06}, one of the main results is:

\begin{Theorem}
\label{t:OSV}
\cite{ort06}
Let $W$ be a uniformly flat hypersurface.
Let $p\geq2$.
If $D^{+}<1$,
then $W$ is an interpolation hypersurface.
\end{Theorem}

Using Theorem \ref{t:Lp_GZ}, we obtain a sufficient condition for
interpolation submanifold in Bargmann-Fock space for $p\leq 2$:

\begin{Theorem}
\label{t:interp_GZ}
Let $W$ be a uniformly flat submanifold.
Let $0<p\leq2$.
If $D^{+}<\frac{p}{2}$,
then $W$ is an interpolation submanifold.
\end{Theorem}

\subsection{Optimal estimate of $L^{2}$ extension theorem of Ohsawa}
$\\$

In this subsection, we give some applications of Theorem
\ref{t:guan-zhou-unify} by giving optimal estimate of the $L^{2}$
extension theorem of Ohsawa in \cite{ohsawa2}. Assume that $(M,S)$
satisfies condition $(ab)$.

Let $c_{\infty}(t):=(1+e^{-\frac{t}{m}})^{-m-\varepsilon}$,
where $\varepsilon$ be a positive constant.
It is clear that
$\int_{-\infty}^{\infty}c_{\infty}(t)e^{-t}dt=
m\sum_{j=0}^{m-1}C^{j}_{m-1}(-1)^{m-1-j}\frac{1}{m-1-j+\varepsilon}<\infty.$

Using Remark \ref{r:c_A3}, we obtain that inequality \ref{equ:c_A}
holds for any $t\in(-\infty,+\infty)$. Let
$\Psi=m\log(|g_{1}|^{2}+\cdots+|g_{m}|^{2})$, where
$S=\{g_{1}=\cdots=g_{m}=0\}$, $g_{i}$ are holomorphic functions on
$M$, which satisfies $\wedge_{j=1}^{m}dg_{j}|_{S_{reg}}\neq0$.

Using Theorem \ref{t:guan-zhou-unify} and Lemma \ref{l:lem9}, we
obtain an optimal estimate version of the main result in
\cite{ohsawa2} as follows:

\begin{Corollary}
\label{c:ohsawa2}
For
any holomorphic section $f$ of $K_{S_{reg}}\otimes E|_{S_{reg}}$ on $S_{reg}$ satisfying
$$\frac{\pi^{m}}{m!}\int_{S_{reg}}\{f,f\}_{h}<\infty,$$
there
exists a holomorphic section $F$ of $K_{M}\otimes E$ on $M$ satisfying
$F = f\wedge \bigwedge_{k=1}^{m}dg_{k}$ on $S_{reg}$ and
\begin{eqnarray*}
\begin{split}
&\int_{M}(1+|g_{1}|^{2}+\cdots+|g_{m}|^{2})^{-m-\varepsilon}\{F,F\}_{h}
\\&\leq\mathbf{C}(m\sum_{j=0}^{m-1}C^{j}_{m-1}(-1)^{m-1-j}\frac{1}
{m-1-j+\varepsilon})\frac{(2\pi)^{m}}{m!}\int_{S_{reg}}\{f,f\}_{h},
\end{split}
\end{eqnarray*}
where the uniform constant $\mathbf{C}=1$, which is optimal for any $m$.
\end{Corollary}

When $M$ is Stein, for any plurisubharmomic function $\varphi$ on M,
we can choose a sequence of smooth plurisubharmomic functions
$\{\varphi_{k}\}_{k=1,2,\cdots}$, which is decreasingly convergent
to $\varphi$. Then the above corollary gives optimal estimate
version of the main theorem in \cite{ohsawa2}.\\

Let $c_{\infty}(t):=(1+e^{-t})^{-1-\varepsilon}$, where
$\varepsilon$ is a positive constant. Using Remark \ref{r:c_A3}, we
obtain that inequality \ref{equ:c_A} holds for any
$t\in(-\infty,+\infty)$. Let
$\Psi=m\log(|g_{1}|^{2}+\cdots+|g_{m}|^{2})$, where
$S=\{g_{1}=\cdots=g_{m}=0\}$, $g_{i}$ are the same as in Corollary
\ref{c:ohsawa2}.

Using Theorem \ref{t:guan-zhou-unify} and Lemma \ref{l:lem9}, we can
formulate a similar version to the above corollary with more concise
estimate:

\begin{Corollary}
\label{t:ohsawa2a}
For
any holomorphic section $f$ of $K_{S_{reg}}\otimes E|_{S_{reg}}$ on $S_{reg}$ satisfying
$$\frac{\pi^{m}}{m!}\int_{S_{reg}}\{f,f\}_{h}<\infty,$$
there
exists a holomorphic section $F$ of $K_{M}\otimes E$ on $M$ satisfying
$F = f\wedge \bigwedge_{k=1}^{m}dg_{k}$ on $S$ and
\begin{eqnarray*}
\int_{M}(1+(|g_{1}|^{2}+\cdots+|g_{m}|^{2})^{m})^{-1-\varepsilon}\{F,F\}_{h}
\leq\mathbf{C}\frac{1}{\varepsilon}\frac{(2\pi)^{m}}{m!}\int_{S_{reg}}\{f,f\}_{h},
\end{eqnarray*}
where uniform constant $\mathbf{C}=1$, which is optimal.
\end{Corollary}

Let $M$ be a Stein manifold, and $S$ be an analytic hypersurface on
$M$, which is locally defined by $\{w_{j}=0\}$ on $U_{j}\subset M$,
where $\{U_{j}\}_{j=1,2,\cdots}$ is an open covering of $M$, and
functions $\{w_{j}\}_{j=1,2,\cdots}$ together give a nonzero
holomorphic section $w$ of the holomorphic line bundle $[S]$
associated to $S$(see \cite{G-H}).

Let $|\cdot|$ be a Hermitian metric on $[S]$ satisfying that
$|\cdot|_{e^{-\psi}}$ is semi-negative, where $\psi$ is an
upper-semicontinuous function on $M$. Assume that
$\log(|w|^{2})+\psi<0$ on $M$. Let $\varphi$ be a plurisubharmonic
function.

Let $\Psi:=\log(|w|^{2})+\psi$. Note that $\Psi$ is
plurisubharmonic. Using Lemma \ref{l:lem9}, we have
$$|F|^{2}dV_{M}[\Psi]=2\frac{c_{n-1}\frac{F}{dw}\wedge \bar{\frac{F}{dw}}}{|dw|^{2}}e^{-\psi},$$
for any continuous $(n,0)$ form $F$ on $M$,
where $|dw|$ is the Hermitian metric on $[-S]|_{S_{reg}}$ induced by
Hermitian metric $|\cdot|$ on $[S]|_{S_{reg}}$.

Let $c_{0}(t):=1$. It is easy to see that
$\int_{0}^{\infty}c_{0}(t)e^{-t}dt=1<\infty$ and $c_{0}(t)e^{-t}$ is
decreasing with respect to t, and inequality \ref{equ:c_A} holds for
any $t\in(-\infty,+\infty)$.

Using Theorem \ref{t:guan-zhou-unify}, we obtain another proof of
the following result in \cite{guan-zhou12}, which is an optimal
estimate version of main results in \cite{ohsawa3},
\cite{guan-zhou-zhu11} and \cite{guan-zhou-zhu10}, etc.

\begin{Corollary}
\label{t:ohsawa3}\cite{guan-zhou12}
For
any holomorphic section $f$ of $K_{S_{reg}}$ on $S_{reg}$ satisfying
$$c_{n-1}\int_{S_{reg}}\frac{f\wedge \bar{f}}{|dw|^{2}}e^{-\varphi-\psi}<\infty,$$
there
exists a holomorphic section $F$ of $K_{M}$ on $M$ satisfying $F = f\wedge dw$ on $ S_{reg}$ and
\begin{eqnarray*}
c_{n}\int_{M}F\wedge \bar{F}e^{-\varphi}
\leq 2\pi\mathbf{C}c_{n-1}\int_{S_{reg}}\frac{f\wedge \bar{f}}{|dw|^{2}}e^{-\varphi-\psi},
\end{eqnarray*}
where the uniform constant $\mathbf{C}=1$, which is optimal.
\end{Corollary}

When $w$ is a holomorphic function on $M$,
the above corollary is an optimal estimate version of the $L^{2}$ extension theorems in
\cite{manivel93},
\cite{ohsawa3},
\cite{siu96},
\cite{berndtsson},
\cite{demailly99},
\cite{berndtsson05},
\cite{demailly2010},
\cite{guan-zhou-zhu11},
\cite{guan-zhou-zhu10}, \cite{blocki12}, etc.

\subsection{Optimal constant version of $L^2$ extension theorems of Manivel and Demailly}

\begin{Theorem}\label{t:manivel-demailly}(\cite{manivel93} and \cite{demailly99})
Let $(X,g)$ be a Stein $n$-dimensional manifold possessing a
K\"{a}hler metric $g$, and let $L$ (resp. $E$) be a Hermitian
holomorphic line bundle (resp. a Hermitian holomorphic vector bundle
of rank $r$ over $X$), and $w$ be a global holomorphic section of
$E$. Assume that $w$ is generically transverse to the zero section,
and let
\[H=\{x\in X:\,w(x)=0,\,\wedge^rdw(x)\neq0 \}.\]
Moreover, assume that the $(1,1)$-form $\sqrt{-1}\Theta(L)+r\sqrt{-1}\partial\bar\partial\log|w|^2$ is
semipositive and that there is a continuous function $\alpha\geq 1$ such that the following two inequalities
hold everywhere on X:
\begin{flalign*}
&(a)\text{ }\text{ } \sqrt{-1}\Theta(L)+r\sqrt{-1}\partial\bar\partial\log|w|^2
\geq\frac{\{\sqrt{-1}\Theta(E)w,w\}}{\alpha|w|^2},&
\\&(b)\text{ }\text{ }|w|\leq e^{-\alpha}.&
\end{flalign*}
Then for every holomorphic section $f$ of $\wedge^nT^*_X\otimes L$ over $H$, such that
\[\int_H|f|^2|\wedge^r(dw)|^{-2}dV_H<+\infty,\]
there exists a holomorphic extension $F$ to $X$ such that $F\big|_H=f$ and
\begin{equation}\label{equ:demailly's estimate}
\int_X\frac{|F|^2}{|w|^{2r}(-\log|w|)^2}dV_X
\leq \mathbf{C}\frac{3r}{4}\frac{(2\pi)^{r}}{r!}\int_H\frac{|f|^2}{|\wedge^r(dw)|^2}dV_H,
\end{equation}
where $\mathbf{C}$ is a uniform constant depending only on $r$.
\end{Theorem}

Using Theorem \ref{t:guan-zhou-semicontinu2}, we obtain the
following

\begin{Corollary}
\label{c:manivel-demailly} Theorem \ref{t:manivel-demailly} holds
with the optimal constant $\mathbf{C}=1$ in the estimate
\ref{equ:demailly's estimate}.
\end{Corollary}

\subsection{Optimal estimate for $L^2$ extension theorems of McNeal and Varolin}
$\\$

In \cite{mcneal-varolin}, McNeal and Varolin defined a
function class $\mathfrak{D}$.

\begin{Definition}
The class $\mathfrak{D}$ consists of nonnegative functions with the following three properties.
\\(I) Each $g\in\mathfrak{D}$ is continuous and increasing.
\\(II) For each $g\in\mathfrak{D}$ the improper integral
\[C(g)=\int_1^\infty\frac{dt}{g(t)}<+\infty.\]
\\For $\delta>0$, set
\[H_\delta(y)=\frac{1}{1+\delta}\bigg(1+\frac{\delta}{C(g)}
\int_1^y\frac{dt}{g(t)}\bigg),\]
and note that this function takes values in $(0,1]$. Let
\[g_\delta(x)=\int_1^x\frac{1-H_\delta(y)}{H_\delta(y)}dy.\]
(III) For each $g\in\mathfrak{D}$ there exists a constant $\delta>0$ such that
\[K_\delta(g)=\sup_{x\geq1}\frac{x+g_\delta(x)}{g(x)}<+\infty.\]
\end{Definition}

The extension theorem proved by McNeal and Varolin is stated as
below.
\begin{Theorem}\label{t:mcneal-varolin}
Let $X$ be a K\"{a}hler manifold of complex dimension $n$. Assume there exists a holomorphic function $w$ on $X$,
such that $\sup\limits_X|w|=1$ and $dw$ is never zero on the set $H=\{w=0\}$. Assume there exists an analytic
subvariety $V\subset X$ such that $H$ is not contained in $ V$ and $X\backslash V$ is Stein. Let $L$ be a
holomorphic line bundle over $X$ together with a singular Hermitian metric. Let
$\psi:\,X\longrightarrow[-\infty,+\infty]$ be a locally integrable function such that for any local
representative $e^{-\varphi}$ of the metric of $L$ over an open set $U$, the function $\psi+\varphi$
is not identically $+\infty$ or $-\infty$ on $H\cap U$. Let $g$ be a function in $\mathfrak{D}$. Assume
that $\psi$ satisfies that for all $\gamma>1$ and $\varepsilon>0$ sufficiently small (depending on $\gamma-1$),
\begin{align*}
&\sqrt{-1}\partial\bar\partial
(\varphi+\psi+\log|w|^2)\geq0,
\\&g^{-1}\big(e^{-\psi} g(1-\log|w|^2)\big)\geq1\text{ and}
\\&\alpha-g^{-1}\big(e^{-\psi}g(\alpha)\big)
\text{ is plurisubharmonic},
\end{align*}
where $\alpha=\gamma-\log(|w|^2+\varepsilon^2)$. Then for any
holomorphic $(n-1)$-form $f\in C^\infty(H,\wedge^{n-1}T^*_H\otimes
L)$ on $H$ with values in $L$ such that
\[\int_H\{f,f\}_{e^{-\varphi-\psi}}\,dV_H<+\infty,\]
there is a holomorphic $n$-form $F\in
C^\infty(X,\wedge^nT^*_X\otimes L)$ with value in $L$ such that
$F\big|_H=f\wedge dw$ and
\begin{equation}
\label{equ:mcneal-varolin}\int_X\frac{\{F,F\}_{e^{-\varphi}}}{|w|^2g\big(\log\frac{e}{|w|^2}\big)}
\,dV_X\leq 2\frac{\pi}{e}\mathbf{C}C(g)
\int_{H}\{f,f\}_{e^{-\varphi-\psi}}\,dV_{H},
\end{equation}
where $\mathbf{C}=4\frac{\big(K_\delta(g)+\frac{1+\delta}{\delta}C(g)\big)}{C(g)}$
\end{Theorem}

By Theorem \ref{t:guan-zhou-unify}, it follows that

\begin{Corollary}
\label{c:mcneal-varolin} Theorem \ref{t:mcneal-varolin} holds with
the optimal constant $\mathbf{C}=1$ in the estimate
\ref{equ:mcneal-varolin}, and $g$ only needs to satisfy (I) and
(II).
\end{Corollary}

In section 3 of \cite{mcneal-varolin}, McNeal and Varolin gave
various cases of extension theorems with gains. We give optimal
estimates of their extension theorems:

Let $g(t):=\frac{1}{c_{A}(t)e^{-t}}$, where $g:[1,+\infty]\to[0,+\infty]$ is
in \cite{mcneal-varolin}, $A=-1$. Let $\Psi=\log|w|^{2}$.

Using Corollary \ref{c:mcneal-varolin}, we obtain optimal estimates
for all extension theorems in section 3 of \cite{mcneal-varolin}.

\subsection{Optimal estimate for an $L^2$ extension theorem on projective families}
$\\$

In \cite{siu00}, \cite{paun07} and \cite{berndtsson09}, one has an
$L^2$ extension theorem on projective families:

\begin{Theorem}
\label{t:bern09} Let $M$ be a projective family fibred over the unit
ball in $\mathbb{C}^{m}$, with compact fibers $M_{t}$. Let $(L,h)$
be a holomorphic line bundle on $M$ with a smooth hermitian metric
$h$ of semipositive curvature. Let $u$ be a holomorphic section of
$K_{M_0}\otimes L$ over $M_{0}$ such that
$$\int_{M_0}\{u,u\}_{h}\leq 1.$$
Then there is a holomorphic section $\tilde{u}$ of $K_{M}\otimes L$ over $M$ such that $\tilde{u}|_{M_{0}}=u\wedge dt$,
and
\begin{equation}
\label{equ:bern09.a}
\int_{M}\{\tilde{u},\tilde{u}\}_{h}\leq C_{b}.
\end{equation}
\end{Theorem}

In \cite{paun07}, one can take $C_{b}<200$.

In Theorem \ref{t:guan-zhou-unify}, take $c_{A}=1$ and let
$\Psi:=2m\log|t|$, we obtain an optimal estimate of the above $L^2$
extension theorem:

\begin{Corollary}
\label{c:bern09} Theorem \ref{t:bern09} holds with
$C_{b}=\frac{2^{m}\pi^{m}}{m!}$, which is optimal.
\end{Corollary}

\subsection{Optimal estimate for an $L^2$ extension theorem of Demailly, Hacon and P\u{a}un}
$\\$

In \cite{D-P}, Demailly, Hacon and P\u{a}un gave an $L^2$ extension
theorem in the following framework:

Let $M$ be a Stein manifold and $S$ be a closed complex submanifold
with globally defining function $w$.

Let $\varphi_{F}$, $\varphi_{G_1}$ and $\varphi_{G_2}$ be plurisubharmonic functions on $M$.

Let $\varphi_{S}:=\varphi_{G_1}-\varphi_{G_2}$. Assume that
$\alpha\varphi_{F}-\varphi_{S}$ is plurisubharmonic on $M$,
$|w|^{2}e^{-\varphi_{S}}\leq e^{-\alpha}$, and
$\varphi_{F}\leq\varepsilon_{0}\varphi_{G_2}+C$ on $M$, where
$\alpha\geq1$, $\varepsilon_{0}>1$ and $C$ are all real numbers.

Let $\bar{\varphi}_{S}$ be a smooth function on $M$,
such that $\max_{M} |w|^{2}e^{-\bar{\varphi}_{S}}< \infty$.

Demailly-Hacon-P\u{a}un's $L^2$ extension theorem is as follows:

\begin{Theorem}
\label{t:D-P10}\cite{D-P} Let $u$ be a section of $K_{S}$ satisfying
$$\int_{S}\{u,u\}e^{-\varphi_{F}}<\infty.$$
Then there exists a section $U$ of $K_{M}$,
such that $U|_{S}=u\wedge dw$ and
\begin{equation}
\label{equ:D-P.a}
\int_{M}\{U,U\}e^{-b\varphi_{S}-(1-b)\bar{\varphi}_{S}-\varphi_{F}}\leq C_{b}\int_{S}\{u,u\}e^{-\varphi_{F}},
\end{equation}
where $1\geq b>0$ is an arbitrary real number, and the constant
$$C_{b}=C_{0}b^{-2}(\max_{M} |w|^{2}e^{-\bar{\varphi}_{S}})^{1-b}$$
where $C_{0}$ depends only on the dimension.
\end{Theorem}

Using Theorem \ref{t:guan-zhou-semicontinu2}, we obtain an optimal
estimate of the above extension theorem:

\begin{Corollary}
\label{c:D-P10}
Theorem \ref{t:D-P10} holds with
$$C_{b}=2\pi (\alpha e^{-b\alpha}+\frac{1}{b}e^{-b\alpha})(\max_{M} |w|^{2}e^{-\bar{\varphi}_{S}})^{1-b},$$
which is optimal, without assuming that
$\varphi_{F}\leq\varepsilon_{0}\varphi_{G_2}+C$ on $M$.
\end{Corollary}


\section{Some results used in the proof of main results and applications}

In this section, we give some lemmas which will be used in the
proofs of main theorems and corollaries of the present paper.

\subsection{Some results used in the proofs the main results}
$\\$

In this subsection, we recall some lemmas on $L^{2}$ estimates for
some $\bar\partial$ equations and give some useful lemmas. Denote by
$\bar\partial^*$ or $D''^{*}$ means the Hilbert adjoint operator of
$\bar\partial$.

\begin{Lemma}\label{l:vector}(see \cite{ohsawa3} or \cite{ohsawa-takegoshi})
Let $(X,\omega)$ be a K\"{a}hler manifold of dimension n with a
K\"{a}hler metric $\omega$. Let $(E,h)$ be a hermitian holomorphic
vector bundle. Let $\eta,g>0$ be smooth functions on $X$. Then for
every form $\alpha\in \mathcal {D}(X,\Lambda^{n,q}T_{X}^{*}\otimes
E)$, which is the space of smooth differential forms with values in
$E$ with compact support, we have

\begin{equation}
\label{}
\begin{split}
&\|(\eta+g^{-1})^{\frac{1}{2}}D''^{*}\alpha\|^{2}
+\|\eta^{\frac{1}{2}}D''\alpha\|^{2}
\\&\geq\ll[\eta\sqrt{-1}\Theta_{E}-\sqrt{-1}\partial\bar\partial\eta-
\sqrt{-1}g\partial\eta\wedge\bar\partial\eta,\Lambda_{\omega}]\alpha,\alpha\gg.
\end{split}
\end{equation}
\end{Lemma}

\begin{Lemma}
\label{l:positve} Let $X$ and $E$ be as in the above lemma and
$\theta$ be a continuous $(1,0)$ form on $X$. Then we have
$$[\sqrt{-1}\theta\wedge\bar\theta,\Lambda_{\omega}]\alpha=\bar\theta\wedge(\alpha\llcorner(\bar\theta)^\sharp\big),$$
for any $(n,1)$ form $\alpha$ with value in $E$. Moveover, for any
positive $(1,1)$ form $\beta$, we have $[\beta,\Lambda_{\omega}]$ is
semipositive.
\end{Lemma}

\begin{proof}
For any  $x\in X$, we choose a local coordinate
$(z_{1},\cdots,z_{n})$ near $x$, such that

1), $\theta|_{x}=adz_{1}$;

2), $\omega|_{x}=\sqrt{-1}dz_{1}\wedge d\bar{z}_{1}+\cdots+\sqrt{-1}dz_{n}\wedge d\bar{z}_{n}$.

It suffices to prove $[\sqrt{-1}dz_{1}\wedge d\bar{z}_{1},\Lambda_{\omega}]\alpha=
d\bar{z}_{1}\wedge(\alpha\llcorner(d\bar{z}_{1})^\sharp\big)$.

At $x$, we have $\Lambda_{\omega}\alpha=
(\alpha\llcorner(\sqrt{-1}dz_{1}\wedge d\bar{z}_{1}+\cdots+\sqrt{-1}dz_{n}\wedge d\bar{z}_{n})^\sharp\big)$.
It is clear that $\sqrt{-1}dz_{1}\wedge d\bar{z}_{1}\wedge\Lambda_{\omega}\alpha=\sqrt{-1}dz_{1}\wedge d\bar{z}_{1}
\wedge(\alpha\llcorner(\sqrt{-1}dz_{1}\wedge d\bar{z}_{1})^\sharp\big)$.

Let $\alpha|_{x}=\alpha^{j}e_{j}=
\sum_{k=1}^{n}\alpha^{j}_{k}\bigwedge^{n}_{l=1}dz_{l}\wedge
d\bar{z}_{k}\otimes e_{j}$, where $\{e_{j}\}$ is an orthonormal
basis of $E_{x}$.

Then we have $(\alpha\llcorner(\sqrt{-1}dz_{1}\wedge
d\bar{z}_{1})^\sharp\big)|_{x}=
-\sqrt{-1}(-1)^{n-1}\alpha^{j}_{1}\bigwedge^{n}_{l=2}dz_{l}\otimes
e_{j}$, and
\begin{equation}
\label{}
\begin{split}
[\sqrt{-1}dz_{1}\wedge d\bar{z}_{1},\Lambda_{\omega}]\alpha
=&
\sqrt{-1}dz_{1}\wedge d\bar{z}_{1}\wedge(\Lambda_{\omega}\alpha)
\\=&
\alpha^{j}_{1}\bigwedge^{n}_{l=1}dz_{l}\wedge d\bar{z}_{1}\otimes e_{j}
\\=&d\bar{z}_{1}\wedge(\alpha\llcorner(d\bar{z}_{1})^\sharp\big)
\end{split}
\end{equation}
\end{proof}

\begin{Lemma}\label{l:vector7}(see \cite{demailly99,demailly2010})
Let $X$ be a complete K\"{a}hler manifold equipped with a (non
necessarily complete) K\"{a}hler metric $\omega$, and let $E$ be a
Hermitian vector bundle over $X$. Assume that there are smooth and
bounded functions $\eta$, $g>0$ on $X$ such that the (Hermitian)
curvature operator

$$B:=[\eta\sqrt{-1}\Theta_{E}-\sqrt{-1}\partial\bar\partial\eta-
\sqrt{-1}g\partial\eta\wedge\bar\partial\eta,\Lambda_{\omega}]$$
is positive definite everywhere on $\Lambda^{n,q}T^{*}_{X}\otimes E$, for some $q\geq 1$.
Then for every form $\lambda\in L^{2}(X,\Lambda^{n,q}T^{*}_{X}\otimes E)$ such that $D''\lambda=0$ and
$\int_{X}\langle B^{-1}\lambda,\lambda\rangle dV_{\omega}<\infty$,
there exists $u\in L^{2}(X,\Lambda^{n,q-1}T^{*}_{X}\otimes E)$ such that $D''u=\lambda$ and
$$\int_{X}(\eta+g^{-1})^{-1}|u|^{2}dV_{\omega}\leq\int_{X}\langle B^{-1}\lambda,\lambda\rangle dV_{\omega}.$$
\end{Lemma}\emph{}

For any point $x\in S$, we have a neighborhood $U_{x}\subset M$ of $x$ and a biholomorphic map $p$
from $U_{x}$ to $\Delta^{n}$,
such that $p(U_{x}\cap S)={\Delta}^{dim S_{x}}$,
and $p(U_{x}\setminus S)={\Delta}^{dim S_{x}}\times({\Delta}^{codim S_{x}})^{*}$.
Then we can use the following lemma to study high dimension cases:

\begin{Lemma}
\label{l:lim_disc_unbounded}
Let $\Delta$ be the unit disc, and $\Delta_{r}$ be the disc with radius $r$.
Then for any holomorphic function $f$ on $\Delta$, which satisfies
$$\int_{\Delta}|f|^{2}d\lambda<\infty,$$
we have a uniformly constant $C_{r}=\frac{1}{1-r^{2}}$, which is
only dependent on $r$, such that
$$\int_{\Delta}|f|^{2}d\lambda\leq C_{r}\int_{\Delta\setminus\Delta_{r}}|f|^{2}d\lambda,$$
where $\lambda$ is the Lebesgue measure on $\mathbb{C}$.
\end{Lemma}

\begin{proof}
By Taylor expansion at $o\in \mathbb{C}$, it suffices to check the
Lemma for $f=z^{j}$. By some simple calculations, the Lemma follows.
\end{proof}

Let $L_{h}^{2}(M):=\{F|F\in \Lambda^{n,0}T^{*}M\otimes E,\int_{M}\{F,F\}_{h}<\infty\}$.

We now discuss the convergence of holomorphic $(n,0)$ forms with
values in $E$ as follows:

\begin{Lemma}
\label{l:uniform_converg_compact} Let $M$ be a complex manifold with
dimension $n$ and a continuous volume form $dV_{M}$. Let $E$ be a
holomorphic vector bundle with rank $r$ and $h$ be a Hermitian
metric on $E$. Let $\{F_{j}\}_{j=1,2,\cdots}$ be a sequence of
holomorphic $(n,0)$ forms with values in $E$. Assume that for any
compact subset $K$ of $M$, there exists a constant $C_{K}>0$, such
that
\begin{equation}
\label{equ:L2_hol}
\int_{K}|F_{j}|^{2}_{h}dV_{M}\leq C_{K}
\end{equation}
holds for any $j=1,2,\cdots$. Then we have a subsequence of
$\{F_{j}\}_{j=1,2,\cdots}$, which is uniformly convergent to a
holomorphic section of $K_{M}\otimes E$ on any compact subset of
$M$.
\end{Lemma}

\begin{proof}
We can choose a covering $\{U_{i}\}_{i=1,2,\cdots}$ of $M$,
which satisfies

$1),$ $U_{i}\subset\subset M$, and $\exists K_{i}\subset\subset U_{i}$, such that $\cup_{i=1}^{\infty}K_{i}=M$;

$2),$ $E|_{U_{i}}$ is trivial with holomorphic basis $e_{1}^{i},\cdots,e_{r}^{i}$;

$3),$ $K_{M}|_{U_{i}}$ is trivial with holomorphic basis $v^{i}$.

Then we may write $F_{j}|_{U_{i}}=f^{k}_{j,i}e_{k}^{i}\otimes
v^{i}$, where $f^{k}_{j,i}$ are holomorphic functions on $U_{i}$. As
$h$ is a Hermitian metric and $U_{i}\subset\subset M$, there exists
a constant $B_{K}>0$, such that
$$\sum_{1\leq k,l\leq r}h(e_{k}^{i},e_{l}^{i})f^{k}_{j,i}\bar{f}^{l}_{j,i}\geq B_{K}\sum_{k=1}^{r}|f^{k}_{j,i}|^{2}.$$

By inequality \ref{equ:L2_hol}, it follows that
\begin{equation}
\label{equ:L2_hol_local}\int_{U_{j}}\sum_{k=1}^{r}|f^{k}_{j,i}|^{2}c_{n}v^{i}\wedge \bar{v}^{i}\leq \frac{C_{K}}{B_{K}},
\end{equation}
for any $j=1,2,\cdots$.

We can obtain a subsequence of $\{F_{j}\}_{j=1,2,\cdots}$ which is
uniformly convergent on any compact subset of $M$ by the following
steps:

$1),$ On $U_{1}$, by inequality \ref{equ:L2_hol_local},
we can obtain subsequence $\{F'_{1_{j}}\}_{j=1,2,\cdots}$ of $\{F_{j}\}_{j=1,2,\cdots}$ which is
uniformly convergent on $K_{1}$;

$2),$ On $U_{2}$, by inequality \ref{equ:L2_hol_local},
we can obtain subsequence $\{F'_{2_{j}}\}_{j=1,2,\cdots}$ of $\{F'_{1,j}\}_{j=1,2,\cdots}$ which is
uniformly convergent on $K_{2}$;

$3),$ On $U_{3}$, by inequality \ref{equ:L2_hol_local},
we can obtain subsequence $\{F'_{3_{j}}\}_{j=1,2,\cdots}$ of $\{F'_{2,j}\}_{j=1,2,\cdots}$ which is
uniformly convergent on $K_{3}$;

$\cdots$

As the transition matrix of $E$ is invertible, we see that
$\{F'_{j_{j}}\}_{j=1,2,\cdots}$ is uniformly convergent on any
compact subset of $M$. Thus we have proved the Lemma.

\end{proof}

\begin{Lemma}
\label{l:lim_unbounded} Let $M$ be a complex manifold. Let $S$ be a
closed complex submanifold of $M$. Let $\{U_{j}\}_{j=1,2,\cdots}$ be
a sequence of open subsets on $M$, which satisfies
$$U_{1}\subset U_{2}\subset\cdots\subset U_{j}\subset U_{j+1}\subset\cdots,$$
and $\bigcup_{j=1}^{\infty}U_{j}=M\setminus S$. Let
$\{V_{j}\}_{j=1,2,\cdots}$ be a sequence of open subsets on $M$,
which satisfies
$$V_{1}\subset V_{2}\subset\cdots\subset V_{j}\subset V_{j+1}\subset\cdots,$$
$V_{j}\supset U_{j}$, and $\bigcup_{j=1}^{\infty}V_{j}=M$.

Let $\{g_{j}\}_{j=1,2,\cdots}$ be a sequence of positive Lebesgue
measurable functions on $U_{k}$, which satisfies that $g_{j}$ are
almost everywhere convergent to $g$ on any compact subset of $U_{k}$
$(j\geq k)$, and $g_{j}$ have uniformly positive lower and upper
bounds on any compact subset of $U_{k}$ $(j\geq k)$, where $g$ is a
positive Lebesgue measurable function on $M\setminus S$.

Let $E$ be a holomorphic vector bundle on $M$, with Hermitian metric
$h$. Let $\{F_{j}\}_{j=1,2,\cdots}$ be a sequence of holomorphic
$(n,0)$ form on $V_{j}$ with values in $E$. Assume that
$\lim_{j\to\infty}\int_{U_{j}}\{F_{j},F_{j}\}_{h}g_{j}=C,$ where $C$
is a positive constant.

Then there exists a subsequence $\{F_{j_{l}}\}_{l=1,2,\cdots}$ of
$\{F_{j}\}_{j=1,2,\cdots}$, which satisfies that $\{F_{j_{l}}\}$ is
uniformly convergent to an $(n,0)$ form $F$ on $M$ with value in $E$
on any compact subset of $M$ when $l\to+\infty$, such that
$$\int_{M}\{F,F\}_{h}g\leq C.$$
\end{Lemma}

\begin{proof}
As
$\liminf_{j\to\infty}\int_{U_{j}}\{F_{j},F_{j}\}_{h}g_{j}=C<\infty$,
it follows that there exists a subsequence of $\{F_{j}\}$, denoted
still by $F_{j}$ without ambiguity, such that
$\lim_{j\to\infty}\int_{U_{j}}\{F_{j},F_{j}\}_{h}g_{j}=C.$

By Lemma \ref{l:lim_disc_unbounded}, for any compact set
$K_{k}\subset\subset M$, it follows that there exists
$\tilde{K}_{j_{k}}\subset\subset M\setminus S$, which satisfies
$\tilde{K}_{j_{k}}\subset U_{j_k}$, and
$$\int_{K_{k}}\{F_{j},F_{j}\}_{h}\leq C_{k}\int_{\tilde{K}_{j_k}}\{F_{j},F_{j}\}_{h}g_{j},$$
for any $j\geq j_{k}$, where $C_{k}$ is a constant which is only
dependent on $k$.

Using Lemma \ref{l:uniform_converg_compact}, we have a subsequence
of $F_{j}$, which is uniformly convergent on $K_{k}^{\circ}$,
denoted still by $F_{j}$ without ambiguity. Assume
$\bigcup_{k=1}^{\infty}K_{k}^{\circ}=M$, and $K_{k}\subset\subset
K_{k+1}$.

Using diagonal method for $k$, we obtain a subsequence of $F_{j}$,
denoted by $F_{j}$ without ambiguity, which is uniformly convergent
to a holomorphic $(n,0)$ form $F$ with value in $E$ on any compact
subset of $M$.

Given $\tilde{K}\subset\subset M\setminus S$, as $\{F_{j}\}$ (resp.
$g_{j}$) is uniformly convergent to $F$ (resp. $g$) for $j\geq
k_{\tilde{K}}$, we have $\int_{K}\{F,F\}_{h}g\leq
\lim_{j\to\infty}\int_{U_{j}}\{F_{j},F_{j}\}_{h}g_{j},$ where
$k_{\tilde{K}}$ satisfies $U_{k_{\tilde{K}}}\supset\tilde{K}$. It is
clear that $\int_{M}\{F,F\}_{h}g\leq
\lim_{j\to\infty}\int_{U_{j}}\{F_{j},F_{j}\}_{h}g_{j}.$
\end{proof}

We now give a remark to illustrate the extension properties of
holomorphic sections of holomorphic vector bundles from $M\setminus
X$ to $M$.

\begin{Remark}
\label{r:extend} Let $(M,S)$ satisfy condition $(ab)$, $h$ be a
singular metric on a holomorphic line bundle $L$ on $M$ (resp.
continuous metric on holomorphic vector bundle $E$ on $M$ with rank
$r$) such tah $h$ has locally a positive lower bound. Let $F$ be a
holomorphic section of $K_{M\setminus X}\otimes E|_{M\setminus X}$,
which satisfies $\int_{M\setminus X}|F|^{2}_{h}<\infty$. As $h$ has
locally a positive lower bound and $M$ satisfies $(a)$ of condition
$(ab)$, there is a holomorphic section $\tilde{F}$ of $K_{M}\otimes
L$ on $M$ (resp. $K_{M}\otimes E$), such that
$\tilde{F}|_{M\setminus X}=F$.
\end{Remark}

We now give an approximation property of the function $c_{A}(t)$
$(A<+\infty)$ as follows.

\begin{Lemma}
\label{l:c_A}
Let $c_{A}(t)$ be a positive function in $C^{\infty}((-A,+\infty))$,
which
satisfies $\int_{-A}^{\infty}c_{A}(t)e^{-t}dt<\infty$ and inequality \ref{equ:c_A},
for any $t\in(-A,+\infty)$.
Then there exists a sequence of positive $C^{\infty}$ smooth functions $\{c_{A,m}(t)\}_{m=1,2,\cdots}$
 on $(-A,+\infty)$, which satisfies:

1). $c_{A,m}(t)$ are continuous near $+\infty$ and $\lim_{t\to +\infty}c_{A,m}(t)>0$;

2). $c_{A,m}(t)$ are uniformly convergent to $c_{A}(t)$ on any compact subset of $(-A,+\infty)$,
when $m$ goes to $\infty$;

3). $\int_{-A}^{\infty}c_{A,m}(t)e^{-t}dt$ is convergent to
$\int_{-A}^{\infty}c_{A}(t)e^{-t}dt$ when $m$ approaches to
$\infty$;

4). for any $t\in(-A,+\infty)$,
$$(\int_{-A}^{t}c_{A,m}(t_{1})e^{-t_{1}}dt_{1})^{2}
>c_{A,m}(t)e^{-t}\int_{-A}^{t}
\int_{-A}^{t_{2}}c_{A,m}(t_{1})e^{-t_{1}}dt_{1}dt_{2},$$
holds.

\end{Lemma}

\begin{proof}
We give a construction of $c_{A,m}$.

Firstly, we consider the case that $A<+\infty$.

Let $g_{B}(t):=c_{A}(t)$ when $t\in(-A,-A+B]$,
we can choose $g_{B}(t)$,
which is a positive continuous decreasing function on $t\in[-A+B,\infty)$,
and smooth on $(-A+B,\infty)$, which satisfies $\lim_{t\to+\infty}g_{B}(t)>0$,
such that
\begin{equation}
\label{equ:g_B_int}
\int_{-A+B}^{\infty}g_{B}(t)e^{-t}dt<B^{-1},
\end{equation}
where $B>0$.

As $g_{B}(t)=c_{A}(t)$ when $t\in(-A,-A+B)$,
we have
\begin{equation}
\label{equ:g_Bo}(\int_{-A}^{t}g_{B}(t_{1})e^{-t_{1}}dt_{1})^{2}
>g_{B}(t)e^{-t}\int_{-A}^{t}
\int_{-A}^{t_{2}}g_{B}(t_{1})e^{-t_{1}}dt_{1}dt_{2},
\end{equation}
holds for any $t\in(-A,-A+B)$.
As $g_{B}(t)$ is decreasing on $[-A+B,+\infty)$,
it is clear that inequality \ref{equ:g_Bo}
holds for any $t\in(-A,+\infty)$,
and $\lim_{B\to+\infty}\int_{-A}^{\infty}g_{B}(t)e^{-t}dt=\int_{-A}^{\infty}c_{A}(t)dt$
by inequality $\ref{equ:g_B_int}$.

Given $\varepsilon_{B}$ small enough, such that
$[-A+B-\varepsilon_{B},-A+B+\varepsilon_{B}]\subset\subset
(-A,+\infty)$, one can find a sequence of functions
$\{g_{B,j}(t)\}_{j=1,2,\cdots}$ in $C^{\infty}(-A,+\infty)$,
satisfying $g_{B,j}(t)=g_{B}(t)$ when
$t\notin[-A+B-\varepsilon_{B},-A+B+\varepsilon_{B}]$, which is
uniformly convergent to $G_{B}$. Then it is clear that for $j$ big
enough
$$(\int_{-A}^{t}g_{B,j}(t_{1})e^{-t_{1}}dt_{1})^{2}
>g_{B,j}(t)e^{-t}\int_{-A}^{t}
\int_{-A}^{t_{2}}g_{B,j}(t_{1})e^{-t_{1}}dt_{1}dt_{2},
$$
holds for any $t\in(-A,+\infty)$.

For any given $B$, we can choose $j_{B}$ large enough such that
\begin{equation}
\begin{split}
&1).|\int_{-A}^{\infty}g_{B,j_{B}}(t)e^{-t}dt-\int_{-A}^{\infty}g_{B}(t)dt|<B^{-1};
\\
&2).\max_{t\in(-A,+\infty)}|g_{B,j_{B}}(t)-g_{B}(t)|<B^{-1};
\\
&3).(\int_{-A}^{t}g_{B,j_{B}}(t_{1})e^{-t_{1}}dt_{1})^{2}
>g_{B,j_{B}}(t)e^{-t}\int_{-A}^{t}
\int_{-A}^{t_{2}}g_{B,j_{B}}(t_{1})e^{-t_{1}}dt_{1}dt_{2},\\&\forall t\in(-A,+\infty).
\end{split}
\end{equation}
Let $c_{A,m}:=g_{m,j_{m}}$, thus we have proved the case that
$A<+\infty$.

Secondly, we consider the case that $A=+\infty$. Let
$g_{B}(t):=c_{\infty}(t)$ when $t\in(-\infty,B)$,
$g_{B}(t):=c_{\infty}(B)$ when $t\in[B,\infty)$, where $B>0$.

Using the same construction as the case $A<+\infty$, we obtain the
the case that $A=+\infty$.
\end{proof}

\begin{Remark}
\label{r:c_A_lim}
Let $c_{A}(t)$ be the positive function in Theorem \ref{t:guan-zhou-semicontinu2} and \ref{t:guan-zhou-semicontinu}.
By the construction in the proof of the above lemma,
one can choose a sequence of positive smooth functions $\{c_{A,m}(t)\}_{m=1,2,\cdots}$
on $(-A,+\infty)$,
which are continuous on $[-A,+\infty]$
and uniformly convergent to $c_{A}(t)$ on any compact subset of $(-A,+\infty)$,
and satisfying the same conditions as $c_{A}(t)$ in Theorem \ref{t:guan-zhou-semicontinu2}
and \ref{t:guan-zhou-semicontinu},
such that
$\int_{-A}^{\infty}c_{A,m}(t)e^{-t}dt+\frac{1}{\delta}c_{A,m}(-A)e^{A}$ are convergent to
$\int_{-A}^{\infty}c_{A}(t)e^{-t}dt+\frac{1}{\delta}c_{A}(-A)e^{A}$
when $m$ goes to $\infty$.
\end{Remark}

In fact, we may replace smoothness of $c_{A}(t)$ by continuity:

\begin{Remark}
\label{r:c_A_continu} Using partition of unity $\{\rho_{j}\}_{j}$ on
$(-A,+\infty)$ and smoothing for $\rho_{j}c_{A}$, we can replace
smoothness of $c_{A}(t)$ by continuity in Lemma \ref{l:c_A}.
\end{Remark}

Now we introduce a relationship between inequality \ref{equ:c_A} and
\ref{equ:c_A_delta}.

\begin{Lemma}
\label{l:relate_c_A_delta} Let $c_{A}(t)$ satisfy
$\int_{-A}^{+\infty}c_{A}(t)e^{-t}dt<\infty$ and inequality
\ref{equ:c_A} $(A\in(-\infty,+\infty])$. For each $A'<A$, there
exists $A''$ and $\delta''>0$, such that $A>A''>A'$ and there exists
$c_{A''}(t)\in C^{0}([-A'',+\infty))$ satisfying

1), $c_{A''}(t)=c_{A}(t)|_{[-A',+\infty)}$;

2), $\int_{-A''}^{+\infty}c_{A''}(t)e^{-t}dt+\frac{1}{\delta''}c_{A''}(-A'')e^{A''}=
\int_{-A}^{+\infty}c_{A}(t)e^{-t}dt$;

3), $\int_{-A''}^{t}(\frac{1}{\delta''}c_{A''}(-A'')e^{A''}+\int_{-A''}^{t_{2}}c_{A''}(t_{1})e^{-t_{1}}dt_{1})
dt_{2}+\frac{1}{{\delta''}^{2}}c_{A''}(-A'')e^{A''}<\\\int_{-A}^{t}(\int_{-A}^{t_{2}}c_{A}(t_{1})e^{-t_{1}}dt_{1})
dt_{2}$.
\end{Lemma}

\begin{proof}
Given $A'<A$. Let
$g(t)|_{[-A',+\infty)}:=c_{A}(t)|_{[-A',+\infty)}$. As $c_{A}(t)$
satisfies $\int_{-A}^{+\infty}c_{A}(t)e^{-t}dt<\infty$ and
inequality \ref{equ:c_A} holds $(A\in(-\infty,+\infty])$, we can
choose a continuous function $g(t)$ such that it's decreasing
rapidly enough on $[A'',A']$ ($A''$ can be chosen near $A'$ enough),
and the following holds:

$1).$ $\int_{-A''}^{+\infty}c_{A''}(t)e^{-t}dt+\frac{1}{\delta''}c_{A''}(-A'')e^{A''}=
\int_{-A}^{+\infty}c_{A}(t)e^{-t}dt$;

$2).$
$\int_{-A''}^{t}(\frac{1}{\delta''}c_{A''}(-A'')e^{A''}+\int_{-A''}^{t_{2}}c_{A''}(t_{1})e^{-t_{1}}dt_{1})
dt_{2}+\frac{1}{{\delta''}^{2}}c_{A''}(-A'')e^{A''}<\\\int_{-A}^{t}(\int_{-A}^{t_{2}}c_{A}(t_{1})e^{-t_{1}}dt_{1})
dt_{2}$. Thus we have proved the lemma.
\end{proof}

Since $A$ may be chosen as positive infinity, we have a sufficient
condition for inequality \ref{equ:c_A} holding:

\begin{Remark}\label{r:c_A3}
Assume that $\frac{d}{dt}c_{A}(t)e^{-t}>0$ for $t\in (-A,a)$,
and $\frac{d}{dt}c_{A}(t)e^{-t}\leq0$ for $t\in[a,+\infty)$,
where $A=+\infty$, and $a>-A$.
Assume that $\frac{d^{2}}{dt^{2}}\log(c_{A}(t)e^{-t})<0$ for $t\in (-A,a)$.
Then inequality \ref{equ:c_A} holds.
\end{Remark}

\begin{proof}
Let $H(t,f):=(\int_{-A}^{t}f(t_{1})dt_{1})^{2}-f(t)\int_{-A}^{t}(\int_{-A}^{t_{2}}f(t_{1})dt_{1})dt_{2}$,
where $f(t)$ is a positive smooth function on $(-A,+\infty)$.

Inequality \ref{equ:c_A} becomes $H(t,c_{A}(t)e^{-t})>0$ for any $t\in(-A,+\infty)$.
That is $\frac{H(t,c_{A}(t)e^{-t})}{c_{A}(t)e^{-t}}>0$ for any $t\in(-A,+\infty)$.

It suffices to prove
$\frac{d}{dt}\frac{H(t,c_{A}(t)e^{-t})}{c_{A}(t)e^{-t}}>0$ for any
$t\in(-\infty,a)$, therefore $$H(t,\frac{d}{dt}(c_{A}(t)e^{-t}))>0$$
for any $t\in(-\infty,a)$.

As $\frac{d}{dt}(c_{A}(t)e^{-t})>0$ for any $t\in(-\infty,a)$, it suffices to prove that
$$\frac{d}{dt}\frac{H(t,\frac{d}{dt}(c_{A}(t)e^{-t}))}{\frac{d}{dt}(c_{A}(t)e^{-t})}>0$$ for any $t\in(-\infty,a)$,
which is $H(t,\frac{d}{dt}\frac{d}{dt}(c_{A}(t)e^{-t}))>0$ for any $t\in(-\infty,a)$.

Note that $H(t,\frac{d}{dt}\frac{d}{dt}(c_{A}(t)e^{-t}))=
-(c_{A}(t)e^{-t})^{2}\frac{d}{dt}\frac{d}{dt}\log(c_{A}(t)e^{-t})$.
Thus we have proved the present Remark.
\end{proof}

In the last part of this section, we recall a theorem of Fornaess
and narasimhan on approximation property of plurisubharmonic
functions of Stein manifolds.

\begin{Lemma}
\label{l:FN1}\cite{FN1980} Let $X$ be a Stein manifold and $\varphi \in PSH(X)$. Then there exists a sequence
$\{\varphi_{n}\}_{n=1,2,\cdots}$ of smooth strongly plurisubharmonic functions such that
$\varphi_{n} \downarrow \varphi$.
\end{Lemma}

\subsection{Properties of polar functions}
$\\$

In this subsection, we give some lemmas on properties of polar functions.

\begin{Lemma}\label{l:lem9}
Let $M$ be a complex manifold of dimension $n$ and $S$ be an
$(n-l)-$dimensional closed complex submanifold. Let
$\Psi\in\Delta(S)$. Assume that there exists a local coordinate
$(z_{1},\cdots,z_{n})$ on a neighborhood $U$ of $x\in M$ such that
$\{z_{n-l+1}=,\cdots,z_{n}=0\}=S\cap U$ and
$\psi:=\Psi-l\log(|z_{n-l+1}|^{2}+\cdots+|z_{n}|^{2})$ is continuous
on $U$. Then we have $d\lambda_{z}[\Psi]=e^{-\psi}d\lambda_{z'}$,
where $d\lambda_{z}$ and $d\lambda_{z'}$ denote the Lebesgue
measures on $U$ and $S\cap U$. Especially,
$$|f\wedge dz_{n-l+1}\wedge\cdots\wedge dz_{n}|^{2}_{h}d\lambda_{z}[\Psi]
=2^{l}\{f, f\}_{h}e^{-\psi},$$ where $f$ is a continuous $(n-l,0)$
form with value in the Hermitian vector bundle $(E,h)$ on $S\cap U$.
\end{Lemma}

\begin{proof}
Note that $d\lambda_{z}[l\log(|z_{n-l+1}|^{2}+\cdots+|z_{n}|^{2})] =
d\lambda_{z'}$ for $z=(z',z_{n-l+1},\cdots,z_{n})$. According to the
definition of generalized residue volume form $d\lambda_{z}[\Psi]$
and the continuity of $\psi$, the lemma follows.
\end{proof}

Using similar method as in the proof of the above lemma, we obtain a
remark as follows:

\begin{Remark}
\label{r:rem9} Let $M$ be a complex manifold of dimension $n$, and
$S$ be an $(n-l)-$dimensional closed complex submanifold. Let
$\Psi\in\Delta(S)$. Assume that there exists a local coordinate
$(z_{1},\cdots,z_{l},w_{2l+1},\cdots,w_{2n})$ on a neighborhood $U$
of $x\in M$ such that $\{w_{2l+1}=,\cdots,w_{2n}=0\}=S\cap U$ and
$\psi:=\Psi-l\log(|w_{2l+1}|^{2}+\cdots+|w_{2n}|^{2})$ is continuous
on $U$, where $z'=(z_{1},\cdots,z_{l})$ are complex coordinates, and
$w_{2l+1},\cdots,w_{2n}$ are real coordinates. Then we have
$dV_{z',w}[\Psi]=e^{-\psi}d\lambda_{z'}$, where $dV_{z',w}$ and
$d\lambda_{z'}$ denote the Lebesgue measures on $U$ and $S\cap U$.
Especially,
$$|F|_{S}|^{2}_{h}d\lambda_{z}[\Psi]=\frac{\{F,F\}_{h}}{dV_{z',w}}dV_{z',w}[\Psi]
=\frac{\{F,F\}_{h}}{dV_{z',w}}e^{-\psi}d\lambda_{z'}$$ where $F$ is
a continuous $(n,0)$ form with value in the Hermitian vector bundle
$(E,h)$ on $U$.
\end{Remark}

\begin{Lemma}
\label{l:extension.equ}
Let $d_{1}(t)$ and $d_{2}(t)$ be two
positive continuous functions on $(0,+\infty)$, which satisfy
$$\int_{0}^{+\infty}d_{1}(t)e^{-t}dt=\int_{0}^{+\infty}d_{2}(t)e^{-t}dt<\infty,$$

$$d_{1}(t)|_{\{t>r_{1}\}\cup \{t<r_{3}\}}=d_{2}(t)|_{\{t>r_{1}\}\cup \{t<r_{3}\}},$$

$$d_{1}(t)|_{\{r_{2}<t<r_{1}\}}>d_{2}(t)|_{\{r_{2}<t<r_{1}\}},$$
and
$$d_{1}(t)|_{\{r_{3}<t<r_{2}\}}<d_{2}(t)|_{\{r_{3}<t<r_{2}\}},$$
where $0<r_{3}<r_{2}<r_{1}<+\infty$.
Let $f$ be a holomorphic function on $\Delta$, then we have
$$\int_{\Delta}d_{1}(-\ln(|z|^{2}))|f|^{2}d\lambda\leq\int_{\Delta}d_{2}(-\ln(|z|^{2}))|f|^{2}d\lambda<+\infty,$$
where $\lambda$ is the Lebesgue measure on $\Delta$. Moreover, the
equality holds if and only if $f\equiv f(0)$.
\end{Lemma}

\begin{proof}
Set
$$f(z)=\sum_{k=0}^{\infty}a_{k}z^{k},$$
a Taylor expansion of $f$ at $0$,
which is uniformly convergent on any given compact subset of
$\Delta$.

As
$$\int_{\Delta}d_{1}(-\ln(|z|^{2}))z^{k_{1}}\bar{z}^{k_{2}}d\lambda=0$$
when $k_{1}\neq k_{2}$, it follows that
\begin{equation}
\label{equ:polar.0919.1.a}
\begin{split}
\int_{\Delta}d_{1}(-\ln(|z|^{2}))|f|^{2}d\lambda
&=\int_{\Delta}\sum_{k=0}^{\infty}d_{1}(-\ln(|z|^{2}))|a_{k}|^{2}|z|^{2k}d\lambda
\\&=\pi\sum_{k=0}^{\infty}|a_{k}|^{2}\int_{0}^{+\infty}d_{1}(t)e^{-kt}e^{-t}dt,
\end{split}
\end{equation}
and
\begin{equation}
\label{equ:polar.0919.2.a}
\begin{split}
\int_{\Delta}d_{2}(-\ln(|z|^{2}))|f|^{2}d\lambda
&=\int_{\Delta}\sum_{k=0}^{\infty}d_{2}(-\ln(|z|^{2}))|a_{k}|^{2}|z|^{2k}d\lambda
\\&=\pi\sum_{k=0}^{\infty}|a_{k}|^{2}\int_{0}^{+\infty}d_{2}(t)e^{-kt}e^{-t}dt.
\end{split}
\end{equation}

As
$$\int_{0}^{+\infty}d_{1}(t)e^{-t}dt=\int_{0}^{+\infty}d_{2}(t)e^{-t}dt<\infty,$$

$$d_{1}(t)|_{\{r_{2}<t<r_{1}\}}>d_{2}(t)|_{\{r_{2}<t<r_{1}\}},$$
and
$$d_{1}(t)|_{\{r_{3}<t<r_{2}\}}<d_{2}(t)|_{\{r_{3}<t<r_{2}\}},$$
it follows that
\begin{equation}
\label{equ:polar.0919.c}
\begin{split}
\int_{r_{3}}^{r_{2}}(d_{2}(t)-d_{1}(t))e^{-kt}e^{-t}dt
&>\int_{r_{3}}^{r_{2}}(d_{2}(t)-d_{1}(t))e^{-kr_{2}}e^{-t}dt
\\&=\int_{r_{2}}^{r_{1}}(d_{1}(t)-d_{2}(t))e^{-kr_{2}}e^{-t}dt
\\&>\int_{r_{2}}^{r_{1}}(d_{1}(t)-d_{2}(t))e^{-kt}e^{-t}dt,
\end{split}
\end{equation}
therefore
$$\int_{r_{3}}^{r_{1}}d_{1}(t)e^{-kt}e^{-t}dt<\int_{r_{3}}^{r_{1}}d_{2}(t)e^{-kt}e^{-t}dt,$$
for every $k\geq1$.

Since
$$d_{1}(t)|_{\{t>r_{1}\}\cup \{t<r_{3}\}}=d_{2}(t)|_{\{t>r_{1}\}\cup \{t<r_{3}\}},$$
we have
$$\int_{0}^{+\infty}d_{1}(t)e^{-kt}e^{-t}dt<\int_{0}^{+\infty}d_{2}(t)e^{-kt}e^{-t}dt,$$
for every $k\geq1$.

Comparing the equalities \ref{equ:polar.0919.1.a} and
\ref{equ:polar.0919.2.a}, we obtain that the inequality in the Lemma
holds, and the equality in the Lemma holds if and only if $a_{k}=0$
for any $k\geq1$, i.e. $f=f(0)$. Then we are done.
\end{proof}

Let $\Omega$ be an open Riemann surface.
Let $z_{0}\in\Omega$,
and $V_{z_0}$ be a neighborhood of $z_{0}$ with local coordinate $w$,
such that $w(z_{0})=0$.

Using the above lemma, we have the following lemma on open Riemann
surfaces:

\begin{Lemma}
\label{l:extension_equ.2} Assume that there is a negative
subharmonic function $\Psi$ on $\Omega$, such that
$\Psi|_{V_{z_0}}=\ln|w|^{2}$, and $\Psi|_{\Omega\setminus
V_{z_0}}\geq\sup_{z\in V_{z_0}}\Psi(z)$. Let $d_{1}(t)$ and
$d_{2}(t)$ be two positive continuous functions on $(0,+\infty)$ as
in lemma \ref{l:extension.equ}. Assume that
$\{\Psi<-r_{3}+1\}\subset\subset V_{z_0}$ is a disc with the
coordinate $z$. Let $F$ be a holomorphic $(1,0)$ form, which
satisfies $F|_{z_{0}}=dw$, then we have
$$\int_{\Omega}d_{1}(-\Psi)\sqrt{-1}F\wedge\bar{F}\leq\int_{\Omega}d_{2}(-\Psi)\sqrt{-1}F\wedge\bar{F}<+\infty,$$
Moreover, the equality holds if and only if $F|_{V_{z_0}}=dw$.
\end{Lemma}

\begin{proof}
It is clear that
\begin{equation}
\label{equ:polar.0921.1}
\begin{split}
&\int_{\Omega}d_{1}(-\Psi)\sqrt{-1}F\wedge\bar{F}
=
\\&\int_{\{\log|w|^{2}<-r_{3}+1\}}d_{1}(-\Psi)\sqrt{-1}|\frac{F}{dw}|^{2}dw\wedge d\bar{w}+
\int_{\Omega\setminus\{\log|w|^{2}<-r_{3}+1\}}d_{1}(-\Psi)\sqrt{-1}F\wedge\bar{F},
\end{split}
\end{equation}
\begin{equation}
\label{equ:polar.0921.2}
\begin{split}
&\int_{\Omega}d_{2}(-\Psi)\sqrt{-1}F\wedge\bar{F}
=
\\&\int_{\{\log|w|^{2}<-r_{3}+1\}}d_{2}(-\Psi)\sqrt{-1}|\frac{F}{dw}|^{2}dw\wedge d\bar{w}+
\int_{\Omega\setminus\{\log|w|^{2}<-r_{3}+1\}}d_{2}(-\Psi)\sqrt{-1}F\wedge\bar{F}.
\end{split}
\end{equation}
Note that $-\Psi|_{\Omega\setminus\{\log|w|^{2}<-r_{3}+1\}}<r_{3}-1$,
then
$$\int_{\Omega\setminus\{\log|w|^{2}<-r_{3}+1\}}d_{1}(-\Psi)\sqrt{-1}F\wedge\bar{F}
=\int_{\Omega\setminus\{\log|w|^{2}<-r_{3}+1\}}d_{2}(-\Psi)\sqrt{-1}F\wedge\bar{F}.$$

Applying Lemma \ref{l:extension.equ} to the rest parts of equalities
\ref{equ:polar.0921.1} and \ref{equ:polar.0921.2}, we get the
present lemma.
\end{proof}

Let $\Omega$ be an open Riemann surface with a Green function. Let
$p:\Delta\to \Omega$ be the universal covering of $\Omega$. We can
choose $V_{z_{0}}$ small enough, such that $p$ restricted on any
component of $p^{-1}(V_{z_{0}})$ is biholomorphic. Let $h$ be a
harmonic function on $\Omega$ and $\rho=e^{-2h}$. As $h$ is harmonic
on $\Omega$, then there exists a multiplicative holomorphic function
$f_{h}$ on $\Delta$, such that $|f_{h}|=e^{p^{*}h}=p^{*}e^{h}$. Let
$f_{-h}:=f_{h}^{-1}$. Let $f_{-h,j}:=f_{-h}|_{U_{j}}$ and
$p_{j}:=p|_{U_{j}}$, where $U_{j}$ is a component of
$p^{-1}(V_{z_{0}})$ for any fixed $j$.

Using Lemma \ref{l:extension.equ}, we obtain the following lemma.

\begin{Lemma}
\label{l:extension_equ.extended} Let $\Omega$ be an open Riemann
surface with Green function $G_{\Omega}$. Let $z_{0}\in\Omega$, and
$V_{z_0}$ be a neighborhood of $z_{0}$ with local coordinate $w$,
such that $w(z_{0})=0$. Assume that there is a negative subharmonic
function $\Psi$ on $\Omega$, such that $\Psi|_{V_{z_0}}=\ln|w|^{2}$
and $\Psi|_{\Omega\setminus V_{z_0}}\geq\sup_{z\in V_{z_0}}\Psi(z)$.
Let $d_{1}(t)$ and $d_{2}(t)$ be two positive continuous functions
on $(0,+\infty)$ as in Lemma \ref{l:extension.equ}. Assume that
$\{\Psi<-r_{3}+1\}\subset\subset V_{z_0}$, which is a disc with the
coordinate $w$. Let $F$ be a holomorphic $(1,0)$ form, which
satisfies $ ((p_{j})_{*}(f_{-h,j}))F|_{z_{0}}=dw$, then we have
$$\int_{\Omega}d_{1}(-\Psi)\sqrt{-1}\rho F\wedge\bar{F}\leq\int_{\Omega}d_{2}(-\Psi)\sqrt{-1}\rho F\wedge\bar{F},$$
Moreover, the equality holds if and only if $((p_{j})_{*}(f_{-h,j}))F|_{V_{z_0}}=dw$.
\end{Lemma}

\begin{proof}
It is clear that
\begin{equation}
\label{equ:polar.0921.3}
\begin{split}
&\int_{\Omega}d_{1}(-\Psi)\rho \sqrt{-1}F\wedge\bar{F}
=
\\&\int_{\{\log|w|^{2}<-r_{3}+1\}}d_{1}(-\Psi)\rho \sqrt{-1}|\frac{F}{dw}|^{2}dw\wedge
d\bar{w}+\int_{\Omega\setminus
\{\log|w|^{2}<-r_{3}+1\}}d_{1}(-\Psi)\rho \sqrt{-1}F\wedge\bar{F},
\end{split}
\end{equation}

\begin{equation}
\label{equ:polar.0921.4}
\begin{split}
&\int_{\Omega}d_{2}(-\Psi)\rho \sqrt{-1}F\wedge\bar{F}
=
\\&\int_{\{\log|w|^{2}<-r_{3}+1\}}d_{2}(-\Psi)\rho \sqrt{-1}|\frac{F}{dw}|^{2}dw\wedge d\bar{w}+
\int_{\Omega\setminus\{\log|w|^{2}<-r_{3}+1\}}d_{2}(-\Psi)\rho \sqrt{-1}F\wedge\bar{F}.
\end{split}
\end{equation}
Note that
$-\Psi|_{\Omega\setminus\{\log|w|^{2}<-r_{3}+1\}}<r_{3}-1$. Then one
has
$$\int_{\Omega\setminus\{\log|w|^{2}<-r_{3}+1\}}d_{1}(-\Psi)\rho \sqrt{-1}F\wedge\bar{F}
=\int_{\Omega\setminus\{\log|w|^{2}<-r_{3}+1\}}d_{2}(-\Psi)\rho \sqrt{-1}F\wedge\bar{F}.$$

Applying Lemma \ref{l:extension.equ} to the rest parts of equalities
\ref{equ:polar.0921.3} and \ref{equ:polar.0921.4}, we get the
present lemma.
\end{proof}

\subsection{Basic properties of Green function}

Let $\Omega$ be an open Riemann surface with a Green function $G_{\Omega}$,
and $z_{0}$ be a point of $\Omega$
with a fixed local coordinate $w$ on the neighborhood $V_{z_0}$ of $z_{0}$,
such that $w(z_{0})=0$.

\begin{Remark}\label{r:green}(see \cite{sario} or \cite{tsuji})
$G_{\Omega}(z,z_{0})=\sup_{u\in\Delta_{0}(z_{0})}u(z)$, where
$\Delta_{0}(z_{0})$ is the set of negative subharmonic functions on
$\Omega$ satisfying that $u-\log|w|$ has a locally finite upper
bound near $z_{0}$.
\end{Remark}

\begin{Remark}
\label{r:green_harmonic}(see \cite{sario} or \cite{tsuji})
$G_{\Omega}(z,z_{0})$ is harmonic on $\Omega\setminus\{z_{0}\}$,
and
$G_{\Omega}(z,z_{0})-\log|w|$ is harmonic near $z_0$.
\end{Remark}

\subsection{Results used in the proofs of the conjecture of Suita, $L-$conjecture and extended Suita conjecture}
$\\$

In this subsection, we give some lemmas which are used to prove the
conjecture of Suita, $L-$conjecture and the extended Suita
conjecture.

Using Theorem \ref{t:guan-zhou-unify} and Lemma
\ref{l:extension_equ.2}, we obtain the following proposition which
will be used in the proof of the conjecture of Suita.

\begin{Proposition}
\label{p:unique} Let $\Omega$ be an open Riemann surface with Green
function $G_{\Omega}$. Let $z_{0}\in\Omega$ and $V_{z_0}$ be a
neighborhood of $z_{0}$ with local coordinate $w$, such that
$w(z_{0})=0$ and $G_{\Omega}|_{V_{z_0}}=\log|w|$. Assume that there
is a unique holomorphic $(1,0)$ form $F$ on $\Omega$, which
satisfies $F|_{z_{0}}=b_{0}dw$ $(b_{0}$ is a complex constant which
is not $0$, such that
$$\int_{\Omega}\sqrt{-1}F\wedge \bar{F}\leq\pi\int_{z_{0}}|b_{0}dw|^{2}dV_{\Omega}[2G(z,z_{0})].$$
Then $F|_{V_{z_0}}=b_{0}dw$.
\end{Proposition}

\begin{Remark}
In Theorem \ref{t:guan-zhou-unify}, let
$\Psi:=2G_{\Omega}(\cdot,z_{0})+2G_{\Omega}(\cdot,z_{2})$, where
$z_{2}$ near $z_{0}$ and $z_{0}\neq z_{2}$, $c_{A}(t)\equiv1$ and
$A=0$, then we have $F_{2}$ such that $F_{2}|_{z_{0}}=b_{0}dw$,
$F_{2}|_{z_{2}}=0$ and
$$\int_{\Omega}\sqrt{-1}F_{2}\wedge \bar{F}_{2}\leq\pi\int_{z_{0}}|b_{0}dw|^{2}dV_{\Omega}[2G_{\Omega}(\cdot,z_{0})+
2G_{\Omega}(\cdot,z_{2})]<+\infty.$$
If there exists a holomorphic $(1,0)$ form,
which satisfies
$$\int_{\Omega}\sqrt{-1}F\wedge \bar{F}<\pi\int_{z_{0}}|b_{0}dw|^{2}dV_{\Omega}[2G(z,z_{0})],$$
then
there exists $\varepsilon_{0}>0$,
such that for any $\varepsilon\in(0,\varepsilon_{0})$,
$$\int_{\Omega}\sqrt{-1}((1-\varepsilon)F+\varepsilon F_{2})\wedge \overline{((1-\varepsilon)F+\varepsilon F_{2})}
<\pi\int_{z_{0}}|b_{0}dw|^{2}dV_{\Omega}[2G(z,z_{0})].$$
Since $(1-\varepsilon)F+\varepsilon F_{2})|_{o}=b_{0}dw$,
and $(1-\varepsilon)F+\varepsilon F_{2})$ also satisfies the inequality in Proposition \ref{p:unique},
it is a contradiction to the uniqueness of $F$.
Then we have
$$\int_{\Omega}\sqrt{-1}F\wedge \bar{F}=\pi\int_{z_{0}}|b_{0}dw|^{2}dV_{\Omega}[2G(z,z_{0})].$$
\end{Remark}

\begin{proof}(proof of Proposition \ref{p:unique})
Let $\Psi:=2G_{\Omega}(\cdot,z_{0})$.
We can choose $r_{3}$ big enough,
such that $\{\Psi<-r_{3}\}\subset\subset\{\Psi<-r_{3}+1\}\subset\subset V_{z_{0}}$,
and $\{\Psi<-r_{3}+1\}$ is a disc with the coordinate $w$.

Let $d_{1}(t)=1$, one can find smooth $d_{2}(t)$ as in Lemma
\ref{l:extension_equ.2}, such that $d_{2}(t)e^{-t}$ is decreasing
with respect to $t$.

Using Theorem \ref{t:guan-zhou-unify}, we have a holomorphic $(1,0)$
form $F_{1}$ on $\Omega$, which satisfies $F_{1}|_{z_{0}}=b_{0}dw$
and
$$\int_{\Omega}d_{2}(-\Psi)\sqrt{-1}F_{1}\wedge \bar{F}_{1}\leq\pi\int_{z_{0}}|b_{0}dw|^{2}dV_{\Omega}[\Psi].$$

Using Lemma \ref{l:extension_equ.2}, we have
$$\int_{\Omega}\sqrt{-1}F_{1}\wedge \bar{F}_{1}\leq\int_{\Omega}d_{2}(-\Psi)\sqrt{-1}F_{1}\wedge \bar{F}_{1},$$

Therefore,
$$\int_{\Omega}\sqrt{-1}F_{1}\wedge \bar{F}_{1}\leq\pi\int_{z_{0}}|b_{0}dw|^{2}dV_{\Omega}[\Psi].$$

According to the assumption of uniqueness of $F$ and the above
remark, it follows that
$$\int_{\Omega}d_{1}(-\Psi)\sqrt{-1}F_{1}\wedge \bar{F}_{1}=
\int_{\Omega}d_{2}(-\Psi)\sqrt{-1}F_{1}\wedge \bar{F}_{1},$$
and $F_{1}=F$.

Using Lemma \ref{l:extension_equ.2}, we have
$F_{1}|_{V_{z_0}}=b_{0}dw$, therefore $F|_{V_{z_0}}=b_{0}dw$.
\end{proof}

Let $\Omega$ be an open Riemann surface with Green function $G_{\Omega}$.
Let $z_{0}\in\Omega$,
and $V_{z_0}$ be a neighborhood of $z_{0}$ with local coordinate $w$,
such that
$w(z_{0})=0$.
Note that there exists a holomorphic function $f_{0}$ near $z_{0}$,
which is locally injective near $z_0$,
such that $|f_{0}|=e^{G_{\Omega}(\cdot,z_{0})}$.

Let $w=f_{0}$, then we have a local coordinate $w$, such that
$G_{\Omega}(\cdot,z_{0})=\log|w|$ near $z_{0}$.

Using Theorem \ref{t:guan-zhou-unify} and Lemma
\ref{l:extension_equ.extended}, we obtain the following proposition
which will be used in the proof of the extended Suita conjecture.

\begin{Proposition}
\label{p:extended_unique} Let $\Omega$ be an open Riemann surface
with Green function $G_{\Omega}$. Let $z_{0}\in\Omega$, and
$V_{z_0}$ be a neighborhood of $z_{0}$ with local coordinate $w$,
such that $w(z_{0})=0$ and $G_{\Omega}(z,z_{0})|_{V_{z_0}}=\log|w|$.
Assume that there is a unique holomorphic $(1,0)$ form $F$ on
$\Omega$, which satisfies
$((p_{j})_{*}(f_{-h,j}))F|_{z_{0}}=b_{0}dw$ $(b_{0}$ is a complex
constant which is not $0)$, and
$$\int_{\Omega}\sqrt{-1}\rho F\wedge \bar{F}\leq\pi\int_{z_{0}}\rho|b_{0}dw|^{2}dV_{\Omega}[2G(z,z_{0})].$$
Then $((p_{j})_{*}(f_{-h,j}))F|_{V_{z_0}}=b_{0}dw$.
\end{Proposition}

\begin{Remark}
In Theorem \ref{t:guan-zhou-unify}, let
$\Psi:=2G_{\Omega}(\cdot,z_{0})+2G_{\Omega}(\cdot,z_{2})$, where
$z_{2}$ near $z_{0}$ and $z_{0}\neq z_{2}$, $c_{A}(t)\equiv1$ and
$A=0$, then we have $F_{2}$ such that $F_{2}|_{z_{0}}=b_{0}dw$,
$F_{2}|_{z_{2}}=0$, and
$$\int_{\Omega}\sqrt{-1}\rho F_{2}\wedge \bar{F}_{2}
\leq\pi\int_{z_{0}}\rho|b_{0}dw|^{2}dV_{\Omega}[2G(z,z_{0})+2G_{\Omega}(\cdot,z_{2})]<+\infty.$$
If there exists a holomorphic $(1,0)$ form $F$, which satisfies
$$\int_{\Omega}\sqrt{-1}\rho F\wedge \bar{F}<\pi\int_{z_{0}}\rho |b_{0}dw|^{2} dV_{\Omega}[2G(z,z_{0})],$$
then
there exists $\varepsilon_{0}>0$,
such that for any $\varepsilon\in(0,\varepsilon_{0})$,
$$\int_{\Omega}\sqrt{-1}\rho((1-\varepsilon)F+\varepsilon F_{2})\wedge \overline{((1-\varepsilon)F+\varepsilon F_{2})}
<\pi\int_{z_{0}}\rho|b_{0}dw|^{2}dV_{\Omega}[2G(z,z_{0})].$$
Since $(1-\varepsilon)F+\varepsilon F_{2})|_{o}=b_{0}dw$,
and $(1-\varepsilon)F+\varepsilon F_{2})$ also satisfies the inequality in the present Proposition,
it is a contradiction to the uniqueness of $F$.
Then we have
$$\int_{\Omega}\sqrt{-1}\rho F\wedge \bar{F}=\pi\int_{z_{0}}\rho|b_{0}dw|^{2}dV_{\Omega}[2G(z,z_{0})].$$
\end{Remark}

\begin{proof}(Proof of Proposition \ref{p:extended_unique})
Let $\Psi:=2G_{\Omega}(\cdot,z_{0})$.
We can choose $r_{3}$ big enough,
such that $\{\Psi<-r_{3}\}\subset\subset\{\Psi<-r_{3}+1\}\subset\subset V_{z_{0}}$,
and $\{\Psi<-r_{3}+1\}$ is a disc with the coordinate $w$.

Let $d_{1}(t)=1$, one can find smooth $d_{2}(t)$ as in Lemma
\ref{l:extension_equ.extended}, which satisfies that
$d_{2}(t)e^{-t}$ is decreasing with respect to $t$.

From Theorem \ref{t:guan-zhou-unify}, it follows that there exists a
holomorphic $(1,0)$ form $F_{1}$ on $\Omega$, which satisfies
$F_{1}|_{z_{0}}=b_{0}dw$, and
$$\int_{\Omega}d_{2}(-\Psi)\sqrt{-1}\rho F_{1}\wedge \bar{F}_{1}\leq\pi\int_{z_{0}}\rho(z_{0})
|b_{0}dw|^{2}dV_{\Omega}[\Psi].$$

Using Lemma \ref{l:extension_equ.extended}, we have
$$\int_{\Omega}\sqrt{-1}\rho F_{1}\wedge \bar{F}_{1}\leq\int_{\Omega}d_{2}(-\Psi)\rho\sqrt{-1}F_{1}\wedge \bar{F}_{1},$$
Therefore
$$\int_{\Omega}\sqrt{-1}\rho F_{1}\wedge \bar{F}_{1}\leq\pi\int_{z_{0}}\rho(z_{0})|b_{0}dw|^{2}dV_{\Omega}[\Psi].$$

From the assumption of uniqueness of $F$ and the above remark, it
follows that
$$\int_{\Omega}d_{1}(-\Psi)\rho\sqrt{-1}F_{1}\wedge \bar{F}_{1}=
\int_{\Omega}d_{2}(-\Psi)\rho\sqrt{-1}F_{1}\wedge \bar{F}_{1},$$
and $F_{1}=F$.

Using Lemma \ref{l:extension_equ.extended}, we have
$$((p_{j})_{*}(f_{-h,j}))F_{1}|_{V_{z_0}}=b_{0}dw,$$
therefore
$$((p_{j})_{*}(f_{-h,j}))F|_{V_{z_0}}=b_{0}dw.$$
We have thus proved the proposition.
\end{proof}

Let $\Omega$ be an open Riemann surface with a Green function $G$,
and $z_{0}$ be a point of $\Omega$
with a fixed local coordinate $w$ on the neighborhood $V_{z_0}$ of $z_{0}$,
such that $w(z_{0})=0$.

Let $\mathscr{A}_{z_{0}}$ be a family of analytic functions $f$ on
$\Omega$ satisfying the normalization condition: $f|_{z_{0}}=0$ and
$df|_{z_{0}}=dw$. Analytic capacity $c_{B}$ is defined as follows:
$$c_{B}:=c_{B}(z_{0})=\frac{1}{\min_{f\in \mathscr{A}_{z_{0}}}\sup_{z\in\Omega}|f(z)|}.$$

About a relation between $c_{\beta}$ and $c_{B}$, it is well known
that one has $c_{\beta}^{2}(z_{0})\geq c_{B}^{2}(z_{0})$.
Furthermore, one has the following lemma:

\begin{Lemma}
\label{l:c_beta.c_B} If there is a holomorphic function $g$ on
$\Omega$, which satisfies $|g(z)|=\exp G(z,z_{0})$, then we have
$c_{\beta}^{2}(z_{0})= c_{B}^{2}(z_{0})$.
\end{Lemma}

\begin{proof}
For the sake of completeness, we give a proof of the inequality
$c_{\beta}^{2}(z_{0})\geq c_{B}^{2}(z_{0})$.

Consider
$$\mathscr{A}^{M}_{z_{0}}:=\mathscr{A}_{z_{0}}\cap\{f||f|\leq M\},$$
as $|g(z)|=\exp G(z,z_{0})$, then $\mathscr{A}^{M}_{z_{0}}$ is not empty.

As $\mathscr{A}^{M}_{z_{0}}$ is a normal family, there exists a
holomorphic function $f_{1}\in\mathscr{A}_{z_{0}}$, such that
$\sup_{z\in\Omega}|f_{1}|=\min_{f\in
\mathscr{A}_{z_{0}}}\sup_{z\in\Omega}|f(z)|$.

That is $|f_{1}(z)|c_{B}(z_{0})<1$ for any $z\in\Omega$, and note
that $\log(|f_{1}(z)|c_{B}(z_{0}))-\log|w(z)|$ is locally finite on
$V_{z_0}$, then by Remark \ref{r:green}, we have
$\log|f_{1}(z)|c_{B}(z_{0})\leq G(z,z_{0})$, therefore
\begin{equation}
\label{equ:c_B}
\lim_{z\to z_{0}}(\log(|f_{1}(z)|c_{B}(z_{0}))-\log|w(z)|)
\leq\lim_{z\to z_{0}}(G(z,z_{0})-\log|w(z)|).
\end{equation}

As $df_{1}|_{z_0}=dw$, we have $\lim_{z\to
z_{0}}(\log(|f_{1}(z)|-\log|w(z)|)=0$. Then inequality
(\ref{equ:c_B}) implies that $c_{B}(z_{0}) \leq\lim_{z\to
z_{0}}(G(z,z_{0})-\log|w(z)|)=c_{\beta}(z_{0})$.

Then we prove $c_{\beta}^{2}(z_{0})= c_{B}^{2}(z_{0})$ under the assumption in the present lemma.

Suppose that there is a holomorphic function $g$ on $\Omega$ which
satisfies $|g(z)|=\exp G(z,z_{0})$, as
$\sup_{z\in\Omega}|f_{1}|=\min_{f\in
\mathscr{A}_{z_{0}}}\sup_{z\in\Omega}|f(z)|$, we have
$\sup|f_{1}(z)|\leq\sup\frac{|g(z)|}{|g'(z_{0})|}$, therefore
\begin{equation}
\label{equ:green0928a}
\log|f_{1}||g'(z_{0})|\leq 0,
\end{equation}
where $g'(z_{0})=\frac{dg}{dw}|_{z_{0}}$.

As $\log|f_{1}(z))||g'(z_{0})|-\log|w(z)|$ has a locally finite
upper bound near $z_{0}$, we have $\log|f_{1}||g'(z_{0})|\leq
G(z,z_{0})=\log |g|$ by Remark \ref{r:green} (see \cite{ahlfors} or
\cite{tsuji}).

Note that
$$\lim_{z\to
z_{0}}\log(|f_{1}(z)||g'(z_{0})|-\log|w(z)|)=\log|g'(z_{0})|=\lim_{z\to
z_{0}}(\log|g(z)|-\log|w(z)|),$$
it follows that
$$\lim_{z\to
z_{0}}(\log|f_{1}(z)||g'(z_{0})|-\log|g(z)|)=0,$$
therefore
$$\lim_{z\to
z_{0}}(\log|f_{1}(z)||g'(z_{0})|-G(z,z_{0}))=0.$$

From inequality (\ref{equ:green0928a}), it follows that
$\log(|f_{1}(z)||g'(z_{0})|)-G(z,z_{0})$ is a negative subharmonic
function on $\Omega$.

Applying the maximal principle to
$\log(|f_{1}(z)||g'(z_{0})|)-G(z,z_{0})$, since
$$\lim_{z\to
z_{0}}(\log|f_{1}(z)||g'(z_{0})|-G(z,z_{0})=0,$$ we have
$$\log|f_{1}(z)||g'(z_{0})|-G(z,z_{0}))=0,$$
i.e.
$$|f_{1}||g'(z_{0})|=|g|.$$
Then it follows that
$$c_{B}(z_{0})=\frac{1}{\sup_{z\in\Omega}|f_{1}|}=\frac{|g'(z_{0})|}{\sup_{z\in\Omega}|g(z)|}
=|g'(z_{0})|.$$

As $c_{\beta}(z_{0}):=\exp\lim_{z\to
z_{0}}(G(z,z_{0})-\log|w(z)|)=\exp\lim_{z\to z_{0}}(\log
|g(z)|-\log|w(z)|)=|g'(z_{0})|=c_{B}(z_{0}),$ we have
$c_{\beta}(z_{0})=c_{B}(z_{0})$.
\end{proof}

Let's recall the following result of Suita in \cite{suita}:

\begin{Lemma}
\label{l:suita}\cite{suita}Assume that $\Omega$ admits a Green function. Then
$\pi B_{\Omega}(z)\geq c^{2}_{B}(z)$ for any $z\in\Omega$. There exists $z_{0}\in\Omega$ such that
equality holds if and only if
$\Omega$ is conformally equivalent to the unit disc less a (possible) closed set of inner capacity zero.
\end{Lemma}

\begin{Remark}
We now present the relationship between the definition of
$c_{B}(z_{0}):=\sup_{\{f|f_{z_0}=0\&|f|<1\}}|f'(z_{0})|$ in
\cite{suita} and the definition of $c_{B}(z_{0})$ used in the present
paper. When $\mathscr{A}^{M}_{z_{0}}$ is not empty, for any element
$g$ in $\mathscr{A}^{M}_{z_{0}}$, one can normalized the norm of $g$
by $\sup_{z\in \Omega}|g|$ denoted by $\sup|g|$ for convenience,
then it is clear that
\begin{equation}
\begin{split}
\sup_{\{f|f(z_0)=0\&|f|<1\}}|f'(z_{0})|
&=(\min_{\{f|f(z_0)=0\&|f|<1\}}|f'(z_{0})|^{-1})^{-1}
\\&=(\min_{\{g|g(z_0)=0\&|\frac{g}{\sup|g|}|<1\}}|\frac{d}{dt}\frac{g}{\sup|g|}|_{z_0}|^{-1})^{-1}
\\&=(\min_{\{g|g(z_0)=0\&|\frac{g}{\sup|g|}|<1\}}|\frac{\frac{dg}{dt}|_{z_0}}{\sup|g|}|^{-1})^{-1}
\\&=(\min_{g\in \mathscr{A}^{M}_{z_{0}}}\sup|g|)^{-1}
\\&=(\min_{g\in \mathscr{A}_{z_{0}}}\sup|g|)^{-1},
\end{split}
\end{equation}
where $g\in\mathscr{A}^{M}_{z_{0}}$. If $\mathscr{A}^{M}_{z_{0}}$ is
not empty (i.e. $\{f|f(z_0)=0\&|f|<1\}$ doesn't only contains $0$),
the above two definitions of $c_{B}(z_{0})$ are equivalent. If
$\mathscr{A}^{M}_{z_{0}}$ is empty (i.e. $\{f|f(z_0)=0\&|f|<1\}$
only contains $0$), the above two definitions of $c_{B}(z_{0})$ are
both 0. Then the above two definitions of $c_{B}(z_{0})$ are the same.
\end{Remark}

Now we prove an identity theorem of holomorphic maps between complex
spaces, which is useful.

\begin{Lemma}
\label{l:identity} Let $X$ be a irreducible complex space and $Y$ be
a complex space. Let $f,g:X\to Y$ are holomorphic maps. Assume that
for a point $a\in X$, the germs $f_{a}$ and $g_{a}$ of holomorphic
maps $f$ and $g$ satisfy $f_{a}=g_{a}$. Then we have $f=g$.
\end{Lemma}

\begin{proof}
Consider a map $(f,g):X \to Y\times Y$, which is
$(f,g)(x)=(f(x),g(x))$. Denote that $A:=\{x\in X|f(x)=g(x)\}$.

Note that $A=(f,g)^{-1}(\Delta_{Y})$, where $\Delta_{Y}$ is the
diagonal of $Y\times Y$. Then $A$ is an analytic set.

As $f_{a}=g_{a}$, there is a neighborhood $U_{a}$ of $a$ in $X$,
such that $f|_{U_{a}}=g|_{U_{a}}$.

Using the Identity Lemma in \cite{grauert-remmert}, we obtain
$A=X_{a}$, which is the irreducible component of $X$ containing $a$.

As $X$ is irreducible, it is clear that $X=X_{a}=A$. Thus we have
proved the Lemma.
 \end{proof}

\begin{Remark}
By the above Lemma, one can see that if two holomorphic maps $f$ and $g$
from irreducible complex space $X$ to complex space $Y$,
which satisfy $f|_{S}=g|_{S}$, where $S$ is totally real with maximal dimension in $X$, then $f\equiv g$.
\end{Remark}

\begin{Lemma}
\label{l:identity_function}
Let $g_{1}$ and $g_{2}$ be two holomorphic functions on domain $\Omega$ in $\mathbb{C}$,
such that $|g_{1}|=|g_{2}|$,
and $dg_{1}=dg_{2}$. Assume that $dg_{1}=dg_{2}$ are not vanish identically.
Then we have $g_{1}=g_{2}$.
\end{Lemma}

\begin{proof}
As $|g_{1}|=|g_{2}|$,
we have $g_{1}\bar{g}_{1}=g_{2}\bar{g}_{2}$.
Then $g_{1}\bar\partial\bar{g}_{1}=g_{2}\bar\partial\bar{g}_{2}$.

It is known that the zero sets of $\bar\partial\bar{g}_{1}$ and
$\bar\partial\bar{g}_{2}$ are both analytic sets on $\Delta$.

From the assumption, it follows that
$d\bar{g}_{1}=\bar\partial\bar{g}_{1}=\bar\partial\bar{g}_{2}=d\bar{g}_{2}$.

As $dg_{1}$ and $dg_{2}$ are not vanish identically and so are
$\bar\partial\bar{g}_{1}$ and $\bar\partial\bar{g}_{2}$, then
$g_{1}=g_{2}$ on an open subset of $\Delta$.

It is clear that $g_{1}=g_{2}$ on $\Delta$ by the identity theorem
of holomorphic functions.
\end{proof}

\begin{Lemma}
\label{l:minimal}(see \cite{berndtsson05})
Let $\mathcal{H}$ be a Hilbert space with norm $\|\cdot\|$,
and $\mathcal{C}$ be a convex subset of $\mathcal{H}$.
Let $\alpha\in\mathcal{C}$, such that $\|\alpha\|=\inf_{\beta\in\mathcal{C}}\|\beta\|$.
Then $\alpha$ is unique.
\end{Lemma}

\begin{proof}
If not, there are $\alpha_{1}$ and $\alpha_{2}$ in $\mathcal{C}$, such that
$$\|\alpha_{1}\|=\|\alpha_{2}\|=\inf_{\beta\in\mathcal{C}}\|\beta\|.$$

As
$$\|\frac{\alpha_{1}+\alpha_{2}}{2}\|^{2}+\|\frac{\alpha_{1}-\alpha_{2}}{2}\|^{2}=
\frac{\|\alpha_{1}\|^{2}+\|\alpha_{2}\|^{2}}{2}
$$
and $\|\frac{\alpha_{1}-\alpha_{2}}{2}\|>0$,
we have
\begin{equation}
\label{equ:unique_Hilbert}
\|\frac{\alpha_{1}+\alpha_{2}}{2}\|<\sqrt{\frac{\|\alpha_{1}\|^{2}+\|\alpha_{2}\|^{2}}{2}}=
\inf_{\beta\in\mathcal{C}}\|\beta\|.
\end{equation}

Note that $\frac{\alpha_{1}+\alpha_{2}}{2}\in\mathcal{C}$, then
inequality \ref{equ:unique_Hilbert} contradicts with
$\|\alpha_{1}\|=\|\alpha_{2}\|=
\inf_{\beta\in\mathcal{C}}\|\beta\|.$
\end{proof}

\begin{Remark}
\label{r:minimal} Let $\Omega$ be an open Riemann surface with a
Green function $G_{\Omega}$, and $z_{0}$ be a point of $\Omega$ with
a fixed local coordinate $w$ on the neighborhood $V_{z_0}$ of
$z_{0}$, such that $w(z_{0})=0$. Let $c_{0}(t)=1$,
$\Psi=2G_{\Omega}(\cdot,z_{0})$.

From Theorem \ref{t:guan-zhou-unify} and the definition
$c_{\beta}:=\exp\lim_{z\to z_{0}}(G_{\Omega}(z,z_{0})-\log|w(z)|)$,
it follows that there is a holomorphic $(1,0)$ form $F$ on $\Omega$,
which satisfies $F|_{z_{0}}=dw|_{z_{0}}$ and
$$\sqrt{-1}\int_{\Omega}F\wedge\bar{F}\leq\pi\int_{z_{0}}|dw|^{2}dV_{\Omega}[2 G_{\Omega}(z,z_{0})]=
\frac{2\pi}{c_{\beta}^{2}(z_{0})},$$
therefore
$$\pi B_{\Omega}(z_{0})\geq c_{\beta}^{2}(z_{0}),$$
by the extremal property of the Bergman kernel.

If there is another holomorphic $(1,0)$ form $\tilde{F}$ on
$\Omega$, which satisfies $\tilde{F}|_{z_{0}}=dw|_{z_{0}}$, and
$$\sqrt{-1}\int_{\Omega}\tilde{F}\wedge\bar{\tilde{F}}\leq\frac{2\pi}{c_{\beta}^{2}(z_{0})},$$
then holomorphic $(1,0)$ form $\frac{F+\tilde{F}}{2}$ on $\Omega$
satisfies $\frac{F+\tilde{F}}{2}|_{z_{0}}=dw|_{z_{0}}$. According to
the proof of Lemma \ref{l:minimal}, it follows that
$$\sqrt{-1}\int_{\Omega}\frac{F+\tilde{F}}{2}\wedge\overline{\frac{F+\tilde{F}}{2}}<\frac{2\pi}{c_{\beta}^{2}(z_{0})}.$$
therefore
$$\pi B_{\Omega}(z_{0})> c_{\beta}^{2}(z_{0}),$$
by the extremal property of the Bergman kernel.
\end{Remark}

\begin{Remark}
\label{r:minimal_extend}
Let $\Omega$ be an open Riemann surface with a Green function $G_{\Omega}$,
and $z_{0}$ be a point of $\Omega$
with a fixed local coordinate $w$ on the neighborhood $V_{z_0}$ of $z_{0}$,
such that $w(z_{0})=0$.
Let $c_{0}(t)=1$, $\Psi=2G_{\Omega}(\cdot,z_{0})$, $h=\rho$.
By Theorem \ref{t:guan-zhou-unify} and $c_{\beta}:=\exp\lim_{z\to z_{0}}(G_{\Omega}(z,z_{0})-\log|w(z)|)$,
there is a holomorphic $(1,0)$ form $F$ on $\Omega$,
which satisfies $F|_{z_{0}}=dw|_{z_{0}}$,
such that
$$\sqrt{-1}\int_{\Omega}\rho F\wedge\bar{F}\leq\pi\int_{z_{0}}\rho(z_{0})|dw|^{2}dV_{\Omega}[2 G_{\Omega}(z,z_{0})]=
\frac{2\pi\rho(z_{0})}{c_{\beta}^{2}(z_{0})},$$
therefore
$$\pi\rho(z_{0}) B_{\Omega,\rho}(z_{0})\geq c_{\beta}^{2}(z_{0}),$$
by the extremal property of the Bergman kernel.

If there is another holomorphic $(1,0)$ form $\tilde{F}$ on
$\Omega$, which satisfies $\tilde{F}|_{z_{0}}=dw|_{z_{0}}$, and
$$\sqrt{-1}\int_{\Omega}\rho\tilde{F}\wedge\bar{\tilde{F}}\leq\frac{2\pi\rho(z_{0})}{c_{\beta}^{2}(z_{0})},$$
then $(1,0)$ form $\frac{F+\tilde{F}}{2}$ on $\Omega$, which
satisfies $\frac{F+\tilde{F}}{2}|_{z_{0}}=dw|_{z_{0}}$. From the
proof of Lemma \ref{l:minimal}, it follows that
$$\sqrt{-1}\int_{\Omega}\rho\frac{F+\tilde{F}}{2}\wedge\overline{\frac{F+\tilde{F}}{2}}
<\frac{2\pi\rho(z_{0})}{c_{\beta}^{2}(z_{0})}.$$
therefore
$$\pi\rho(z_{0}) B_{\Omega\rho}(z_{0})> c_{\beta}^{2}(z_{0}),$$
by the extremal property of the Bergman kernel.
\end{Remark}

We now show a lemma which will be used to discuss the uniform bound
of a sequence of holomorphic functions:

\begin{Lemma}
\label{l:uniform_bound} Let $\varphi$ be a plurisubharmonic function
on $\Omega\subset\subset\mathbb{C}^{n}$, which is not identically
$-\infty$. Let $\{f_{n}\}_{n=1,2,\cdots}$ be a sequence of
holomorphic functions on $\Omega$, such that
$\int_{\Omega}|f_{n}|^{2}e^{\varphi}<C$, where $C$ is positive
constant which is independent of $n$. Then the sequence
$\{f_{n}\}_{n=1,2,\cdots}$ has a uniform bound on any compact subset
of $\Omega$.
\end{Lemma}

\begin{proof}
Let $K$ be a compact subset of $\Omega$, such that
$0<2r<dist(K,\partial\Omega)$. Let $\Omega_{0}:=\{z|dist(z,K)<r\}$.
As $\varphi$ is plurisubharmonic, then there is $N>0$, such that
$\int_{\Omega}e^{-\frac{\varphi}{N}}dV_{\Omega}<C_{0}<+\infty$.

Note that
\begin{equation}
\label{}
\begin{split}
(\int_{\Omega_{0}}|f_{n}|^{2}e^{\varphi}dV_{\Omega})^{\frac{1}{N+1}}
(\int_{\Omega_{0}}e^{-\frac{\varphi}{N}}dV_{\Omega})^{\frac{N}{N+1}}
\geq\int_{\Omega_{0}}|f_{n}|^{\frac{2}{N+1}}dV_{\Omega}\geq\frac{\pi^{n}r^{n}}{n!}|f_{n}(w)|^{\frac{2}{N+1}},
\end{split}
\end{equation}
where $w\in K$, then the lemma follows.
\end{proof}


\section{Proofs of the main theorems}

In this section,
we give proofs of the main theorems.

\subsection{Proof of Theorem \ref{t:guan-zhou-semicontinu2}}
$\\$

By Remark \ref{r:extend}, it suffices to prove the case that $M$ is a Stein manifold.

By Lemma \ref{l:lim_unbounded} and Remark \ref{r:c_A_lim},
it suffices to prove the case that $c_{A}$ is smooth on $(A,+\infty)$ and continuous on $[A,+\infty]$,
such that $\lim_{t\to+\infty}c_{A}(t)>0$.

Since $M$ is a Stein manifold, there is a sequence of Stein
manifolds $\{D_m\}_{m=1}^\infty$ satisfying $D_m\subset\subset
D_{m+1}$ for all $m$ and
$\overset{\infty}{\underset{m=1}{\cup}}D_m=M$. It's known that all
$D_{m}\setminus S$ are complete K\"{a}hler (\cite{grauert}).

Since $M$ is Stein, there is a holomorphic section $\tilde{F}$ of
$K_{M}$ on $M$ such that $\tilde{F}|_{S}={f}$.

Let $ds_{M}^{2}$ be a K\"{a}hler metric on $M$, and $dV_{M}$ be the
volume form with respect to $ds_{M}^{2}$.

Let
$\{v_{t_0,\varepsilon}\}_{t_{0}\in\mathbb{R},\varepsilon\in(0,\frac{1}{4})}$
be a family of smooth increasing convex functions on $\mathbb{R}$,
such that:

$1).$ $v_{t_{0},\varepsilon}(t)=t$, for $t\geq-t_{0}-\varepsilon$;
$v_{t_{0},\varepsilon}(t)$ is a constant depending on $t_{0}$ and
$\varepsilon$, for
 $t<-t_{0}-1+\varepsilon$;

$2).$ the sequence $v''_{t_0,\varepsilon}(t)$ is pointwise
convergent to $\mathbb{I}_{\{-t_{0}-1< t<-t_{0}\}}$
 when $\varepsilon\to 0$, and
 $0\leq v''_{t_0,\varepsilon}(t)\leq 2$ for any $t\in \mathbb{R}$;

$3).$ the sequence $v_{t_0,\varepsilon}(t)$ is $C^{1}$ convergent to
$b_{t_{0}}(t)$

($$b_{t_{0}}(t):=\int_{-\infty}^{t}(\int_{-\infty}^{t_{2}}\mathbb{I}_{\{-t_{0}-1<
t_{1}<-t_{0}\}}dt_{1})dt_{2}-\int_{-\infty}^{0}
(\int_{-\infty}^{t_{2}}\mathbb{I}_{\{-t_{0}-1<t_{1}<-t_{0}\}}dt_{1})dt_{2}$$
is also a $C^1$ function on $\mathbb{R}$) when $\varepsilon\to 0$,
and $0\leq v'_{t_0,\varepsilon}(t)\leq1$ for any $t\in \mathbb{R}$.

We can construct the family
$\{v_{t_0,\varepsilon}\}_{t_{0}\in\mathbb{R},\varepsilon\in(0,\frac{1}{4})}$
by setting
\begin{equation}
\label{}
\begin{split}
v_{t_0,\varepsilon}(t):=&\int_{-\infty}^{t}\int_{-\infty}^{t_{1}}\frac{1}
{1-2\varepsilon}\mathbb{I}_{\{-t_{0}-1+\varepsilon< s<-t_{0}-\varepsilon\}}
*\rho_{\frac{1}{4}\varepsilon}dsdt_{1}\\&-\int_{-\infty}^{0}\int_{-\infty}^{t_{1}}\frac{1}{1-2\varepsilon}
\mathbb{I}_{\{-t_{0}-1+\varepsilon< s<-t_{0}-\varepsilon\}}*\rho_{\frac{1}{4}\varepsilon}dsdt_{1},
\end{split}
\end{equation}
where $\rho_{\frac{1}{4}\varepsilon}$ is the kernel of convolution satisfying $supp(\rho_{\frac{1}{4}\varepsilon})
\subset (-\frac{1}{4}\varepsilon,\frac{1}{4}\varepsilon)$.

Then we have
$$v'_{t_0,\varepsilon}(t)=\int_{-\infty}^{t}\frac{1}{1-2\varepsilon}
\mathbb{I}_{\{-t_{0}-1+\varepsilon<
s<-t_{0}-\varepsilon\}}*\rho_{\frac{1}{4}\varepsilon}ds,$$
and
$$v''_{t_0,\varepsilon}(t)=
\frac{1}{1-2\varepsilon}\mathbb{I}_{\{-t_{0}-1+\varepsilon<
t<-t_{0}-\varepsilon\}}*\rho_{\frac{1}{4}\varepsilon}.$$

Let $s$ and $u$ be two undetermined real functions which will be
determined later on after doing calculations based on Lemma
\ref{l:vector} and Lemma \ref{l:positve}. Let
$\eta=s(-v_{t_{0},\varepsilon}\circ\Psi)$ and
$\phi=u(-v_{t_{0},\varepsilon}\circ\Psi)$, where $s\in
C^{\infty}((-A,+\infty))$ satisfying $s\geq \frac{1}{\delta}$ and
$u\in C^{\infty}((-A,+\infty))$ satisfying $\lim_{t\to+\infty}u(t)$
exists (which will be determined to be
$=-\log(\frac{1}{\delta}c_{A}(-A)e^{A}+\int_{-A}^{\infty}c_{A}(t)e^{-t}dt)$).
Let $\tilde{h}=he^{-\Psi-\phi}$.

Now let $\alpha\in \mathcal{D}(X,\Lambda^{n,1}T_{D_{m}\setminus
S}^{*}\otimes E)$ be an $E$-valued smooth $(n,1)$- form with compact
support on $D_{m}\setminus S$. Using Lemma \ref{l:vector} and Lemma
\ref{l:positve} and the assumption $\sqrt{-1}\Theta_{he^{-\Psi}}\geq
0$ on $D_{m}\setminus S$, we get

\begin{equation}
\label{equ:smooth.3.31.1}
\begin{split}
&\|(\eta+g^{-1})^{\frac{1}{2}}D''^{*}\alpha\|^{2}_{D_m\setminus S,\tilde{h}}
+\|\eta^{\frac{1}{2}}D''\alpha\|^{2}_{D_m\setminus S,\tilde{h}}
\\&\geq\ll[\eta\sqrt{-1}\Theta_{\tilde{h}}-\sqrt{-1}\partial\bar\partial\eta-
\sqrt{-1}g\partial\eta\wedge\bar\partial\eta,\Lambda_{\omega}]\alpha,\alpha\gg_{D_m\setminus S,\tilde{h}}
\\&=\ll[\eta\sqrt{-1}\partial\bar\partial\phi+\eta\sqrt{-1}\Theta_{he^{-\Psi}}-\sqrt{-1}\partial\bar\partial\eta-
\sqrt{-1}g\partial\eta\wedge\bar\partial\eta,\Lambda_{\omega}]\alpha,\alpha\gg_{D_m\setminus S,\tilde{h}}.
\end{split}
\end{equation}
where $g$ is a positive continuous function on $D_{m}\setminus S$.
We need some calculations to determine $g$.

We have

\begin{equation}
\label{}
\begin{split}
&\partial\bar{\partial}\eta=-s'(-v_{t_0,\varepsilon}\circ \Psi)\partial\bar{\partial}(v_{t_0,\varepsilon}\circ \Psi)
+s''(-v_{t_0,\varepsilon}\circ \Psi)\partial(v_{t_0,\varepsilon}\circ \Psi)\wedge
\bar{\partial}(v_{t_0,\varepsilon}\circ \Psi),
\end{split}
\end{equation}

and
\begin{equation}
\label{}
\begin{split}
&\partial\bar{\partial}\phi=-u'(-v_{t_0,\varepsilon}\circ \Psi)\partial\bar{\partial}v_{t_0,\varepsilon}\circ \Psi
+
u''(-v_{t_0,\varepsilon}\circ \Psi)\partial(v_{t_0,\varepsilon}\circ \Psi)
\wedge\bar{\partial}(v_{t_0,\varepsilon}\circ \Psi).
\end{split}
\end{equation}

Therefore,
\begin{equation}
\label{equ:smooth.vector1}
\begin{split}
&\eta\sqrt{-1}\partial\bar\partial\phi-\sqrt{-1}\partial\bar\partial\eta-
\sqrt{-1}g\partial\eta\wedge\bar\partial\eta
\\=&(s'-su')\sqrt{-1}\partial\bar{\partial}(v_{t_0,\varepsilon}\circ \Psi)
+((u''s-s'')-gs'^{2})\sqrt{-1}\partial(v_{t_0,\varepsilon}\circ \Psi)
\wedge\bar{\partial}(v_{t_0,\varepsilon}\circ \Psi)
\\=&
(s'-su')((v'_{t_0,\varepsilon}\circ\Psi)\sqrt{-1}\partial\bar{\partial}
\Psi+(v''_{t_0,\varepsilon}\circ \Psi)\sqrt{-1}\partial(\Psi)\wedge\bar{\partial}(\Psi))
\\+&((u''s-s'')-gs'^{2})\sqrt{-1}\partial(v_{t_0,\varepsilon}\circ \Psi)
\wedge\bar{\partial}(v_{t_0,\varepsilon}\circ \Psi).
\end{split}
\end{equation}

We omit composite item $(-v_{t_0,\varepsilon}\circ \Psi)$ after $s'-su'$ and $(u''s-s'')-gs'^{2}$
in the above equalities.

It's natural to ask $u''s-s''>0$. Let
$g=\frac{u''s-s''}{s'^{2}}\circ(-v_{t_0,\varepsilon}\circ \Psi)$. We
have
$\eta+g^{-1}=(s+\frac{s'^{2}}{u''s-s''})\circ(-v_{t_0,\varepsilon}\circ
\Psi)$.

Since $\sqrt{-1}\Theta_{he^{-\Psi}}\geq0$,
$a(-\Psi)\sqrt{-1}\Theta_{he^{-\Psi}}+\sqrt{-1}\partial\bar\partial\Psi\geq0$
on $M\setminus S$, and $0\leq v'_{t_{0},\varepsilon}\circ\Psi\leq1$,
we have
\begin{equation}
\eta(1-v'_{t_0,\varepsilon}\circ\Psi)\sqrt{-1}\Theta_{he^{-\Psi}}+
(v'_{t_0,\varepsilon}\circ\Psi)(\eta\sqrt{-1}\Theta_{he^{-\Psi}}+\sqrt{-1}\partial\bar\partial\Psi)\geq 0,
\end{equation}
on $M\setminus S$ for $t_{0}$ big enough, which means that
\begin{equation}
\label{equ:smooth.vector2}
\eta\sqrt{-1}\Theta_{he^{-\Psi}}+(v'_{t_0,\varepsilon}\circ\Psi)\sqrt{-1}\partial\bar{\partial}\Psi\geq 0,
\end{equation}
on $M\setminus S$.

From equality  \ref{equ:smooth.vector1}, in order to do $L^2$
estimate, it's natural to let $s'-su'>0$; since to find $s$ and $u$
is an extremal problem, it's natural to let $s'-su'$ be a constant;
by the boundary condition, the constant should be $1$.

Using the inequality $v'_{t_0,\varepsilon}\geq 0$, Lemma
\ref{l:positve}, equality \ref{equ:smooth.vector1}, and inequalities
\ref{equ:smooth.3.31.1} and \ref{equ:smooth.vector2}, one has
\begin{equation}
\label{equ:smooth.vector3}
\begin{split}
\langle B\alpha, \alpha\rangle_{\tilde{h}}=
&\langle[\eta\sqrt{-1}\Theta_{\tilde{h}}-\sqrt{-1}\partial\bar\partial\eta-\sqrt{-1}g
\partial\eta\wedge\bar\partial\eta,\Lambda_{\omega}]
\alpha,\alpha\rangle_{\tilde{h}}
\\\geq&
\langle[(v''_{t_0,\varepsilon}\circ \Psi)\sqrt{-1}\partial\Psi\wedge\bar{\partial}
\Psi,\Lambda_{\omega}]\alpha,\alpha\rangle_{\tilde{h}}
=\langle (v''_{t_{0},\varepsilon}\circ \Psi)
\bar\partial\Psi\wedge (\alpha\llcorner(\bar\partial\Psi)^\sharp\big ),\alpha\rangle_{\tilde{h}}.
\end{split}
\end{equation}

Using the definition of contraction, Cauchy-Schwarz inequality and
the inequality \ref{equ:smooth.vector3}, we have
\begin{equation}
\label{}
\begin{split}
|\langle (v''_{t_{0},\varepsilon}\circ \Psi)\bar\partial\Psi\wedge \gamma,\tilde{\alpha}\rangle_{\tilde{h}}|^{2}
=&|\langle (v''_{t_{0},\varepsilon}\circ \Psi)
\gamma,\tilde{\alpha}\llcorner(\bar\partial\Psi)^\sharp\big \rangle_{\tilde{h}}|^{2}
\\\leq&\langle( v''_{t_{0},\varepsilon}\circ \Psi) \gamma,\gamma\rangle_{\tilde{h}}
(v''_{t_{0},\varepsilon}\circ \Psi)|\tilde{\alpha}\llcorner(\bar\partial\Psi)^\sharp\big|_{\tilde{h}}^{2}
\\=&\langle (v''_{t_{0},\varepsilon}\circ \Psi)\gamma,\gamma\rangle_{\tilde{h}}
\langle (v''_{t_{0},\varepsilon}\circ \Psi) \bar\partial\Psi\wedge
(\tilde{\alpha}\llcorner(\bar\partial\Psi)^\sharp\big ),\tilde{\alpha}\rangle_{\tilde{h}}
\\\leq&\langle (v''_{t_{0},\varepsilon}\circ \Psi )\gamma,\gamma\rangle_{\tilde{h}}
\langle B\tilde{\alpha},\tilde{\alpha}\rangle_{\tilde{h}},
\end{split}
\end{equation}
for any $(n,0)$ form $\gamma$ and $(n,1)$ form $\tilde{\alpha}$.

Take
$\lambda=\bar{\partial}[(1-v'_{t_0,\varepsilon}(\Psi)){\tilde{F}}]$,
$\gamma=\tilde{F}$, and $\tilde{\alpha}=B^{-1}\bar\partial\Psi\wedge
\tilde{F}$, it follows that
$$\langle B^{-1}\lambda,\lambda\rangle_{\tilde{h}} \leq (v''_{t_0,\varepsilon}\circ{\Psi})| \tilde{F}|^2_{\tilde{h}},$$
 and therefore
 $$\int_{D_m\setminus S}\langle B^{-1}\lambda,\lambda\rangle_{\tilde{h}} dV_{M}\leq \int_{D_m\setminus S}
 (v''_{t_0,\varepsilon}
 \circ{\Psi})| \tilde{F}|^2_{\tilde{h}}dV_{M}.$$

By Lemma \ref{l:vector7}, there exists an $(n,0)$-form
$\gamma_{m,t_0,\varepsilon}$ with value in $E$ on $D_{m}\setminus S$
satisfying
$$\bar{\partial}\gamma_{m,t_0,\varepsilon}=\lambda,$$
and

\begin{equation}
 \label{equ:smooth.vector3.2}
 \begin{split}
 &\int_{ D_m\setminus S}|\gamma_{m,t_0,\varepsilon}|^{2}_{\tilde{h}}(\eta+g^{-1})^{-1}dV_{M}
  \leq\int_{D_m\setminus S}(v''_{t_0,\varepsilon}\circ{\Psi})| \tilde{F}|^2_{\tilde{h}}dV_M.
  \end{split}
\end{equation}

Let $\mu_{1}=e^{v_{t_0,\varepsilon}\circ\Psi}$,
$\mu=\mu_{1}c_{A}(-v_{t_0,\varepsilon}\circ\Psi)e^{\phi}$.

It's natural to ask $\eta$ and $\phi$ to satisfy $\mu\leq
\mathbf{C}(\eta+g^{-1})^{-1}$, which will be discussed at the end of
this subsection, where $\mathbf{C}$ is just the constant in Theorem
\ref{t:guan-zhou-semicontinu2}.

As $v_{t_0,\varepsilon}(\Psi)\geq\Psi$, we have
\begin{equation}
\label{equ:smooth.vector3.8}
\begin{split}
\int_{ D_m\setminus S}|\gamma_{m,t_0,\varepsilon}|^{2}_{h}c_{A}(-v_{t_0,\varepsilon}\circ\Psi)dV_{M}
\leq\int_{ D_m\setminus S}
|\gamma_{m,t_0,\varepsilon}|^{2}_{\tilde{h}}c_{A}(-v_{t_0,\varepsilon}\circ\Psi)\mu_{1}e^{\phi} dV_{M}.
\end{split}
\end{equation}

From inequalities \ref{equ:smooth.vector3.2} and
\ref{equ:smooth.vector3.8}, it follows that
$$\int_{D_m\setminus S}|\gamma_{m,t_0\varepsilon}|^{2}_{h}c_{A}(-v_{t_0,\varepsilon}\circ\Psi)dV_{M}
\leq\mathbf{C}\int_{D_m\setminus S}
(v''_{t_0,\varepsilon}\circ{\Psi})| \tilde{F}|^2_{\tilde{h}}dV_M,$$
under the assumption $\mu\leq\mathbf{C} (\eta+g^{-1})^{-1}$.

For any given $t_{0}$ there exists a neighborhood $U_{0}$ of
$\{\Psi=-\infty\}\cap \overline{D_{m}}$ on $M$, such that for any
$\varepsilon$, $v''_{t_0,\varepsilon}\circ\Psi|_{U_{0}}=0$.
Therefore $\bar\partial\gamma_{m,t_0,\varepsilon}|_{U_0\setminus
S}=0$.

As $\Psi$ is upper-semicontinuous and $\phi$ is bounded on $D_{m}$,
it is easy to see that $\gamma_{m,t_0,\varepsilon}$ is locally
$L^{2}$ integrable along $S$. Then $\gamma_{m,t_0,\varepsilon}$ can
be extended to $U_{0}$ as a holomorphic function, which is denoted
by $\tilde{\gamma}_{m,t_0,\varepsilon}$.

It follows from $\Psi\in \#(S)$ that $e^{-\Psi}$ is disintegrable
near $S$. Then $\tilde{\gamma}_{m,t_0,\varepsilon}$ satisfies
$$\tilde{\gamma}_{m,t_0,\varepsilon}|_{S}=0,$$ and
\begin{equation}
\label{equ:smooth.vector3.3}\int_{ D_m}|\tilde{\gamma}_{m,t_0,\varepsilon}|^{2}_{h}c_{A}
(-v_{t_0,\varepsilon}\circ\Psi)dV_{M}
\leq\frac{\mathbf{C}}{e^{A_{t_0}}}\int_{D_m}
(v''_{t_0,\varepsilon}\circ{\Psi})| \tilde{F}|^2_{he^{-\Psi}}dV_M,
\end{equation}
where $A_{t_0}:=\inf_{t\geq t_0}\{u(t)\}$.

As
$$\lim_{t\to+\infty}u(t)=-\log(\frac{1}{\delta}c_{A}(-A)e^A+\int_{-A}^{+\infty}c_{A}(t)e^{-t}dt),$$
it is easy to see that
$$\lim_{t_{0}\to\infty}\frac{1}{e^{A_{t_0}}}=\frac{1}{\delta}c_{A}(-A)e^A+\int_{-A}^{+\infty}c_{A}(t)e^{-t}dt.$$

Let
$F_{m,t_0,\varepsilon}:=(1-v'_{t_0,\varepsilon}\circ\Psi)\widetilde{F}-\tilde{\gamma}_{m,t_0,\varepsilon}$.
By $\tilde{\gamma}_{m,t_0,\varepsilon}|_{S}=0$, then
$F_{m,t_0,\varepsilon}$ is a holomorphic $(n,0)$-form with value in
$E$ on $D_{m}$ satisfying
$$F_{m,t_0,\varepsilon}|_{S}=\tilde{F}|_{S},$$
and inequality \ref{equ:smooth.vector3.3} can be reformulated
as follows:
\begin{equation}
\label{equ:smooth.vector3.5}
\begin{split}
&\int_{D_m}|F_{m,t_{0},\varepsilon}-(1-v'_{t_{0},\varepsilon}\circ\Psi)
\tilde{F}|^{2}_{h}c_{A}(-v_{t_0,\varepsilon}\circ\Psi)dV_{M}
\\\leq&\frac{\mathbf{C}}{e^{A_{t_0}}}\int_{D_m}(v''_{t_0,\varepsilon}\circ\Psi)|\tilde{F}|^{2}_{he^{-\Psi}}dV_{M}.
\end{split}
\end{equation}

Given $t_0$ and $D_{m}$, it is easy to check that
$(v''_{t_0,\varepsilon}\circ\Psi)|\tilde{F}|^{2}_{he^{-\Psi}}$ has a
uniform bound on $D_{m}$ independent of $\varepsilon$. Then
$$\int_{D_m}|(1-v'_{t_0,\varepsilon}\circ\Psi)\tilde{F}|^{2}_{h}c_{A}(-v_{t_0,\varepsilon}\circ\Psi)dV_{M}$$
and
$$\int_{D_m}v''_{t_0,\varepsilon}\circ\Psi|\tilde{F}|^{2}_{he^{-\Psi}}dV_{M}$$
has a uniform bound independent of $\varepsilon$, for any given
$t_0$ and $D_m$.

Using $\bar\partial F_{m,t_{0},\varepsilon}=0$ and Lemma
\ref{l:uniform_converg_compact}, we can choose a subsequence of
$\{F_{m,t_0,\varepsilon}\}_{\varepsilon}$, such that the chosen
sequence is uniformly convergent on any compact subset of $D_m$,
still denoted by $\{F_{m,t_0,\varepsilon}\}_{\varepsilon}$ without
ambiguity.

For any compact subset $K$ on $D_m$, it is easy to check that
$F_{m,t_0,\varepsilon}$,
$(1-v'_{t_0,\varepsilon}\circ\Psi)\tilde{F}$ and
$(v''_{t_0,\varepsilon}\circ\Psi)|\tilde{F}|^{2}_{he^{-\Psi}}$ have
uniform bounds on $K$ independent of $\varepsilon$.

By the dominated convergence theorem on any compact subset $K$ of
$D_m$ and inequality \ref{equ:smooth.vector3.5}, it follows that
\begin{equation}
\begin{split}
&\int_{K}|F_{m,t_0}-(1-b'_{t_0}(\Psi))\tilde{F}|^{2}_{h}c_{A}(-b_{t_0}(\Psi))dV_{M}
\\\leq&\frac{\mathbf{C}}{e^{A_{t_0}}}\int_{D_m}
(\mathbb{I}_{\{-t_{0}-1< t<-t_{0}\}}\circ\Psi)|\tilde{F}|^{2}_{he^{-\Psi}}dV_{M},
\end{split}
\end{equation}
which implies that
\begin{equation}
\label{equ:smooth.vector3.4}
\begin{split}
&\int_{ D_m}|F_{m,t_0}-(1-b'_{t_0}(\Psi))\tilde{F}|^{2}_{h}c_{A}(-b_{t_0}(\Psi))dV_{M}
\\\leq&\frac{\mathbf{C}}{e^{A_{t_0}}}\int_{D_m}(\mathbb{I}_{\{-t_{0}-1< t<-t_{0}\}}
\circ\Psi)|\tilde{F}|^{2}_{he^{-\Psi}}dV_{M}.
\end{split}
\end{equation}

From the definition of $dV_{M}[\Psi]$ and the inequality
$\sum_{k=1}^{n}\frac{\pi^{k}}{k!}\int_{S_{n-k}}|f|^{2}_{h}dV_{M}[\Psi]<\infty$,
it follows that

\begin{equation}
\label{equ:smooth.vector3.6}
\begin{split}
&\limsup_{t_{0}\to+\infty}\int_{D_m}
(\mathbb{I}_{\{-t_{0}-1< t<-t_{0}\}}\circ\Psi_{v})|\tilde{F}|^{2}_{he^{-\Psi}}dV_{M}
\\\leq&
\limsup_{t_{0}\to+\infty}\int_{M}\mathbb{I}_{\overline{D}_{m}}
(\mathbb{I}_{\{-t_{0}-1<t<-t_{0}\}}\circ\Psi)|\tilde{F}|^{2}_{he^{-\Psi}}dV_{M}
\\\leq&\sum_{k=1}^{n}\frac{\pi^{k}}{k!}\int_{S_{n-k}}\mathbb{I}_{\overline{D}_{m}}|f|^{2}_{h}dV_{M}[\Psi]
\leq\sum_{k=1}^{n}\frac{\pi^{k}}{k!}\int_{S_{n-k}}|f|^{2}_{h}dV_{M}[\Psi]<\infty
\end{split}
\end{equation}

Then $\int_{D_m}(\mathbb{I}_{\{-t_{0}-1<
t<-t_{0}\}}\circ\Psi)|\tilde{F}|^{2}_{he^{-\Psi}}dV_{M}$ have
uniform bounds independent of $t_{0}$ for any given $D_m$, and
\begin{equation}
\label{equ:smooth.vector3.7}
\begin{split}
&\limsup_{t_{0}\to+\infty}\int_{D_m}
(\mathbb{I}_{\{-t_{0}-1< t<-t_{0}\}}\circ\Psi)|\tilde{F}|^{2}_{he^{-\Psi}}dV_{M}
\\\leq&\sum_{k=1}^{n}\frac{\pi^{k}}{k!}\int_{S_{n-k}}|f|^{2}_{h}dV_{M}[\Psi]<\infty.
\end{split}
\end{equation}

It is clear that
$$\int_{ D_m}|F_{m,t_0}-(1-b'_{t_0}(\Psi))\tilde{F}|^{2}_{h}c_{A}(-b_{t_0}(\Psi))dV_{M}$$
has a uniform bound independent of $t_{0}$, for any given $D_m$.

Using the fact that
$$\int_{ D_m}|(1-b'_{t_0}(\Psi))\tilde{F}|^{2}_{h}c_{A}(-b_{t_0}(\Psi))dV_{M}$$
has a uniform bound independent of $t_{0}$, inequality
\ref{equ:smooth.vector3.4}, and the following inequality

\begin{equation}
\label{equ:smooth.vector3.9}
\begin{split}
&(\int_{ D_m}|F_{m,t_0}-(1-b'_{t_0}(\Psi))\tilde{F}|^{2}_{h}c_{A}(-b_{t_0}(\Psi))dV_{M})^{\frac{1}{2}}
\\+&(\int_{ D_m}|(1-b'_{t_0}(\Psi))\tilde{F}|^{2}_{h}c_{A}(-b_{t_0}(\Psi))dV_{M})^{\frac{1}{2}}
\\\geq&
(\int_{ D_m}|F_{m,t_0}|^{2}_{h}c_{A}(-b_{t_0}(\Psi))dV_{M})^{\frac{1}{2}},
\end{split}
\end{equation}
we obtain that $\int_{
D_m}|F_{m,t_0}|^{2}_{h}c_{A}(-b_{t_0}(\Psi))dV_{M}$ has a uniform
bound independent of $t_{0}$.

Using $\bar\partial F_{m,t_{0}}=0$ and Lemma
\ref{l:uniform_converg_compact}, we can choose a subsequence of
$\{F_{m,t_{0}}\}_{t_{0}}$, such that the chosen subsequence is
uniformly convergent on any compact subset of $D_m$, still denoted
by $\{F_{m,t_0}\}_{t_{0}}$ without ambiguity.

For any compact subset $K$ on $D_m$, it is clear that both
$F_{m,t_0}$ and $(1-b'_{t_0}\circ\Psi)\tilde{F}$ have uniform bounds
on $K$ independent of $t_0$.

By inequality \ref{equ:smooth.vector3.4}, inequality
\ref{equ:smooth.vector3.7}, the equality
$$\lim_{t_{0}\to\infty}\frac{1}{e^{A_{t_0}}}=\frac{1}{\delta}c_{A}(-A)e^A+\int_{-A}^{+\infty}c_{A}(t)e^{-t}dt,$$
and the dominated convergence theorem on any compact subset $K$ of
$D_m$, it follows that
\begin{equation}
\begin{split}
&\int_{D_m}\mathbb{I}_{K}|F_{m}|^{2}_{h}c_{A}(-\Psi)dV_{M}
\\&\leq\mathbf{C}(\frac{1}{\delta}c_{A}(-A)e^A+\int_{-A}^{+\infty}c_{A}
(t)e^{-t}dt)\sum_{k=1}^{n}\frac{\pi^{k}}{k!}\int_{S_{n-k}}|f|^{2}_{h}dV_{M}[\Psi],
\end{split}
\end{equation}
which implies that
\begin{equation}
\begin{split}
&\int_{ D_m}|F_{m}|^{2}_{h}c_{A}(-\Psi)dV_{M}
\\&\leq\mathbf{C}(\frac{1}{\delta}c_{A}(-A)e^A+
\int_{-A}^{+\infty}c_{A}(t)e^{-t}dt)\sum_{k=1}^{n}\frac{\pi^{k}}{k!}\int_{S_{n-k}}|f|^{2}_{h}dV_{M}[\Psi],
\end{split}
\end{equation}
where the Lebesgue measure of $\{\Psi=-\infty\}$ is zero.

Define $F_m=0$ on $M\backslash D_m$. Then the weak limit of some
weakly convergent subsequence of $\{F_m\}_{m=1}^\infty$ gives a
holomorphic section $F$ of $K_{M}\otimes E$ on $M$ satisfying
$F|_{S}=\tilde{F}|_{S}$, and
$$\int_{ M}|F|^{2}_{h}c_{A}(-\Psi)dV_{M}
\leq\mathbf{C}(\frac{1}
{\delta}c_{A}(-A)+\int_{-A}^{+\infty}c_{A}(t)e^{-t}dt)\sum_{k=1}^{n}
\frac{\pi^{k}}{k!}\int_{S_{n-k}}|f|^{2}_{h}dV_{M}[\Psi].$$

To finish the proof of Theorem \ref{t:guan-zhou-semicontinu2}, it
suffices to determine $\eta$ and $\phi$ such that $(\eta+g^{-1})\leq
\mathbf{C}c^{-1}_{A}(-v_{t_0,\varepsilon}\circ\Psi)e^{-v_{t_0,\varepsilon}\circ\Psi}e^{-\phi}=\mathbf{C}\mu^{-1}$
on $D_m$.

Recall that $\eta=s(-v_{t_0,\varepsilon}\circ\Psi)$ and
$\phi=u(-v_{t_0,\varepsilon}\circ\Psi)$. So we have $(\eta+g^{-1})
e^{v_{t_0,\varepsilon}\circ\Psi}e^{\phi}=(s+\frac{s'^{2}}{u''s-s''})e^{-t}e^{u}\circ(-v_{t_0,\varepsilon}\circ\Psi)$.

Summarizing the above discussion about $s$ and $u$, we are naturally
led to a system of ODEs:
\begin{equation}
\label{equ:unify2.GZ_unify}
\begin{split}
&1).\,\,(s+\frac{s'^{2}}{u''s-s''})e^{u-t}=\frac{\mathbf{C}}{c_{A}(t)}, \\
&2).\,\,s'-su'=1,
\end{split}
\end{equation}
where $t\in(-A,+\infty)$ and $\mathbf{C}=1$; $s\in
C^{\infty}((-A,+\infty))$ satisfies $s\geq \frac{1}{\delta}$ and
$u\in C^{\infty}((-A,+\infty))$ satisfies
$\lim_{t\to+\infty}u(t)=-\log(\frac{1}{\delta}c_{A}(-A)e^{A}+\int_{-A}^{\infty}c_{A}(t)e^{-t}dt)$
such that $u''s-s''>0$.

We solve the above system of ODEs in subsection \ref{subsec:ODE} and
get the solution of the system of ODEs \ref{equ:unify2.GZ_unify}:
\begin{equation}
\begin{split}
&1).u=-\log(\frac{1}{\delta}c_{A}(-A)e^{A}+\int_{-A}^{t}c_{A}(t_{1})e^{-t_{1}}dt_{1}),
\\&
2).s=\frac{\int_{-A}^{t}(\frac{1}{\delta}c_{A}(-A)e^{A}+\int_{-A}^{t_{2}}c_{A}
(t_{1})e^{-t_{1}}dt_{1})dt_{2}+\frac{1}{\delta^{2}}c_{A}(-A)e^{A}}
{\frac{1}{\delta}c_{A}(-A)e^{A}+\int_{-A}^{t}c_{A}(t_{1})e^{-t_{1}}dt_{1}},
\end{split}
\end{equation}

One can check that $s\in C^{\infty}((-A,+\infty)),$ and $u\in
C^{\infty}((-A,+\infty))$ with
$\lim_{t\to+\infty}u(t)=-\log(\frac{1}{\delta}c_{A}(-A)e^{A}+\int_{-A}^{+\infty}c_{A}(t_{1})e^{-t_{1}}dt_{1})$.

It follows from $su''-s''=-s'u'$ and $u'<0$ that $u''s-s''>0$ is
equivalent to $s'>0$. It's easy to see that the inequality
\ref{equ:c_A_delta} is just $s'>0$. Therefore $u''s-s''>0$.

In conclusion, we have proved Theorem \ref{t:guan-zhou-semicontinu2}
with the constant $\mathbf{C}=1$.

\begin{Remark}\label{r:guan-zhou-unify-exa2}
Both $\mathbf{C}$ and the power of $\delta$ in Theorems
\ref{t:guan-zhou-semicontinu2} and \ref{t:guan-zhou-semicontinu} are
optimal on the ball $\mathbb{B}^{m}(0,e^{\frac{A}{2m}})$ with
trivial holomorphic line bundle when $S=\{0\}$.
\end{Remark}

\subsection{A singular metric version of Theorem \ref{t:guan-zhou-semicontinu2}}
$\\$

In this subsection, we formulate and prove the following singular
metric version of Theorem \ref{t:guan-zhou-semicontinu2}:

\begin{Theorem}\label{t:guan-zhou-semicontinu}
Let $(M,S)$ satisfy condition $(ab)$, $h$ be a singular metric on a
holomorphic line bundle $L$ on $M$, which is locally integrable on
$M$. Then, for any function $\Psi$ on $ M$ such that $\Psi\in
\Delta_{A,h,\delta}(S)$, there exists a uniform constant
$\mathbf{C}=1$, which is optimal, such that, for any holomorphic
section $f$ of $K_{M}\otimes L|_{S}$ on $S$ satisfying $L^2$
integrable condition \ref{equ:condition}, there exists a holomorphic
section $F$ of $K_{M}\otimes L$ on $M$ satisfying $F = f$ on $ S$
and optimal estimate \ref{equ:optimal_delta}.
\end{Theorem}

\begin{proof}
By Remark \ref{r:extend}, it suffices to prove the case that $M$ is a Stein manifold.

By Lemma \ref{l:lim_unbounded} and Lemma \ref{l:c_A},
it suffices to prove the case that $c_{A}$ is smooth on $(A,+\infty)$ and continuous on $(A,+\infty]$,
such that $\lim_{t\to +\infty}c_{A}(t)>0$.

Since $M$ is a Stein manifold, we can find a sequence of Stein manifolds $\{D_m\}_{m=1}^\infty$
satisfying $D_m\subset\subset D_{m+1}$ for all $m$ and
$\overset{\infty}{\underset{m=1}{\cup}}D_m=M$.

As $\varphi+\Psi$ and $\varphi+(1+\delta)\Psi$ are plurisubharmonic functions on $M$,
then by Lemma \ref{l:FN1}, we have smooth functions $\varphi_{k}$ and $\Psi_{k}$ on $M$,
such that $\varphi_{k}+\Psi_{k}$ and $\varphi_{k}+(1+\delta)\Psi_{k}$ are plurisubharmonic functions on $M$,
which are deceasing convergent to $\varphi+\Psi$ and $\varphi+(1+\delta)\Psi$ respectively.

Since $M$ is Stein, there is a holomorphic section $\tilde{F}$ of
$K_{M}$ on $M$ such that $\tilde{F}|_{S}={f}$. Let $ds_{M}^{2}$ be a
K\"{a}hler metric on $M$ and $dV_{M}$ be the volume form with
respect to $ds_{M}^{2}$.

Let
$\{v_{t_0,\varepsilon}\}_{t_{0}\in\mathbb{R},\varepsilon\in(0,\frac{1}{4})}$
be a family of smooth increasing convex functions on $\mathbb{R}$,
such that:

 $1).$ $v_{t_{0},\varepsilon}(t)=t$ for $t\geq-t_{0}-\varepsilon$, $v_{t_{0},\varepsilon}(t)$ is a constant
 for $t<-t_{0}-1+\varepsilon$ depending
 on $t_{0},\varepsilon$;

 $2).$ $v''_{t_0,\varepsilon}(t)$ is pointwise convergent to $\mathbb{I}_{\{-t_{0}-1< t<-t_{0}\}}$
 when $\varepsilon\to 0$, and
 $0\leq v''_{t_0,\varepsilon}(t)\leq 2$ for any $t\in \mathbb{R}$;

 $3).$ $v_{t_0,\varepsilon}(t)$ is $C^{1}$ convergent to
 $b_{t_{0}}(t)$ ($b_{t_{0}}(t):=
 \int_{-\infty}^{t}(\int_{-\infty}^{t_{2}}\mathbb{I}_{\{-t_{0}-1< t_{1}<-t_{0}\}}dt_{1})dt_{2}-\int_{-\infty}^{0}
 (\int_{-\infty}^{t_{2}}\mathbb{I}_{\{-t_{0}-1< t_{1}<-t_{0}\}}dt_{1})dt_{2}$ is also a
 $C^1$ function on $\mathbb{R}$) when
 $\varepsilon\to 0$, and $0\leq v'_{t_0,\varepsilon}(t)\leq1$ for any $t\in \mathbb{R}$.

As before, let $\eta=s(-v_{t_{0},\varepsilon}\circ\Psi_{k})$ and
$\phi=u(-v_{t_{0},\varepsilon}\circ\Psi_{k})$, where $s\in
C^{\infty}((-A,+\infty))$ satisfies $s\geq \frac{1}{\delta}$, and
$u\in C^{\infty}((-A,+\infty))\cap C^{\infty}([-A,+\infty))$
satisfies
$\lim_{t\to+\infty}u(t)=-\log(\frac{1}{\delta}c_{A}(-A)e^{A}+\int_{-A}^{\infty}c_{A}(t)e^{-t}dt)$,
such that $u''s-s''>0$ and $s'-u's=1$. Let
$\tilde{h}=e^{-\varphi_{k}-\Psi_{k}-\phi}$.

Now let $\alpha\in \mathcal{D}(X,\Lambda^{n,1}T_{D_{m}}^{*})$ be a
smooth $(n,1)$- form with compact support on $D_{m}$. Using Lemma
\ref{l:vector} and Lemma \ref{l:positve}, the inequality $s\geq
\frac{1}{\delta}$ and the fact that $\varphi_{k}+\Psi_{k}$ is
plurisubharmonic on $D_{m}$, we get

\begin{equation}
\label{equ:10.1}
\begin{split}
&\|(\eta+g^{-1})^{\frac{1}{2}}D''^{*}\alpha\|^{2}_{D_{m},\tilde{h}}
+\|\eta^{\frac{1}{2}}D''\alpha\|^{2}_{D_{m},\tilde{h}}
\\&\geq\ll[\eta\sqrt{-1}\Theta_{\tilde{h}}-\sqrt{-1}\partial\bar\partial\eta-
\sqrt{-1}g\partial\eta\wedge\bar\partial\eta,\Lambda_{\omega}]\alpha,\alpha\gg_{D_{m},\tilde{h}}
\\&\geq\ll[\eta\sqrt{-1}\partial\bar\partial\phi+\frac{1}
{\delta}\sqrt{-1}\partial\bar\partial(\varphi_{k}+\Psi_{k})-\sqrt{-1}\partial\bar\partial\eta-
\sqrt{-1}g\partial\eta\wedge\bar\partial\eta,\Lambda_{\omega}]\alpha,\alpha\gg_{D_{m},\tilde{h}}.
\end{split}
\end{equation}
where $g$ is a positive continuous function on $D_{m}$.
We need some calculations to determine $g$.

We have
\begin{equation}
\label{}
\begin{split}
&\partial\bar{\partial}\eta=
-s'(-v_{t_0,\varepsilon}\circ \Psi_{k})\partial\bar{\partial}(v_{t_0,\varepsilon}\circ \Psi_{k})
+s''(-v_{t_0,\varepsilon}\circ \Psi_{k})\partial(v_{t_0,\varepsilon}\circ \Psi_{k})\wedge
\bar{\partial}(v_{t_0,\varepsilon}\circ \Psi_{k}),
\end{split}
\end{equation}

and
\begin{equation}
\label{}
\begin{split}
&\partial\bar{\partial}\phi=
-u'(-v_{t_0,\varepsilon}\circ \Psi_{k})\partial\bar{\partial}v_{t_0,\varepsilon}\circ \Psi_{k}
+
u''(-v_{t_0,\varepsilon}\circ \Psi_{k})\partial(v_{t_0,\varepsilon}\circ \Psi_{k})
\wedge\bar{\partial}(v_{t_0,\varepsilon}\circ \Psi_{k}).
\end{split}
\end{equation}

Therefore
\begin{equation}
\label{equ:vector1}
\begin{split}
&\eta\sqrt{-1}\partial\bar\partial\phi-\sqrt{-1}\partial\bar\partial\eta-
\sqrt{-1}g\partial\eta\wedge\bar\partial\eta
\\=&(s'-su')\sqrt{-1}\partial\bar{\partial}(v_{t_0,\varepsilon}\circ \Psi_{k})
+((u''s-s'')-gs'^{2})\sqrt{-1}\partial(v_{t_0,\varepsilon}\circ \Psi_{k})\wedge\bar{\partial}
(v_{t_0,\varepsilon}\circ \Psi_{k})
\\=&
(s'-su')((v'_{t_0,\varepsilon}\circ\Psi_{k})\sqrt{-1}\partial\bar{\partial}\Psi_{k}+
(v''_{t_0,\varepsilon}\circ \Psi_{k})\sqrt{-1}\partial(\Psi_{k})
\wedge\bar{\partial}(\Psi_{k}))
\\+&((u''s-s'')-gs'^{2})\sqrt{-1}\partial(v_{t_0,\varepsilon}\circ \Psi_{k})
\wedge\bar{\partial}(v_{t_0,\varepsilon}\circ \Psi_{k}).
\end{split}
\end{equation}
We omit composite item $(-v_{t_0,\varepsilon}\circ \Psi_{k})$ after $s'-su'$ and $(u''s-s'')-gs'^{2}$
in the above equalities.

Let $g=\frac{u''s-s''}{s'^{2}}\circ(-v_{t_0,\varepsilon}\circ \Psi_{k})$.
We have $\eta+g^{-1}=(s+\frac{s'^{2}}{u''s-s''})\circ(-v_{t_0,\varepsilon}\circ \Psi_{k})$.

Since $\varphi_{k}+\Psi_{k}$ and $\varphi_{k}+(1+\delta)\Psi_{k}$ are plurisubharmonic on $M$ and
$0\leq v'_{t_{0},\varepsilon}\circ\Psi_{k}\leq1$,
we have
\begin{equation}
(1-v'_{t_0,\varepsilon}\circ\Psi_{k})\sqrt{-1}\partial\bar\partial(\varphi_{k}+\Psi_{k})+
(v'_{t_0,\varepsilon}\circ\Psi_{k})\sqrt{-1}\partial\bar\partial(\varphi_{k}+(1+\delta)\Psi_{k})\geq 0,
\end{equation}
on $M\setminus S$, which means that
\begin{equation}
\label{equ:semi.vector2}
\frac{1}{\delta}\sqrt{-1}\partial\bar\partial(\varphi_{k}+\Psi_{k})+(v'_{t_0,\varepsilon}
\circ\Psi_{k})\partial\bar{\partial}\Psi_{k}\geq 0,
\end{equation}
on $M$.

As $v'_{t_0,\varepsilon}\geq 0$  and $s'-su'=1$, using Lemma
\ref{l:positve}, equality \ref{equ:vector1}, and inequalities
\ref{equ:10.1} and \ref{equ:semi.vector2}, we have
\begin{equation}
\label{equ:semi.vector3}
\begin{split}
\langle B\alpha, \alpha\rangle_{\tilde{h}}=&\langle[\eta\sqrt{-1}\Theta_{\tilde{h}}-\sqrt{-1}\partial\bar\partial
\eta-\sqrt{-1}g\partial\eta\wedge\bar\partial\eta,\Lambda_{\omega}]
\alpha,\alpha\rangle_{\tilde{h}}
\\\geq&
\langle[(v''_{t_0,\varepsilon}\circ \Psi_{k})
\sqrt{-1}\partial\Psi_{k}\wedge\bar{\partial}\Psi_{k},\Lambda_{\omega}]\alpha,\alpha\rangle_{\tilde{h}}
\\=&\langle (v''_{t_{0},\varepsilon}\circ \Psi_{k}) \bar\partial\Psi_{k}\wedge
(\alpha\llcorner(\bar\partial\Psi_{k})^\sharp\big ),\alpha\rangle_{\tilde{h}}.
\end{split}
\end{equation}

Using the definition of contraction, Cauchy-Schwarz inequality and
the inequality \ref{equ:semi.vector3}, we have
\begin{equation}
\label{}
\begin{split}
|\langle (v''_{t_{0},\varepsilon}\circ \Psi)\bar\partial\Psi\wedge \gamma,\tilde{\alpha}\rangle_{\tilde{h}}|^{2}
=&|\langle (v''_{t_{0},\varepsilon}\circ \Psi) \gamma,\tilde{\alpha}\llcorner(\bar\partial\Psi)^\sharp\big
\rangle_{\tilde{h}}|^{2}
\\\leq&\langle( v''_{t_{0},\varepsilon}\circ \Psi) \gamma,\gamma\rangle_{\tilde{h}}
(v''_{t_{0},\varepsilon}\circ \Psi)|\tilde{\alpha}\llcorner(\bar\partial\Psi)^\sharp\big|_{\tilde{h}}^{2}
\\=&\langle (v''_{t_{0},\varepsilon}\circ \Psi) \gamma,\gamma\rangle_{\tilde{h}}
\langle (v''_{t_{0},\varepsilon}\circ \Psi) \bar\partial\Psi\wedge
(\tilde{\alpha}\llcorner(\bar\partial\Psi)^\sharp\big ),\tilde{\alpha}\rangle_{\tilde{h}}
\\\leq&\langle (v''_{t_{0},\varepsilon}\circ \Psi )\gamma,\gamma\rangle_{\tilde{h}}
\langle B\tilde{\alpha},\tilde{\alpha}\rangle_{\tilde{h}},
\end{split}
\end{equation}
for any $(n,q)$ form $\gamma$ and $(n,q+1)$ form $\tilde{\alpha}$ with values in $E$.

Take $\lambda=\bar{\partial}[(1-v'_{t_0,\varepsilon}(\Psi)){\tilde{F}}]$, $\gamma=\tilde{F}$, and
$\tilde{\alpha}=B^{-1}\bar\partial\Psi\wedge \tilde{F}$,
we have
$$\langle B^{-1}\lambda,\lambda\rangle_{\tilde{h}} \leq (v''_{t_0,\varepsilon}\circ{\Psi})| \tilde{F}|^2_{\tilde{h}},$$
 then it is easy to see that
 $$\int_{D_m\setminus S}\langle B^{-1}\lambda,\lambda\rangle_{\tilde{h}} dV_{M}
 \leq \int_{D_m\setminus S}(v''_{t_0,\varepsilon}\circ{\Psi})| \tilde{F}|^2_{\tilde{h}}dV_{M}.$$

From Lemma \ref{l:vector7}, it follows that there exists an
$(n,0)$-form $\gamma_{m,t_0,\varepsilon,k}$ on $D_{m}$ satisfying
$\bar{\partial}\gamma_{m,t_0,\varepsilon,k}=\lambda$ and

\begin{equation}
 \label{equ:semi.vector3.2}
 \begin{split}
 &\int_{ D_m}|\gamma_{m,t_0,\varepsilon,k}|^{2}_{\tilde{h}}(\eta+g^{-1})^{-1}dV_{M}
  \leq\int_{D_m}(v''_{t_0,\varepsilon,k}\circ{\Psi_{k}})| \tilde{F}|^2_{\tilde{h}}dV_M.
  \end{split}
\end{equation}

Let $\mu_{1}=e^{v_{t_0,\varepsilon}\circ\Psi_{k}}$, $\mu=\mu_{1}c_{A}(-v_{t_0,\varepsilon}\circ\Psi_{k})e^{\phi}$.

Claim that we can choose $\eta$ and $\phi$ satisfying $\mu\leq
\mathbf{C}(\eta+g^{-1})^{-1}$, which will be discussed at the end of
this subsection, where $\mathbf{C}$ is just the constant in Theorem
\ref{t:guan-zhou-semicontinu}.

Let $F_{m,t_0,\varepsilon,k}:=(1-v'_{t_0,\varepsilon}\circ\Psi_{k})\widetilde{F}-\gamma_{m,t_0,\varepsilon,k}$.
Then inequality \ref{equ:semi.vector3.2} means that

\begin{equation}
 \label{equ:semi.3.30}
 \begin{split}
 &\int_{ D_m}|F_{m,t_0,\varepsilon,k}-(1-v'_{t_0,\varepsilon}\circ\Psi_{k})\widetilde{F}|^{2}e^{-\varphi_{k}-
 \Psi_{k}+v_{t_0,\varepsilon}\circ\Psi_{k}}c_{A}(-v_{t_0,\varepsilon}\circ\Psi_{k})dV_{M}
  \\&\leq\int_{D_m}(v''_{t_0,\varepsilon}\circ{\Psi_{k}})| \tilde{F}|^2_{\tilde{h}}dV_M.
  \end{split}
\end{equation}

Note that for any compact subset $K$ of $D_{m}$, we obtain
\begin{equation}
 \label{equ:semi.3.30.2}
 \begin{split}
 &(\int_{K}|F_{m,t_0,\varepsilon,k}-(1-v'_{t_0,\varepsilon}\circ\Psi_{k})\widetilde{F}|^{2}
 e^{-\varphi_{k}-\Psi_{k}+v_{t_0,\varepsilon}\circ\Psi_{k}}c_{A}(-v_{t_0,\varepsilon}\circ\Psi_{k})dV_{M})^{1/2}
 \\&+(\int_{K}|(v'_{t_0,\varepsilon}\circ\Psi_{k})\widetilde{F}|^{2}
 e^{-\varphi_{k}-\Psi_{k}+v_{t_0,\varepsilon}\circ\Psi_{k}}c_{A}(-v_{t_0,\varepsilon}\circ\Psi_{k})dV_{M})^{1/2}
  \\&\geq(\int_{K}|F_{m,t_0,\varepsilon,k}-\widetilde{F}|^{2}
 e^{-\varphi_{k}-\Psi_{k}+v_{t_0,\varepsilon}\circ\Psi_{k}}c_{A}(-v_{t_0,\varepsilon}\circ\Psi_{k})dV_{M})^{1/2},
  \end{split}
\end{equation}

Note that:

1), $e^{-\varphi_{k}-\Psi_{k}}$,
$e^{v_{t_{0},\varepsilon}\circ\Psi_{k}}$ and
$c_{A}(-v_{t_{0},\varepsilon}\circ\Psi_{k})$ have uniform positive
lower bounds independent of $k$;

2),
$|v'_{t_{0},\varepsilon}\circ\Psi_{k})\widetilde{F}|^{2}e^{-\Psi}$
and $\int_{D_m}(v''_{t_0,\varepsilon}\circ{\Psi_{k}})|
\tilde{F}|^2_{\tilde{h}}dV_M$ have uniform positive upper bounds
independent of $k$;

3), $e^{-\varphi}$ is locally integrable on $M$ and the sequence
$\varphi_{k}+\Psi_{k}$ is decreasing with respect to $k$.

According to inequality \ref{equ:semi.3.30.2}, it follows that
$\int_{K}|F_{m,t_0,\varepsilon,k}-\widetilde{F}|^{2}dV_{M}$ has a
uniform bound independent of $k$ for any compact subset $K$ of
$D_{m}$.

Using Lemma \ref{l:uniform_converg_compact}, we have a subsequence
of $\{F_{m,t_0,\varepsilon,k}\}_{k}$, still denoted by
$\{F_{m,t_0,\varepsilon,k}\}_{k}$, which is uniformly convergent to
a holomorphic $(n,0)$ form $F_{m,t_0,\varepsilon}$ on any compact
subset of $D_{m}$.

As all terms $e^{v_{t_{0},\varepsilon}\circ\Psi_{k}}$,
$c_{A}(-v_{t_{0},\varepsilon}\circ\Psi_{k})$,
$(1-v'_{t_{0},\varepsilon}\circ\Psi_{k})\widetilde{F}$, and
$(v''_{t_0,\varepsilon}\circ\Psi_{k})|
\tilde{F}|^2e^{-\varphi_{k}-\Psi_{k}-\phi}$ have uniform positive
upper bounds independent of $k$, and
$v_{t_0,\varepsilon}(\Psi_{k})\geq\Psi_{k}$, it follows from the
dominated convergence theorem that
\begin{equation}
 \label{equ:semi.3.30.3}
 \begin{split}
 &\int_{K}|F_{m,t_0,\varepsilon}-(1-v'_{t_0,\varepsilon}\circ\Psi)\widetilde{F}|^{2}
 e^{-\varphi_{k}-\Psi_{k}+v_{t_0,\varepsilon}\circ\Psi}c_{A}(-v_{t_0,\varepsilon}\circ\Psi)dV_{M}
 \\&\leq
 \int_{D_m}(v''_{t_0,\varepsilon}\circ{\Psi})| \tilde{F}|^2e^{-\varphi-\Psi-u(-v_{t_{0},\varepsilon}(\Psi))}dV_M,
 \end{split}
\end{equation}
 for any compact subset $K$ of $D_{m}$.

As the sequence $\varphi_{k}+\Psi_{k}$ is decreasing convergent to
$\varphi+\Psi$, it follows from Levi's theorem that
\begin{equation}
 \label{equ:semi.3.30.4}
 \begin{split}
 &\int_{K}|F_{m,t_0,\varepsilon}-(1-v'_{t_0,\varepsilon}\circ\Psi)\widetilde{F}|^{2}
 e^{-\varphi-\Psi+v_{t_0,\varepsilon}\circ\Psi}c_{A}(-v_{t_0,\varepsilon}\circ\Psi)dV_{M}
 \\&\leq
 \int_{D_m}(v''_{t_0,\varepsilon}\circ{\Psi})| \tilde{F}|^2e^{-\varphi-\Psi-u(-v_{t_{0},\varepsilon}(\Psi))}dV_M.
 \end{split}
\end{equation}
for any compact subset $K$ of $D_{m}$, which means
\begin{equation}
 \label{equ:semi.3.30.5}
 \begin{split}
 &\int_{D_{m}}|F_{m,t_0,\varepsilon}-(1-v'_{t_0,\varepsilon}\circ\Psi)\widetilde{F}|^{2}
 e^{-\varphi-\Psi+v_{t_0,\varepsilon}\circ\Psi}c_{A}(-v_{t_0,\varepsilon}\circ\Psi)dV_{M}
 \\&\leq
 \int_{D_m}(v''_{t_0,\varepsilon}\circ{\Psi})| \tilde{F}|^2e^{-\varphi-\Psi-u(-v_{t_{0},\varepsilon}(\Psi))}dV_M.
 \end{split}
\end{equation}

Note that $e^{-\Psi}$ is not integrable along $S$, and
$F_{m,t_0,\varepsilon}$ and
$(1-v'_{t_0,\varepsilon}\circ\Psi)\widetilde{F}$ are both
holomorphic near $S$.

Then
$F_{m,t_0,\varepsilon}-(1-v'_{t_0,\varepsilon}\circ\Psi)\widetilde{F}|_{S}=0$,
therefore $F_{m,t_0,\varepsilon}|_{S}=\widetilde{F}|_{S}$. It is
clear that $F_{m,t_0,\varepsilon}$ is an extension of $f$.

Note that $v_{t_0,\varepsilon}(\Psi)\geq\Psi$. Then the inequality
\ref{equ:semi.3.30.5} becomes

\begin{equation}
\label{equ:semi.vector3.31.1}
\begin{split}
 &\int_{D_{m}}|F_{m,t_0,\varepsilon}-(1-v'_{t_0,\varepsilon}\circ\Psi)\widetilde{F}|^{2}
 e^{-\varphi}c_{A}(-v_{t_0,\varepsilon}\circ\Psi)dV_{M}
 \\&\leq
 \int_{D_m}(v''_{t_0,\varepsilon}\circ{\Psi})| \tilde{F}|^2e^{-\varphi-\Psi-u(-v_{t_{0},\varepsilon}(\Psi))}dV_M
 \\&\leq\frac{1}{e^{A_{t_0}}}(v''_{t_0,\varepsilon}\circ{\Psi})| \tilde{F}|^2e^{-\varphi-\Psi}dV_{M},
\end{split}
\end{equation}
where $A_{t_0}:=\inf_{t\geq t_0}\{u(t)\}$.

As
$$\lim_{t\to+\infty}u(t)=-\log(\frac{1}{\delta}c_{A}(-A)+\int_{-A}^{+\infty}c_{A}(t)e^{-t}dt),$$
it is clear that
$$\lim_{t_{0}\to\infty}\frac{1}{e^{A_{t_0}}}=\frac{1}{\delta}c_{A}(-A)+\int_{-A}^{+\infty}c_{A}(t)e^{-t}dt.$$

Given $t_0$ and $D_{m}$,
$$(v''_{t_0,\varepsilon}\circ\Psi)|\tilde{F}|^{2}e^{-\varphi-\Psi}$$
has a uniform bound on $D_{m}$ independent of $\varepsilon$. Then
both
$$\int_{D_m}|(1-v'_{t_0,\varepsilon}\circ\Psi)\tilde{F}|^{2}e^{-\varphi}c_{A}(-v_{t_0,\varepsilon}\circ\Psi)dV_{M}$$
and
$$\int_{D_m}v''_{t_0,\varepsilon}\circ\Psi|\tilde{F}|^{2}e^{-\varphi-\Psi}dV_{M}$$
have uniform bounds independent of $\varepsilon$ for any given $t_0$
and $D_m$.

Using the equation $\bar\partial F_{m,t_0,\varepsilon}=0$ and Lemma
\ref{l:uniform_converg_compact}, we can choose a subsequence of
$\{F_{m,t_0,\varepsilon}\}_{\varepsilon}$, such that the chosen
sequence is uniformly convergent on any compact subset of $D_m$,
still denoted by $\{F_{m,t_0,\varepsilon}\}_{\varepsilon}$ without
ambiguity.

For any compact subset $K$ on $D_m$, all terms
$F_{m,t_0,\varepsilon}$,
$(1-v'_{t_0,\varepsilon}\circ\Psi)\tilde{F}$,
$c_{A}(-v_{t_0,\varepsilon}\circ\Psi)$ and
$(v''_{t_0,\varepsilon}\circ\Psi)|\tilde{F}|^{2}e^{-\varphi-\Psi}$
have uniform bounds on $K$ independent of $\varepsilon$.

Using the dominated convergence theorem on any compact subset $K$ of
$D_m$ and inequality \ref{equ:semi.vector3.31.1}, we have
\begin{equation}
\begin{split}
&\int_{K}|F_{m,t_0}-(1-b'_{t_0}(\Psi))\tilde{F}|^{2}e^{-\varphi}c_{A}(-b_{t_0}(\Psi))dV_{M}
\\\leq&\frac{\mathbf{C}}{e^{A_{t_0}}}\int_{D_m}(\mathbb{I}_{\{-t_{0}-1< t<-t_{0}\}}
\circ\Psi)|\tilde{F}|^{2}e^{-\varphi-\Psi}dV_{M},
\end{split}
\end{equation}
which implies
\begin{equation}
\label{equ:vector3.4}
\begin{split}
&\int_{ D_m}|F_{m,t_0}-(1-b'_{t_0}(\Psi))\tilde{F}|^{2}e^{-\varphi}c_{A}(-b_{t_0}(\Psi))dV_{M}
\\\leq&\frac{\mathbf{C}}{e^{A_{t_0}}}\int_{D_m}(\mathbb{I}_{\{-t_{0}-1< t<-t_{0}\}}
\circ\Psi)|\tilde{F}|^{2}e^{-\varphi-\Psi}dV_{M}.
\end{split}
\end{equation}

According to the definition of $dV_{M}[\Psi]$ and the assumption
$\sum_{k=1}^{n}\frac{\pi^{k}}{k!}\int_{S_{n-k}}|f|^{2}_{h}dV_{M}[\Psi]<\infty$,
it follows that

\begin{equation}
\label{equ:vector3.6}
\begin{split}
&\limsup_{t_{0}\to+\infty}\int_{D_m}(\mathbb{I}_{\{-t_{0}-1< t<-t_{0}\}}
\circ\Psi)|\tilde{F}|^{2}e^{-\varphi-\Psi}dV_{M}
\\\leq&
\limsup_{t_{0}\to+\infty}\int_{M}\mathbb{I}_{\overline{D}_{m}}(\mathbb{I}_{\{-t_{0}-1<t<-t_{0}\}}
\circ\Psi)|\tilde{F}|^{2}e^{-\varphi-\Psi}dV_{M}
\\\leq&\sum_{k=1}^{n}\frac{\pi^{k}}{k!}\int_{S_{n-k}}\mathbb{I}_{\overline{D}}|f|^{2}_{h}dV_{M}[\Psi]
\leq\sum_{k=1}^{n}\frac{\pi^{k}}{k!}\int_{S_{n-k}}|f|^{2}_{h}dV_{M}[\Psi]<\infty
\end{split}
\end{equation}

Then
$$\int_{D_m}(\mathbb{I}_{\{-t_{0}-1< t<-t_{0}\}}\circ\Psi)|\tilde{F}|^{2}e^{-\varphi-\Psi}dV_{M}$$
has a uniform bound independent of $t_{0}$ for any given $D_m$, and
\begin{equation}
\label{equ:vector3.7}
\begin{split}
&\limsup_{t_{0}\to+\infty}\int_{D_m}(\mathbb{I}_{\{-t_{0}-1< t<-t_{0}\}}
\circ\Psi)|\tilde{F}|^{2}e^{-\varphi-\Psi}dV_{M}
\\\leq&\sum_{k=1}^{n}\frac{\pi^{k}}{k!}\int_{S_{n-k}}|f|^{2}e^{-\varphi}dV_{M}[\Psi]<\infty.
\end{split}
\end{equation}

Therefore
$$\int_{ D_m}|F_{m,t_0}-(1-b'_{t_0}(\Psi))\tilde{F}|^{2}e^{-\varphi}c_{A}(-b_{t_0}(\Psi))dV_{M}$$
has a uniform bound independent of $t_{0}$ for any given $D_m$.

Since
$$\int_{ D_m}|(1-b'_{t_0}(\Psi))\tilde{F}|^{2}e^{-\varphi}c_{A}(-b_{t_0}(\Psi))dV_{M}$$
has a uniform bound independent of $t_{0}$,  and

\begin{equation}
\label{equ:vector3.9}
\begin{split}
&(\int_{ D_m}|F_{m,t_0}-(1-b'_{t_0}(\Psi))\tilde{F}|^{2}e^{-\varphi}c_{A}(-b_{t_0}(\Psi))dV_{M})^{\frac{1}{2}}
\\+&(\int_{ D_m}|(1-b'_{t_0}(\Psi))\tilde{F}|^{2}e^{-\varphi}c_{A}(-b_{t_0}(\Psi))dV_{M})^{\frac{1}{2}}
\\\geq&
(\int_{ D_m}|F_{m,t_0}|^{2}e^{-\varphi}c_{A}(-b_{t_0}(\Psi))dV_{M})^{\frac{1}{2}},
\end{split}
\end{equation}
it follows from inequality \ref{equ:vector3.4} that
$\int_{D_m}|F_{m,t_0}|^{2}e^{-\varphi}c_{A}(-b_{t_0}(\Psi))dV_{M}$
has a uniform bound independent of $t_{0}$.

Using the equation $\bar\partial F_{m,t_{0}}=0$ and Lemma
\ref{l:uniform_converg_compact}, we can choose a subsequence of
$\{F_{m,t_{0}}\}_{t_{0}}$, such that the chosen sequence is
uniformly convergent on any compact subset of $D_m$, still denoted
by $\{F_{m,t_0}\}_{t_{0}}$ without ambiguity.

For any compact subset $K$ on $D_m$, both $F_{m,t_0}$ and
$(1-b'_{t_0}\circ\Psi)\tilde{F}$ have uniform bounds on $K$
independent of $t_0$.

Using inequalities \ref{equ:vector3.4} and \ref{equ:vector3.7}, the
following equality
$$\lim_{t_{0}\to\infty}\frac{1}{e^{A_{t_0}}}=\int_{-A}^{+\infty}c_{A}(t)e^{-t}dt,$$
and the dominated convergence theorem on any compact subset $K$ of
$D_m$, we have
\begin{equation}
\begin{split}
&\int_{D_m}\mathbb{I}_{K}|F_{m}|^{2}e^{-\varphi}c_{A}(-\Psi)dV_{M}
\\&\leq\mathbf{C}(\int_{-A}^{+\infty}c_{A}(t)e^{-t}dt)\sum_{k=1}^{n}\frac{\pi^{k}}{k!}
\int_{S_{n-k}}|f|^{2}e^{-\varphi}dV_{M}[\Psi],
\end{split}
\end{equation}
which implies
\begin{equation}
\begin{split}
&\int_{ D_m}|F_{m}|^{2}e^{-\varphi}c_{A}(-\Psi)dV_{M}
\\&\leq\mathbf{C}(\int_{-A}^{+\infty}c_{A}(t)e^{-t}dt)\sum_{k=1}^{n}
\frac{\pi^{k}}{k!}\int_{S_{n-k}}|f|^{2}e^{-\varphi}dV_{M}[\Psi],
\end{split}
\end{equation}
where the Lebesgue measure of $\{\Psi=-\infty\}$ is zero.

Define $F_m=0$ on $M\backslash D_m$. Then the weak limit of some
weakly convergent subsequence of $\{F_m\}_{m=1}^\infty$ gives a
holomorphic section $F$ of $K_{M}\otimes E$ on $M$ satisfying
$F|_{S}=\tilde{F}|_{S}$, and
\begin{equation}
\begin{split}
&\int_{ M}|F|^{2}e^{-\varphi}c_{A}(-\Psi)dV_{M}
\\&\leq\mathbf{C}(\frac{1}{\delta}c_{A}(-A)e^{A}+
\int_{-A}^{+\infty}c_{A}(t)e^{-t}dt)\sum_{k=1}^{n}\frac{\pi^{k}}{k!}\int_{S_{n-k}}|f|^{2}e^{-\varphi}dV_{M}[\Psi].
\end{split}
\end{equation}

To finish the proof of Theorem \ref{t:guan-zhou-semicontinu}, it
suffices to determine $\eta$ and $\phi$ such that $(\eta+g^{-1})\leq
\mathbf{C}c^{-1}_{A}(-v_{t_0,\varepsilon}\circ\Psi)e^{-v_{t_0,\varepsilon}\circ\Psi}e^{-\phi}=\mathbf{C}\mu^{-1}$
on $D_v$.

As $\eta=s(-v_{t_0,\varepsilon}\circ\Psi)$ and
$\phi=u(-v_{t_0,\varepsilon}\circ\Psi)$, we have $(\eta+g^{-1})
e^{v_{t_0,\varepsilon}\circ\Psi}e^{\phi}=(s+\frac{s'^{2}}{u''s-s''})e^{-t}e^{u}\circ(-v_{t_0,\varepsilon}\circ\Psi)$.

We naturally obtain the system of ODEs \ref{equ:unify2.GZ_unify},
where $t\in(-A,+\infty)$, $\mathbf{C}=1$, $s\in
C^{\infty}((-A,+\infty))$ satisfying $s\geq \frac{1}{\delta}$, $u\in
C^{\infty}((-A,+\infty))$ satisfying
$\lim_{t\to+\infty}u(t)=-\log(\frac{1}{\delta}c_{A}(-A)e^{A}+\int_{-A}^{\infty}c_{A}(t)e^{-t}dt)$,
and $u''s-s''>0$.

We solve the system of ODEs \ref{equ:unify2.GZ_unify} in subsection
\ref{subsec:ODE} and get the solution
\begin{equation}
\begin{split}
&1).u=-\log(\frac{1}{\delta}c_{A}(-A)e^{A}+\int_{-A}^{t}c_{A}(t_{1})e^{-t_{1}}dt_{1}),
\\&
2).s=\frac{\int_{-A}^{t}(\frac{1}{\delta}c_{A}(-A)e^{A}+\int_{-A}^{t_{2}}c_{A}(t_{1})e^{-t_{1}}dt_{1})dt_{2}
+\frac{1}{\delta^{2}}c_{A}(-A)e^{A}}
{\frac{1}{\delta}c_{A}(-A)e^{A}+\int_{-A}^{t}c_{A}(t_{1})e^{-t_{1}}dt_{1}}.
\end{split}
\end{equation}

One can check that $s\in C^{\infty}((-A,+\infty))$,
$\lim_{t\to+\infty}u(t)=-\log(\frac{1}{\delta}c_{A}(-A)e^{A}+\int_{-A}^{+\infty}c_{A}(t_{1})e^{-t_{1}}dt_{1}),$
and $u\in C^{\infty}((-A,+\infty))$.

As $su''-s''=-s'u'$ and $u'<0$, it is clear that $u''s-s''>0$ is
equivalent to $s'>0$, and inequality \ref{equ:c_A_delta} means that
$s'>0$, then we obtain $u''s-s''>0$.

In conclusion, we have proved Theorem \ref{t:guan-zhou-semicontinu}.
\end{proof}

Using Remark \ref{r:c_A_continu} and Lemma \ref{l:c_A}, we may
replace smoothness of $c_{A}$ by continuity.

When we take $c_{A}=1$, using the above Theorem
\ref{t:guan-zhou-semicontinu2} and Theorem
\ref{t:guan-zhou-semicontinu}, one obtains main results in
\cite{guan-zhou12a} and \cite{guan-zhou12p}, which are the optimal
estimate versions of main theorems in \cite{ohsawa4,ohsawa5}.

\subsection{Proof of Theorem \ref{t:guan-zhou-unify}}
$\\$

By Remark \ref{r:extend}, it suffices to prove the case that $M$ is a Stein manifold.

By Lemma \ref{l:lim_unbounded} and Lemma \ref{l:c_A}, it is enough
to prove the case that $c_{A}$ is smooth on $(A,+\infty)$ and
continuous on $(A,+\infty]$, such that $\lim_{t\to +\infty}c_{A}(t)$
exists and bigger than $0$.

Since $M$ is a Stein manifold, we can find a sequence of Stein
manifolds $\{D_m\}_{m=1}^\infty$ satisfying $D_m\subset\subset
D_{m+1}$ for all $m$ and
$\overset{\infty}{\underset{m=1}{\cup}}D_m=M$. All $D_{m}\setminus
S$ are complete K\"{a}hler (\cite{grauert}).

As $\Psi$ is a plurisubharmonic function on $M$,
then

$(1).$ when $A<+\infty$, $\sup_{z\in D_{m}}\Psi(z)<A-\varepsilon$,
where $\varepsilon>0$.

$(2).$ when $A=+\infty$, $\sup_{z\in D_{m}}\Psi(z)<A_{m}$,
where $A_{m}<+\infty$ is sufficient large.

We just consider our proof for the condition $(1)$ (the case under
the condition $(2)$ can be proved similarly), By Lemma
\ref{l:relate_c_A_delta}, for any given $A'<A$, it follows that
there exists $c_{A''}$ and $\delta''>0$ satisfying conditions 1),
2), 3) in Lemma \ref{l:relate_c_A_delta}, where $A''<A$ and
$A''>A-\varepsilon$.

Note that $\sqrt{-1}\partial\bar\partial\Psi\geq0$, and
$\sqrt{-1}\Theta_{he^{-\Psi}}\geq0$ on $M\setminus S$ implies
conditions $1)$ and $2)$ in Theorem \ref{t:guan-zhou-semicontinu2}
for any $\delta''>0$.

Using Theorem \ref{t:guan-zhou-semicontinu2}, we obtain a
holomorphic $(n,0)$ form $F_{m,A''}$ with value in $E$ on $D_{m}$,
which satisfies $F_{m}|_{S}=f$ and
$$\int_{D_{m}}c_{A''}(-\Psi)|F_{m,A''}|^{2}_{h}
\leq\mathbf{C}\int_{-A}^{\infty}c_{A}(t)e^{-t}dt\sum_{k=1}^{n}\frac{\pi^{k}}{k!}\int_{S_{n-k}}|f|^{2}_{h}dV_{M}[\Psi].$$

Note that $c_{A''}(-\Psi)$ is uniformly convergent to $c_{A}(-\Psi)$
on any compact subset of $D_{m}$, as $A''\to A$. Let $A'\to A$
($A''\to A$), and then let $m\to+\infty$, using Lemma
\ref{l:lim_unbounded}, we prove the present theorem.

\begin{Remark}\label{r:guan-zhou-unify-exa1}
$\mathbf{C}$ is optimal on the ball $\mathbb{B}^{m}(0,e^{\frac{A}{2m}})$ for trivial
holomorphic line bundle when $S=\{0\}$, and $\Psi=2m\log|z|$.
When $A=+\infty$, $\mathbb{B}^{m}(0,e^{\frac{A}{2m}}):=\mathbb{C}^{m}$.
\end{Remark}

Using Theorem \ref{t:guan-zhou-unify} and Corollary
\ref{l:extension_equ.2} by taking $d_{2}=1$, we obtain

\begin{Corollary}
\label{c:unify.1} Let $\Omega$ be an open Riemann surface which
admits a Green function $G$, and $\Psi:=2G(z,z_{0})$. Let $V_{z_0}$
be a neighborhood of $z_0$ with a local coordinate $z$, which
satisfies $\Psi|_{V_{z_0}}\leq\Psi|_{\Omega\setminus V_{z_0}}$ and
$\Psi|_{V_{z_0}}=\log|z|^{2}$.

If there is a unique holomorphic $(1,0)$ form $F$, such that
$F|_{z_0}=dz$ and
$$\int_{\Omega}\sqrt{-1}F\wedge\bar{F}\leq\pi\int_{z_{0}}|dz|^{2}dV_{\Omega}[\Psi],$$
then we have
$F|_{V_{z_0}}=dz$.
\end{Corollary}

\subsection{Solution of the ODE system \ref{equ:unify2.GZ_unify}}
\label{subsec:ODE}
$\\$

We now solve the equations \ref{equ:unify2.GZ_unify} as follows:

By $2)$ of equation \ref{equ:unify2.GZ_unify}, it follows that
$su''-s''=-s'u'$. Then $1)$ of equation \ref{equ:unify2.GZ_unify}
can be reformulated to
$$(s-\frac{s'}{u'})e^{u-t}=\frac{\mathbf{C}}{c_{A}(t)},$$
i.e.
$$\frac{su'-s'}{u'}e^{u-t}=\frac{\mathbf{C}}{c_{A}(t)}.$$

By $2)$ of equation \ref{equ:unify2.GZ_unify} again, it follows that
$$\frac{\mathbf{C}}{c_{A}(t)}=\frac{su'-s'}{u'}e^{u-t}=\frac{-1}{u'}e^{u-t},$$
therefore
$$\frac{de^{-u}}{dt}=-u'e^{-u}=\frac{c_{A}(t)e^{-t}}{\mathbf{C}}.$$
Note that $2)$ of equation \ref{equ:unify2.GZ_unify} is equivalent to
$\frac{d(se^{-u})}{dt}=e^{-u}$.

As $s\geq 0$, we obtain the solution
\begin{displaymath}
     \begin{cases}
      u=-\log(a+\int_{-A}^{t}c_{A}(t_{1})e^{-t_{1}}dt_{1}), \\
      s=\frac{\int_{-A}^{t}(a+\int_{-A}^{t_{2}}c_{A}(t_{1})e^{-t_{1}}dt_{1})dt_{2}+b}
      {a+\int_{-A}^{t}c_{A}(t_{1})e^{-t_{1}}dt_{1}},
      \end{cases}
\end{displaymath}
when $\mathbf{C}=1$, where $a\geq 0$ and $b\geq0$.

As
$\lim_{t\to+\infty}u(t)=-\log(\frac{1}{\delta}c_{A}(-A)e^{A}+\int_{-A}^{+\infty}c_{A}(t_{1})e^{-t_{1}}dt_{1})$,
we have $a=\frac{1}{\delta}c_{A}(-A)e^{A}$. As
$s\geq\frac{1}{\delta}$, we have $\frac{b}{a}\geq\frac{1}{\delta}$.

As $u'<0$ and $su''-s''=-s'u'$, it is clear that $u''s-s''>0$ is
equivalent to $s'>0$. By inequality $s'>0$, it follows that
$a^{2}\geq c_{A}(-A)e^{A}b$. Then we get $b=\frac{1}{\delta}a$.


\subsection{Verifications of Remark \ref{r:guan-zhou-unify-exa2} and Remark \ref{r:guan-zhou-unify-exa1}}
$\\$

Let $\mathbb{B}^{m}(0,{e^{\frac{A}{2m}}})$ be the unit ball with
radius $e^{\frac{A}{2m}}$ on $\mathbb{C}^{m}$
($\mathbb{B}^{m}(0,+\infty):=\mathbb{C}^{m}$), with coordinate
$z=\{z_{1},\cdots,z_{m}\}$. Let
$$\varphi(z)=(1+\delta)m\max\{\log|z|^{2},\log|a|^{2}\},$$ and
$$\Psi(z)=-m\max\{\log|z|^{2},\log|a|^{2}\}+m\log|z|^{2}+A-\varepsilon,$$
where $a\in(0,+\infty)$ and $\varepsilon>0$.

As both $\varphi$ and $\varphi+(1+\delta)\Psi$ are plurisubharmonic,
and
$$\varphi+\Psi=\frac{\delta\varphi+(\varphi+(1+\delta)\Psi)}{1+\delta},$$ it is clear that
$\Psi(z)\in \Delta_{\varphi,\delta}(S)$, where $S=\{z=0\}$.

For any $f(0)\neq 0$, it suffices to prove
\begin{equation}
\label{equ:opt1}
\begin{split}
&\lim_{a\to 0}\frac{\min_{F\in Hol(\mathbb{B}^{m}(0,{e^{\frac{A}{2m}}}))}
\int_{\mathbb{B}^{m}(0,{e^{\frac{A}{2m}}})} |F|^{2}c_{A}
(-\Psi)e^{-\varphi}d\lambda}{a^{-2\delta}e^{\varepsilon-A}|F(0)|^{2} }
\\&=\frac{\pi^{m}}{m!}(\int_{-A+\varepsilon}^{+\infty}c_{A}(t)e^{-t}dt+
\frac{1}{\delta}c_{A}(-A+\varepsilon)e^{A-\varepsilon}),
\end{split}
\end{equation}
where $F(0)=f(0)$.

Because
$e^{-\varphi}d\lambda[\Psi]=a^{-2\delta}e^{\varepsilon}\delta_{0}$
(by Lemma \ref{l:lem9}), where $\delta_{0}$ is the dirac function at
$0$, let $\varepsilon$ go to zero, then we see that the constant of
Theorem \ref{t:guan-zhou-unify} is optimal.

Set Taylor expansion of at $0\in\mathbb{C}^{m}$ of
$F(z)=\sum_{k=0}^{\infty}a_{k}z^{k}$, where
$k=\{k_{1},\cdots,k_{m}\}$, $a_{k}$ are complex constants, and
$z^{k}=z_{1}^{k_{1}}\cdots z_{m}^{k_{m}}$.

Note that
$\int_{\Delta}z^{k_{1}}\bar{z}^{k_{2}}e^{-\varphi}d\lambda=0$ when
$k_{1}\neq k_{2}$, and
$\int_{\Delta}z^{k_{1}}\bar{z}^{k_{2}}e^{-\varphi}d\lambda>0$ when
$k_{1}= k_{2}$, it is clear that
$$\min_{F\in Hol(\mathbb{B}^{m}(0,{e^{\frac{A}{2m}}}))}\int_{\mathbb{B}^{m}
(0,{e^{\frac{A}{2m}}})}c_{A}(-\Psi)|F|^{2}e^{-\varphi}d\lambda
=\int_{\mathbb{B}^{m}(0,{e^{\frac{A}{2m}}})}c_{A}(-\Psi)|F(0)|^{2}e^{-\varphi}d\lambda.$$

It is not hard to see that
$$\int_{\mathbb{B}^{m}(0,{e^{\frac{A}{2m}}})}c_{A}(-\Psi)e^{-\varphi}d\lambda
=\frac{\pi^{m}}{m!}(a^{-2\delta}e^{-A+\varepsilon}\int_{-A+\varepsilon}^{+\infty}c_{A}(t)e^{-t}dt
+c_{A}(-A+\varepsilon)\frac{a^{-2\delta}-e^{-\delta A}}{\delta}),$$
and
$$\lim_{a\to0}\frac{a^{-2\delta}-e^{-\delta A}}{\delta a^{-2\delta}}=\frac{1}{\delta}.$$

As $\int_{-A}^{\infty}c_{A}(t)e^{-t}dt<\infty$,
$c_{A}(-A)e^{A}<\infty$ and $c_{A}(-A)e^{A}\neq0$, then we have
proved the equality \ref{equ:opt1}. Now we finishing proving Remark
\ref{r:guan-zhou-unify-exa2}.

Let $\varphi=0$ and $\Psi=m\log|z|^{2}$, then we obtain Remark
\ref{r:guan-zhou-unify-exa1} on
$\mathbb{B}^{m}(0,{e^{\frac{A}{2m}}})$, where
$A\in(-\infty,+\infty]$.


\section{Proofs of the main corollaries}

In this section, we give proofs of the main corollaries including a
conjecture of Suita on the equality conditions in Suita's conjecture
and the extended Suita conjecture, optimal estimates of various
known $L^2$ extension theorems, optimal estimate for $L^{p}$
extension and for $L^{\frac{2}{m}}$ extension, etc.

\subsection{Proof of Theorem \ref{c:suita_equ}}
$\\$

It is well-known that if $\Omega$ is conformally equivalent to the unit disc less
a (possible) closed set of inner capacity zero,
then
$$\pi B_{\Omega}(z_{0})= c^{2}_{\beta}(z_{0}).$$

It suffices to prove that if $\pi B_{\Omega}(z_{0})= c^{2}_{\beta}(z_{0})$ holds,
then $\Omega$ is conformally equivalent to the unit disc less
a (possible) closed set of inner capacity zero.

As $\Omega$ is a noncompact Riemann surface,
there exists a holomorphic function $g_{0}$ on $\Omega$,
which satisfies $dg_{0}|_{z_{0}}\neq 0$, $g_{0}|_{z_{0}}=0$,
and $g_{0}|_{\Omega\setminus \{z_{0}\}}\neq 0$.

Let $p:\Delta\to \Omega$ be the universal covering of $\Omega$. We
can choose a connected component $V_{z_{0}}$ small enough, such that
$p$ is biholomorphic on any connected component of
$p^{-1}(V_{z_{0}})$.

Since $p^{*}(G_{\Omega}(z,z_{0})-\log|g_{0}(z)|)$ is a harmonic function on
$\Delta$ (by Lemma \ref{r:green_harmonic}),
then there exists a holomorphic function $f_{1}$ on
$\Delta$,
such that the real part of $f_{1}$ is $p^{*}(G_{\Omega}(z,z_{0})-\log|w|)$.

We want to show that for any $z_{1}\in\Omega$,
$p^{*}(g_{0})\exp f_{1}$ is constant along the fibre $p^{-1}(z_{1})$.

Note that
$$\log|p_{*}((p^{*}(g_{0})\exp f_{1})|_{U^{0}})|=G_{\Omega}(z,z_{0})|_{V_{z_{0}}},$$
where $U^{0}$ is a fixed connected component of $p^{-1}(V_{z_{0}})$.
By assumption $\pi B_{\Omega}(z_{0})= c^{2}_{\beta}(z_{0})$, and by
Remark \ref{r:minimal}, there is a unique holomorphic $(1,0)$ form
$F$ on $\Omega$, which satisfies
$F|_{z_{0}}=dp_{*}((p^{*}(g_{0})\exp f_{1})|_{U^{0}})|_{z_{0}}$, and
$$\sqrt{-1}\int_{\Omega}F\wedge\bar{F}\leq\pi\int_{z_{0}}|dp_{*}
((p^{*}(g_{0})\exp f_{1})|_{U^{0}})|^{2}dV_{\Omega}[2 G_{\Omega}(z,z_{0})].$$

Using Proposition \ref{p:unique}, we have
$$dp_{*}(p^{*}(g_{0})\exp
f_{1}|_{U^{0}})=F|_{V_{z_0}},$$
therefore
$$d(p^{*}(g_{0})\exp
f_{1}|_{U^{0}})=(p^{*}F)|_{U^{0}}.$$

Using Lemma \ref{l:identity}, we have $d(p^{*}(g_{0})\exp
f_{1})=p^{*}F$.

For $z_{1}\in\Omega$,
there exists $V_{z_{1}}$, a connected neighborhood small enough,
such that $p$ is biholomorphic on any connected
component of $p^{-1}(V_{z_{1}})$,
and $U_{1}$ and $U_{2}$ are any two connected components of $p^{-1}(V_{z_{1}})$.

Denote by
$$g_{1}=(p|_{U_{1}})_{*}((p^{*}(g_{0})\exp f_{1})|_{U_{1}}),$$
and
$$g_{2}=(p|_{U_{2}})_{*}((p^{*}(g_{0})\exp f_{1})|_{U_{2}}),$$
they are holomorphic functions on $V_{z_1}$.

As $d(p^{*}(g_{0})\exp f_{1})=p^{*}F$,
therefore
$$(p|_{U_{1}})_{*}(d(p^{*}(g_{0})\exp f_{1})|_{U_{1}})=(p|_{U_{2}})_{*}((dp^{*}(g_{0})\exp f_{1})|_{U_{2}}),$$
i.e.
$$dg_{1}=dg_{2}.$$
As $|p^{*}(g_{0})\exp f_{1}|=\exp (p^{*}G_{\Omega}(\cdot,z_{0}))$,
which restricted on ${p^{-1}(z)}$ takes the same value,
we have $|g_{1}|=|g_{2}|$, which are not constant on $V_{z_1}$.

Using Lemma \ref{l:identity_function}, we have $g_{1}=g_{2}$,
therefore $(p^{*}(g_{0})\exp f_{1})|_{p^{-1}(z)}$ is constant for
any $z\in\Omega$. Then we obtain a well-defined holomorphic function
$$g(z):=(p^{*}(g_{0})\exp f_{1})|_{p^{-1}(z)}$$
on $\Omega$,
which satisfies $|g(z)|=\exp G_{\Omega}(z,z_{0})$.

Using Lemma \ref{l:c_beta.c_B}, we have
$c_{B}(z_{0})=c_{\beta}(z_{0})$. By the assumption $\pi
B_{\Omega}(z_{0})= c^{2}_{\beta}(z_{0})$, it follows that $\pi
B_{\Omega}(z_{0})=c^{2}_{B}(z_{0})$.

Using Lemma \ref{l:suita}, we obtain that $\Omega$ is conformally
equivalent to the unit disc less a (possible) closed set of inner
capacity zero.


\subsection{Proof of Theorem \ref{c:L_conj_proof}}
$\\$

Let $\{\Omega_{m}\}_{m=1,2,\cdots}$ be domains with smooth
boundaries, which satisfies $\Omega_{m}\subset\subset\Omega_{m+1}$
and $\cup_{m=1}^{\infty}\Omega_{m}=\Omega$.

Assume that $t\in \Omega_{1}$.

Denote $B_{\Omega_{m}}$ by $B_{m}$, $L_{\Omega_{m}}$ by $L_{m}$, and
$G_{\Omega_{m}}$ by $G_{m}$. Denote $B_{\Omega}$ by $B$,
$L_{\Omega}$ by $L$, and $G_{\Omega}$ by $G$. Denote $\exp
\lim_{z\to t}(G_{m}(z,t)-\log |z-t|)$ by $c_{\beta,m}(t)$, where $z$
is the local coordinate near $t$.

It is known that
$B_{m}=\frac{2}{\pi}\frac{\partial^{2}G_{m}(z,t)}{\partial z \partial \bar{t}}$,
and
$L_{m}=\frac{2}{\pi}\frac{\partial^{2}G_{m}(z,t)}{\partial z \partial t}=$ by \cite{schiffer} (see also \cite{suita76}).

Note that $B_{m}(z,\bar{t})dz=-L_{m}(z,t)dz$,
for $z\in\Omega$ and $t\in\partial\Omega_{m}$ (see \cite{suita76}).

If $L_{m}(z,t)$ has no zeros for a $t$, then we obtain a subharmonic
function
$$H_{m,t}(z):=|\frac{B_{m}(z,\bar{t})}{-L_{m}(z,t)}\exp-2G_{m}(z,t)|$$
which is $1$ at $\partial\Omega_{m}$.

By maximum principle, it follows that $H_{m,t}(z)\leq 1$ for any
$z\in\Omega$. As $L_{m}(z,t)-\frac{1}{\pi(z-t)^{2}}$ is holomorphic
near $t$(see page 92, \cite{schiffer}), then we have
\begin{equation}
\label{}
\begin{split}
&\lim_{z\to t}|L_{m}(z,t)|\exp2G_{m}(z,t)
\\=&\lim_{z\to t}\frac{\exp(2G_{m}(z,t))}{\pi |z-t|^{2}}
\\=&\frac{1}{\pi}\exp2\lim_{z\to t} (G_{m}(z,t)-\log|z-t|)
\\=&\frac{c^{2}_{\beta,m}(t)}{\pi}
\end{split}
\end{equation}

Note that $\lim_{m\to+\infty}c_{\beta,m}(t)=c_{\beta}(t)$ and
$\lim_{m\to+\infty}B_{m}(t,\bar{t})=B(t,\bar{t})$, and by Corollary
\ref{c:suita_equ}, it follows that
\begin{equation}
\label{}
\begin{split}
B_{m}(t,\bar{t})
>\lim_{z\to t}|L_{m}(z,t)|\exp2G_{m}(z,t),
\end{split}
\end{equation}
for $m$ big enough, therefore $H_{m,t}(t)>1$.

It contradicts with $H_{m}(z)\leq 1$ for any $z\in\Omega$, when $m$
is big enough. Then Theorem \ref{c:L_conj_proof} follows.


\subsection{Proof of Theorem \ref{c:extend_suita_conj}}
$\\$

Let $p:\Delta\to \Omega$ be the universal covering of $\Omega$. We
can choose $V_{z_{0}}$ small enough, such that $p$ is biholomorphic
on any component $U_{j}$ $(j=1,2,\cdots)$ of $p^{-1}(V_{z_{0}})$.

Let $z_{0}\in\Omega$ with local coordinate $w=(p|_{U_{j}})_{*}(f_{z_0}|_{U_{j}})$
for a fixed $j$.
It is known that if $\chi_{-h}=\chi_{z_{0}}$,
then
$c_{\beta}^{2}(z_{0})=\pi\rho(z_{0}) B_{\Omega,\rho}(z_{0})$ holds (see \cite{yamada98}).
Then it suffices to prove that if
$$c_{\beta}^{2}(z_{0})=\pi\rho(z_{0}) B_{\Omega,\rho}(z_{0})$$
holds,
then $$\chi_{-h}=\chi_{z_{0}}.$$

By the assumption
$$c_{\beta}^{2}(z_{0})=\pi\rho(z_{0}) B_{\Omega,\rho}(z_{0}),$$
and by Remark \ref{r:minimal_extend}, it follows that there is a
unique holomorphic $(1,0)$ form $F$ on $\Omega$, which satisfies
$((p|_{U_j})_{*}(f_{-h}|_{U_j}))F|_{z_{0}}=dw$, and
$$\sqrt{-1}\int_{\Omega}F\wedge\bar{F}\leq\pi\int_{z_{0}}|dw|^{2}dV_{\Omega}[2 G_{\Omega}(z,z_{0})].$$

It follows from \ref{p:extended_unique} that
$((p|_{U_j})_{*}(f_{-h}|_{U_j}))F|_{V_{z_0}}=dw$. Then we have
$$f_{-h}(p^{*}F)|_{U_j}=(p|_{U_{j}})^{*}dw=df_{z_{0}}|_{U_j}.$$

It follows from Lemma \ref{l:identity} that
$f_{-h}p^{*}F=df_{z_{0}}$. As $p^{*}F$ is single-valued and
$df_{z_{0}}\in\Gamma^{\chi_{z_0}}$, it is clear that
$\chi_{-h}=\chi_{z_0}$.

\subsection{Proof of Theorem \ref{t:guan-zhou-limiting}}
$\\$

Note that $\frac{-r}{\delta}$ has uniform positive upper and lower bound on $D$,
then we can consider the function $-r$ instead of $\delta$ in the present theorem.

Let $\Psi:=-\log(-\frac{r}{\varepsilon_{0}|s|^{2}}+1)|_{D}<0$,
where $\varepsilon_{0}$ is a positive constant small enough.
As $r$ is strictly plurisubharmonic on $\bar{D}$,
we have $r-\varepsilon_{0}|s|^{2}$ is a plurisubharmonic function on $D$ for $\varepsilon_{0}$ small enough.

Note that $-\log(-t)$ is increasing convex when $t<0$, then
$-\log(-r+\varepsilon_{0}|s|^{2})$ is a plurisubharmonic function on
$D$. As $\log\varepsilon_{0}|s|^{2}$ is a plurisubharmonic function
on $D$, then $\Psi$ is a plurisubharmonic function on $D$.

Let $c_{0}(t)|_{0<t<1}:=t^{\alpha}$ and $c_{0}(t)|_{t\geq1}:=1$.
Then we have $\int_{0}^{+\infty}c_{0}(t)e^{-t}dt<\frac{1}{1+\alpha}+1$.
Let $h:=e^{-(\varphi-\alpha\log(-r+\varepsilon_{0}|s|^{2}))}$.
Then we have $\Theta_{he^{-\Psi}}\geq 0$.

Note that there are positive constants $C_{3}$ and $C_{4}$, which are independent of $\alpha$,
such that $\frac{c_{0}(-\Psi)e^{\alpha\log(-r+\varepsilon_{0}|s|^{2})}}
{r^{\alpha}}\leq \max\{C_{3}^{\alpha},C_{4}^{\alpha}\}$ on $D$.

By the similar method in the proof of Theorem
\ref{t:guan-zhou-unify}, it follows that when $h$ is $C^{2}$ smooth,
$\Psi$ is $C^{2}$ plurisubharmonic function, and
$\Theta_{he^{-\Psi}}\geq0$, Theorem \ref{t:guan-zhou-unify} also
holds.

For any point $z\in H$, there exists a local holomorphic defining
function $e$ of $H$, such that $2\log|s|-2\log|e|$ is continuous
near $z$. Then using Lemma \ref{l:lem9}, for any holomorphic
section $f$ on $H\cap D$, we have an extension $F$ of $f$ on $D$,
such that
$$\int_{D}|F|^{2}(-r)^{\alpha}e^{-\varphi}d\lambda\leq C_{(D,H)}\max\{C_{3}^{\alpha},C_{4}^{\alpha}\}\frac{2+\alpha}
{1+\alpha}e^{-\varepsilon_{0}}\int_{D\cap H}|f|^{2}(-r)^{1+\alpha}e^{-\varphi}d\lambda_{H},$$
where $C_{(D,H)}$ only depends on $D$ and $H$.

As $\frac{-r}{\delta}$ has uniform upper and lower bounds on $D$,
thus we have proved Theorem \ref{t:guan-zhou-limiting}.


\subsection{Proof of Theorem \ref{t:Lp_GZ}}
$\\$

As $M$ is a Stein manifold, then for any given $f$, there exists a
holomorphic section $F_{1}$ on $K_{M}\otimes L$, such that
$F_{1}|_{S}=f$.

Note that
$$\sqrt{-1}\Theta_{he^{-(2-p)\log|F_{1}|_{h}}}\geq \sqrt{-1}\frac{p}{2}\Theta_{h}+\frac{2-p}{2}\sqrt{-1}\partial\bar\partial \varphi.$$
Then the metric $he^{-(2-p)\log|F_{1}|_{h}}$ and $\Psi$ satisfy
conditions $1)$ and $2)$ in Theorem \ref{t:guan-zhou-semicontinu2}
on Stein manifold $M\setminus \{F_{1}=0\}$.

Since $M$ is a Stein manifold, we can find a sequence of Stein
subdomains $\{D_j\}_{j=1}^\infty$ satisfying $D_j\subset\subset
D_{j+1}$ for all $j$ and
$\overset{\infty}{\underset{j=1}{\cup}}D_j=M$, and all
$D_{j}\setminus S$ are complete K\"{a}hler (\cite{grauert}).

Let $A_{1}:=\int_{D_j}c_{A}(-\Psi)|F_{1}|^{p}_{h}dV_{M}<+\infty$.

By the upper semicontinuity of $\log|F_{1}|_{h}$ on $M$,
it follows
that there exists a new extension $F_{2}$ on $M$ of $f$ satisfying:

\begin{equation}
\label{equ:optimal_delta_616c}
\int_{D_j}c_{A}(-\Psi)|F_{2}|^{2}_{he^{-(2-p)\log|F_{1}|_{h}}}dV_{M}
\leq\frac{1}{\delta}c_{A}(-A)e^{A}+\int_{-A}^{\infty}c_{A}(t)e^{-t}dt,
\end{equation}

By H\"{o}lder's inequality, it follows that

\begin{equation}
\label{}
\begin{split}
&\int_{D_j}c_{A}(-\Psi)|F_{2}|^{p}_{h}dV_{M}=
\int_{D_j}c_{A}(-\Psi)\frac{|F_{2}|_{h}^{p}}{|F_{1}|^{p-\frac{p^2}{2}}_{h}}|F_{1}|^{p-\frac{p^{2}}{2}}_{h}dV_{M}
\\&\leq
(\int_{D_j}c_{A}(-\Psi)|F_{2}|^{2}_{he^{-(2-p)\log|F_{1}|_{h}}}dV_{M})^{\frac{p}{2}}
(\int_{D_j}c_{A}(-\Psi)|F_{1}|^{p}_{h}dV_{M})^{1-\frac{p}{2}},
\end{split}
\end{equation}
which is smaller than
$$\max\{(\frac{1}{\delta}c_{A}(-A)e^{A}+\int_{-A}^{\infty}c_{A}(t)e^{-t}dt)^{\frac{p}{2}}A_{1}^{1-\frac{p}{2}}
,\frac{1}{\delta}c_{A}(-A)e^{A}+\int_{-A}^{\infty}c_{A}(t)e^{-t}dt\}=:A_{2}.$$

If
$$A_{1}\leq\frac{1}{\delta}c_{A}(-A)e^{A}+\int_{-A}^{\infty}c_{A}(t)e^{-t}dt,$$
then we are done. We only need to consider the case that
$$A_{1}>\frac{1}{\delta}c_{A}(-A)e^{A}+\int_{-A}^{\infty}c_{A}(t)e^{-t}dt.$$
In this case, $A_{2}<A_{1}$.

We can repeat the same argument with $F_{1}$ replaced by $F_{2}$
etc, and get a decreasing sequence of numbers $A_{k}$, such that
$$A_{k+1}:=\max\{(\frac{1}{\delta}c_{A}(-A)e^{A}+\int_{-A}^{\infty}c_{A}(t)e^{-t}dt)^{\frac{p}{2}}A_{k}^{1-\frac{p}{2}},
\frac{1}{\delta}c_{A}(-A)e^{A}+\int_{-A}^{\infty}c_{A}(t)e^{-t}dt\}$$
for $k\geq1$.

It is clear that
$$A_{k+1}>\frac{1}{\delta}c_{A}(-A)e^{A}+\int_{-A}^{\infty}c_{A}(t)e^{-t}dt,$$
and $A_{k+1}<A_{k}$. Then $\lim_{k\to\infty}A_{k}$ exists.

By the definition of $A_{k}$, it follows that
$$\lim_{k\to\infty}A_{k}=\frac{1}{\delta}c_{A}(-A)e^{A}+\int_{-A}^{\infty}c_{A}(t)e^{-t}dt.$$
Then the present theorem for $D_{j}$ has been proved. Let $j$ tend
to $\infty$, thus we have proved the present theorem.

\subsection{Proof of Theorem \ref{t:interp_GZ}}
$\\$

Let $h:=e^{-\varphi_{r}}$,
where
$$\varphi_{r}:=\varphi*\frac{\mathbf{1}_{B(0,r)}}{Vol(B(0,r))}.$$
Let $$\Psi:=2(T-T*\frac{\mathbf{1}_{B(0,r)}}{Vol(B(0,r))}),$$ where
$T$ is a plurisubharmonic polar function of $W$ on $\mathbb{C}^{n}$,
such that $(\partial\bar\partial
T*\frac{\mathbf{1}_{B(0,r)}}{Vol(B(0,r))})(z)$ has a uniformly upper
bound on $\mathbb{C}^{n}$ which is independent of
$z\in\mathbb{C}^{n}$ and $r$.

As $D^{+}(W)<\frac{p}{2}$,
there exists $T$,
for $r$ large enough,
we have $D(W,T,z,r)<(1-\epsilon)\frac{p}{2}$,
($\epsilon>0$),
which implies that
$$\sqrt{-1}\partial\bar\partial((1+\delta)\Psi+\varphi_{r})\geq 0,$$
for positive $\delta$ small enough.

Note that $\Psi$ has uniformly upper bound on $\mathbb{C}^{n}$.
There exists positive constant $C$ and $C'$, such that
$C\omega<\sqrt{-1}\partial\bar\partial\varphi<C'\omega$, then we
have $\varphi_{r}-\varphi<C_{r}<+\infty$, where
$\omega=\sqrt{-1}\partial\bar\partial|z|^{2}$. Let $c_{A}=1$. Using
Theorem \ref{t:Lp_GZ}, we obtain the present theorem.

\subsection{Proof of Corollary \ref{c:manivel-demailly}}
$\\$

Let $c_{A}(t):=e^{t}t^{-2}$. It is easy to see that
$\int_{-A}^{\infty}c_{A}(t)e^{-t}dt<+\infty$ and $c_{A}(t)e^{-t}$ is
decreasing with respect to t, where $t\in(-A,+\infty)$ and $A=-2r$.

Let $\Psi=r\log(|w|^{2})<-2r$, where $S=\{s=0\}$. Let
$\delta=\frac{1}{r}$.

Note that $s(t)=\frac{(1+\frac{1}{r})t-\log t-1}{2-\frac{1}{t}}\geq \frac{t}{2}$ in Theorem
\ref{t:guan-zhou-semicontinu2},
$$\frac{\{\sqrt{-1}\Theta(E)w,w\}}{|w|^2}\geq -\sqrt{-1}\partial\bar\partial\log|w|^2.$$
Then condition 2) in Theorem \ref{t:guan-zhou-semicontinu2} holds.

Note that $$\frac{|f|^2}{|\wedge^r(dw)|^2}dV_{H}=\sqrt{-1}^{(n-r)^{2}}
\{\frac{f}{\wedge^r(dw)},\frac{f}{\wedge^r(dw)}\}_{h}e^{-\psi}$$
(see Remark 12.7 in \cite{demailly2010}),
and
$$2^{r}\sqrt{-1}^{(n-r)^{2}}\{\frac{f}{\wedge^r(dw)},\frac{f}{\wedge^r(dw)}\}_{h}e^{-\psi}=|f|^{2}_{h}dV_{M}[\Psi].$$

Then it follows from Remark \ref{r:rem9} and Theorem
\ref{t:guan-zhou-semicontinu2} that Corollary
\ref{c:manivel-demailly} holds.

Then we illustrate that the estimate is optimal.

Let $M$ be the disc $\Delta_{e^{-1}}\subset\mathbb{C}$.
Let $E$ be a trivial line bundle with
Hermitian metric $h_{E,a}=e^{-\max\{\log|z|^{2},\log|a|^{2}\}-2}$,
and $w=z$.
Let $L$ be a trivial line bundle with Hermitian metric
$h_{L,a}=e^{-2\max\{\log|z|^{2},\log|a|^{2}\}}$.
It is clear that $r=1$, $\delta=1$.
Let $\alpha=1$.
Then
$$|w|=|z|e^{\frac{1}{2}(-\max\{\log|z|^{2},\log|a|^{2}\}-2)}\leq e^{-1}=e^{-\alpha}$$
satisfying inequality $(b)$ in Theorem \ref{t:manivel-demailly},
and inequality $(a)$ in Theorem \ref{t:manivel-demailly} becomes
$$\sqrt{-1}\partial\bar\partial2\max\{\log|z|^{2},\log|a|^{2}\}-2\sqrt{-1}
\partial\bar\partial(\max\{\log|z|^{2},\log|a|^{2}\}+2)\geq0.$$

Note that
\begin{equation}
\begin{split}
&\sqrt{-1}\Theta(L)+r\sqrt{-1}\partial\bar\partial\log|w|^2
=\\&\sqrt{-1}\partial\bar\partial2\max\{\log|z|^{2},\log|a|^{2}\}-\sqrt{-1}
\partial\bar\partial(\max\{\log|z|^{2},\log|a|^{2}\}+2)\geq0,
\end{split}
\end{equation}
and
$\frac{1}{\delta}c_{A}(-A)e^{A}+\int_{2}^{+\infty}c_{A}(t)e^{-t}dt=\frac{1}{4}+
\int_{2}^{+\infty}t^{-2}dt=\frac{3}{4}$.

Let $a$ go to zero, by arguments in the proof of Remark
\ref{r:guan-zhou-unify-exa2}, it follows that the estimate in
Corollary \ref{c:manivel-demailly} is optimal.

\subsection{Proof of Corollary \ref{c:mcneal-varolin}}
$\\$

It is not hard to see that $\varphi+\psi$ and $\log\frac{|w|^{2}}{e}-g^{-1}(e^{-\psi}g(1-\log|w|^{2}))$ are
plurisubharmonic functions.

It suffices to prove the case that $M$ is a Stein manifold and $L$
is a trivial line bundle with singular metric $e^{-\varphi}$
globally.

Let $\varphi_{n}+\psi_{n}$ and $\tilde{\psi}_{n}$ be smooth
plurisubharmonic functions, which are decreasingly convergent to
$\varphi+\psi$ and
$\log\frac{|w|^{2}}{e}-g^{-1}(e^{-\psi}g(1-\log|w|^{2}))$
respectively, when $n\to+\infty$.

Let $g(t):=\frac{1}{c_{-1}(t)e^{-t}}$,
$\Psi:=\log\frac{e}{|w|^{2}}+\tilde{\psi}_{n_{2}}$,
and $h=e^{-\varphi_{n_{1}}-\psi_{n_{1}}+\tilde{\psi}_{n_{2}}}$.

Since $M$ is a Stein manifold, we can find a sequence of Stein
subdomains $\{D_m\}_{m=1}^\infty$ satisfying $D_m\subset\subset
D_{m+1}$ for all $m$ and
$\overset{\infty}{\underset{m=1}{\cup}}D_m=M$.

Note that
$\log\frac{|w|^{2}}{e}-g^{-1}(e^{-\psi}g(1-\log|w|^{2}))<0$. Given
$n_{2}$, for $m$ large enough, we have
$\Psi|_{D_{m}}=-\log\frac{e}{|w|^{2}}+\tilde{\psi}_{n_{2}}|_{D_{m}}<1$.

$\sqrt{-1}\partial\bar\partial\Psi\geq0$ and
$\sqrt{-1}\Theta_{he^{-\Psi}}\geq0$ on $M\setminus S$ imply
conditions $1)$ and $2)$ in Theorem \ref{t:guan-zhou-unify}.

Using Theorem \ref{t:guan-zhou-unify} and Lemma \ref{l:lem9}, we
obtain a holomorphic $(n,0)$ form $F_{m,n_{1},n_{2}}$ on $D_{m}$,
which satisfies $F_{m,n_{1},n_{2}}|_{S}=f$ and
\begin{equation}
\label{}
\begin{split}
&\int_{D_{m}}c_{-1}(\log\frac{e}{|w|^{2}}-\tilde{\psi}_{n_{2}})|F_{m,n_{2},n_{1}}|^{2}_{h}dV_{M}
\\&\leq\mathbf{C}2\pi\int_{-A}^{\infty}c_{-1}(t)e^{-t}dt\int_{S}|f|^{2}e^{-\varphi_{n_1}-\psi_{n_1}}dV_{S}
\\&\leq\mathbf{C}2\pi\int_{-A}^{\infty}c_{-1}(t)e^{-t}dt\int_{S}|f|^{2}e^{-\varphi-\psi}dV_{S},
\end{split}
\end{equation}

therefore
$$\int_{D_{m}}\frac{ee^{-\varphi_{n_{1}}-\psi_{n_{1}}}}
{|w|^{2}g(\log\frac{e}{|w|^{2}}-\tilde{\psi}_{n_{2}})}|F_{m,n_{2},n_{1}}|^{2}dV_{M}
\leq\mathbf{C}2\pi C(g)\int_{S}|f|^{2}e^{-\varphi-\psi}dV_{S}.$$

As
$\frac{ee^{-\varphi_{n_{1}}-\psi_{n_{1}}}}{|w|^{2}g(\log\frac{e}{|w|^{2}}-\tilde{\psi}_{n_{2}})}$
has a uniform lower bound for any compact subset of $D_{m}\setminus
S$, which is independent of $n_{1}$, it follows from Lemma
\ref{l:lim_unbounded} that there exists a subsequence of
$\{F_{m,n_{2},n_{1}}\}_{n_{1}}$, which is uniformly convergent to a
holomorphic $(n,0)$ form $F_{m,n_{2}}$ on any compact subset of
$D_m$.

By dominated convergence theorem, it follows that
$$\int_{D_{m}}\frac{ee^{-\varphi_{n_{1}}-\psi_{n_{1}}}}
{|w|^{2}g(\log\frac{e}{|w|^{2}}-\tilde{\psi}_{n_{2}})}|F_{m,n_{2}}|^{2}dV_{M}
\leq\mathbf{C}2\pi C(g)\int_{S}|f|^{2}e^{-\varphi-\psi}dV_{S}.$$

By Levi's theorem, it follows that
$$\int_{D_{m}}\frac{ee^{-\varphi-\psi}}{|w|^{2}g(\log\frac{e}{|w|^{2}}-\tilde{\psi}_{n_{2}})}|F_{m,n_{2}}|^{2}dV_{M}
\leq\mathbf{C}2\pi C(g)\int_{S}|f|^{2}e^{-\varphi-\psi}dV_{S}.$$

As
$\frac{ee^{-\varphi-\psi}}{|w|^{2}g(\log\frac{e}{|w|^{2}}-\tilde{\psi}_{n_{2}})}$
has a uniform lower bound for any compact subset of $D_{m}\setminus
S$, which is independent of $n_{2}$, it follows from Lemma
\ref{l:lim_unbounded} that there exists a subsequence of
$\{F_{m,n_{2}}\}_{n_{2}}$, which is uniformly convergent to a
holomorphic $(n,0)$ form $F_{m}$ on any compact subset of $D_m$.

Note that
$\frac{ee^{-\psi}}{|w|^{2}g(\log\frac{e}{|w|^{2}}-\tilde{\psi}_{n_{2}})}$
is decreasingly convergent to
$\frac{e}{|w|^{2}g(\log\frac{e}{|w|^{2}})}$, it follows from Levi's
theorem that
$$\int_{D_{m}}\frac{ee^{-\max\{\varphi,K\}}}{|w|^{2}g(\log\frac{e}{|w|^{2}})}|F_{m,n_{2}}|^{2}dV_{M}
\leq\mathbf{C}2\pi C(g)\int_{S}|f|^{2}e^{-\varphi-\psi}dV_{S},$$
where $K$ is a real number.

From dominated convergence theorem on $M\setminus S$, it follows
that
$$\int_{D_{m}}\frac{ee^{-\max\{\varphi,K\}}}{|w|^{2}g(\log\frac{e}{|w|^{2}})}|F_{m}|^{2}dV_{M}
\leq\mathbf{C}2\pi C(g)\int_{S}|f|^{2}e^{-\varphi-\psi}dV_{S}.$$

Using Lemma \ref{l:lim_unbounded}, we have a subsequence of
$\{F_{m}\}_{m}$, which is uniformly convergent to a holomorphic
$(n,0)$ form $F$ on any compact subset of $M$.

Using dominated convergence theorem on $M\setminus S$, we have
$$\int_{D_{m}}\frac{ee^{-\max\{\varphi,K\}}}{|w|^{2}g(\log\frac{e}{|w|^{2}})}|F|^{2}dV_{M}
\leq\mathbf{C}2\pi C(g)\int_{S}|f|^{2}e^{-\varphi-\psi}dV_{S}.$$

When $K$ goes to $-\infty$, using Levi's theorem, we have
$$\int_{M}\frac{ee^{-\varphi}}{|w|^{2}g(\log\frac{e}{|w|^{2}})}|F|^{2}dV_{M}
\leq\mathbf{C}2\pi C(g)\int_{S}|f|^{2}e^{-\varphi-\psi}dV_{S}.$$
Thus the present Corollary follows.


\subsection{Proof of Corollary \ref{c:D-P10}}
$\\$

Let $\Psi:=\log(|s|^{2}e^{-\varphi_{S}})$, and
$h:=e^{-\varphi_{F}-\varphi_{S}}$. Then it is clear that
$\Psi\leq-\alpha$ and $A=-\alpha$.

Let $c_{-\alpha}(t):=e^{(1-b)t}$, and $\delta=\frac{1}{\alpha}$.
Then we have $c_{-\alpha}(\alpha)e^{-\alpha}=e^{-b\alpha}$, and
$\int_{\alpha}^{+\infty}c_{-\alpha}(t)e^{-t}dt=\frac{1}{b}e^{-b\alpha}$.

When $\varphi_{S}$ and $\varphi_{F}$ are both smooth, using Theorem
$\ref{t:guan-zhou-semicontinu2}$ and Remark \ref{l:lem9}, we obtain
$C_{b}=2\pi (\alpha e^{-b\alpha}+\frac{1}{b}e^{-b\alpha})(\max_{M}
|s|^{2}e^{-\bar{\varphi}_{S}})^{1-b}$.

Now we discuss the general case ($\varphi_{S}$ and $\varphi_{F}$ may not be smooth).

As $M$ is Stein, we can choose relatively compact strongly
pseudoconvex domains $\{\Omega_{n}\}_{n=1,2,\cdots}$ of $M$
exhausting $M$.

Note that
$\varphi_{S}=\alpha\varphi_{S}-(\alpha\varphi_{S}-\varphi_{F})$. By
Lemma \ref{l:FN1}, it follows that there exist smooth functions
$\{\varphi_{S,j}\}_{j=1,2,\cdots}$ and
$\{\varphi_{F,j}\}_{j=1,2,\cdots}$, such that

1), $\{\varphi_{S,j}\}_{j=1,2,\cdots}$ are plurisubharmonic functions;

2), $\{\alpha\varphi_{S,j}-\varphi_{F,j}\}_{j=1,2,\cdots}$ are plurisubharmonic functions;

3), $\{\varphi_{S,j}\}_{j=1,2,\cdots}$ and
$\{\alpha\varphi_{S,j}-\varphi_{F,j}\}_{j=1,2,\cdots}$ are
decreasingly convergent to $\varphi_{S}$ and
$\alpha\varphi_{S}-\varphi_{F}$ respectively.

4), Given $n$, there exists $j_{n}$ such that for any $j\geq j_{n}$,
$|w|^{2}e^{-\varphi_{S,j}}|_{\Omega_{n}}\leq e^{-\alpha}$.

Using the smooth case which we already discuss, we obtain
holomorphic $(n,0)$ forms $\{U_{n,j}\}_{n,j}$ satisfying the optimal
estimate \ref{equ:D-P.a} on $\Omega_{n}$ for $\varphi_{S,j}$ and
$\varphi_{F,j}$.

Note that
$b\varphi_{S,j}+\varphi_{F,j}=-b(\alpha\varphi_{F,j}-\varphi_{S,j})+(b\alpha+1)\varphi_{F,j})$.
While $\varphi_{F,j}$ is invariant, let
$\alpha\varphi_{F,j}-\varphi_{S,j}$ go to
$\alpha\varphi_{F}-\varphi_{S}$, from Lemma \ref{l:uniform_bound} it
follows that there exists a subsequence of $\{U_{n,j}\}_{j}$,
denoted by $\{U_{n,j}\}_{j}$, which is uniformly convergent on any
compact subset of $\Omega_{n}$.

Now first let $\alpha\varphi_{F,j}-\varphi_{S,j}$ go to
$\alpha\varphi_{F}-\varphi_{S}$, and then let $\alpha\varphi_{F,j}$
go to $\alpha\varphi_{F}$, using Levi's Theorem, we obtain that the
limit $U_{n}$ of $\{U_{n,j}\}_{j}$ satisfies the estimate
\ref{equ:D-P.a} on $\Omega_{n}$.

Using weak compactness of unit ball in the Hilbert space
$L^{2}_{e^{-b\varphi_{S}-\varphi_{F}-(1-b)\bar\varphi_{S}}}(\Omega_{n})\cap\{$holomorphic
$(n,0)$ form$\}$, Lemma \ref{l:uniform_bound} and diagonal method,
we have a subsequence of $\{U_{n}\}_{n}$, still denoted by
$\{U_{n}\}_{n}$, uniformly convergent to a holomorphic $(n,0)$ form
$U$ on any compact subset of $\Omega$ for $n$ large enough, such
that $U$ satisfies the estimate \ref{equ:D-P.a} on any $\Omega_{n}$.

Therefore $U$ satisfies the estimate \ref{equ:D-P.a} on $\Omega$.
Then Corollary \ref{c:D-P10} follows.\\

We conclude the present subsection by pointing out that $C_{b}$ is
optimal.

Let $M$ be the disc
$\Delta_{e^{-\frac{\alpha}{2}}}\subset\mathbb{C}$. By Remark
\ref{r:guan-zhou-unify-exa2}, and letting $e^{\bar{\varphi}_{S}}$ be
decreasingly convergent to $|s|^{2}$, we can obtain that the
estimate in corollary \ref{c:D-P10} is optimal.


\vspace{.1in} {\em Acknowledgements}. The authors would like to
thank Prof. Yum-Tong Siu, Prof. J-P. Demailly and Prof. Bo
Berndtsson for giving series of talks at CAS and explaining us their
works. The authors would also like to thank the referees for
suggestions and comments. An announcement of the present paper
appears in \cite{guan-zhou13a}.

\bibliographystyle{references}
\bibliography{xbib}

\end{document}